\documentclass{amsbook}

\usepackage{tikz}

\usetikzlibrary{decorations.markings}

\usepackage[normalem]{ulem}

\usepackage{verbatim}

\usepackage{amsfonts}

\usepackage[lite]{amsrefs}

\usepackage{amsthm}

\usepackage{amscd}
\usepackage{amssymb}

\usepackage{amsmath}

\usepackage{mathabx}

\usepackage{enumerate}

\usepackage[all]{xy}

\usepackage{graphicx}

\usepackage{hyperref}

\newtheorem{theorem}{Theorem}[chapter]
\newtheorem{lemma}[theorem]{Lemma}
\newtheorem{corollary}[theorem]{Corollary}
\newtheorem{proposition}[theorem]{Proposition}

\newtheorem{definition-theorem}[theorem]{Definition / Theorem}

\newtheorem*{conjecture*}{Conjecture}
\newtheorem*{theorem*}{Theorem}

\theoremstyle{definition}
\newtheorem{definition}[theorem]{Definition}
\newtheorem{example}[theorem]{Example}

\newtheorem{observation}[theorem]{Observation}
\newtheorem{question}[theorem]{Question}

\theoremstyle{remark}
\newtheorem{remark}[theorem]{Remark}

\newtheorem*{example*}{Example}  
\newtheorem*{remark*}{Remark}

\numberwithin{section}{chapter}
\numberwithin{equation}{chapter}

\newcommand{\N}{\mathbb{N}}

\newcommand{\R}{\mathbb{R}}

\newcommand{\C}{\mathbb{C}}

\newcommand{\T}{\mathbb{T}}

\newcommand{\Manoa}{M\=anoa}

\newcommand{\Hawaii}{Hawai\kern.05em`\kern.05em\relax i}

\newcommand{\inj}{\mathrm{inj}}

\newcommand{\red}{\mathrm{red}}

\usepackage{xcolor}

\newcommand*{\mcomment}[1]{\textcolor{magenta}{#1}}

\DeclareMathOperator{\Aut}{Aut}

\DeclareMathOperator{\Ind}{Ind}

\DeclareMathOperator{\supp}{\mathrm{supp}}

\DeclareMathOperator{\pt}{\mathrm{pt}}

\DeclareMathOperator{\Prim}{\mathrm{Prim}}
\DeclareMathOperator{\SL}{\mathrm{SL}}
\DeclareMathOperator{\PSL}{\mathrm{PSL}}

\newcommand{\op}{{\operatorname{op}}}

\newcommand{\Bd}{\mathcal{B}}

\newcommand*{\nb}{\nobreakdash}

\newcommand*{\dd}{\,d}

\newcommand*{\K}{\mathcal K}
\newcommand*{\M}{\mathcal{M}} 

\newcommand*{\cont}{C}
\newcommand*{\contz}{\cont_0}
\newcommand*{\contc}{\cont_c}
\newcommand*{\contub}{\cont_{ub}}

\newcommand*{\id}{\textup{id}}
\newcommand*{\Ad}{\textup{Ad}}

\newcommand*{\U}{\mathcal U}

\newcommand*{\braket}[2]{\langle#1\!\mid\!#2\rangle}

\newcommand*{\sbe}{\subseteq} 

\newcommand*{\cstar}{\texorpdfstring{$C^*$\nobreakdash-\hspace{0pt}}{*-}}

\newcommand*{\into}{\hookrightarrow}
\newcommand*{\onto}{\twoheadrightarrow}
\renewcommand*{\max}{\mathrm{max}}
\renewcommand*{\min}{\mathrm{min}}

\newcommand{\om}{\omega}

\newcommand{\dach}{{\!\!\widehat{\ \ \ }}}

\begin{document}

\title[Amenability and weak containment]{Amenability and weak containment for actions of  locally compact groups on $C^*$-algebras}
\author{Alcides Buss}
\address{Departamento de Matem\'atica\\
 Universidade Federal de Santa Catarina\\
 88.040-900 Florian\'opolis-SC\\
 Brazil}
\email{alcides.buss@ufsc.br}

\author{Siegfried Echterhoff}
\address{Mathematisches Institut\\
 Westf\"alische Wilhelms-Universit\"at M\"un\-ster\\
 Einsteinstr.\ 62\\
 48149 M\"unster\\
 Germany}
\email{echters@uni-muenster.de}

\author{Rufus Willett}
\address{Mathematics Department\\
 University of \Hawaii~at \Manoa\\
Keller 401A \\
2565 McCarthy Mall \\
 Honolulu\\
 HI 96822\\
USA}
\email{rufus@math.hawaii.edu}



\maketitle

\begin{abstract} In this work we introduce and study a new notion of amenability for actions of locally compact groups on $C^*$-algebras.  Our definition extends the definition of amenability for actions of discrete groups due to Claire Anantharaman-Delaroche.  We show that our definition has several characterizations and permanence properties analogous to those known in the discrete case.   For example, for actions on commutative $C^*$-algebras, we show that our notion of amenability is equivalent to measurewise amenability.
Combined with a recent result of Alex Bearden and Jason Crann, this also settles a long standing open problem about the equivalence of  topological amenability and measurewise amenability for a second countable $G$-space $X$.

We use our new notion of amenability to study when the maximal and reduced crossed products agree.  One of our main results generalizes a theorem of Matsumura: we show that for an action of an exact locally compact group $G$ on a locally compact space $X$ the full and reduced crossed products $C_0(X)\rtimes_\max G$ and $C_0(X)\rtimes_\red G$ coincide if and only if the action of $G$ on $X$ is amenable.  We also show that the analogue of this theorem does not hold for actions on noncommutative $C^*$-algebras. 

Finally, we study amenability as it relates to more detailed structure in the case of $C^*$-algebras that fibre over an appropriate $G$-space $X$, and the interaction of amenability with various regularity properties such as nuclearity, exactness, and the (L)LP, and the equivariant versions of injectivity and the WEP.
\end{abstract}

\tableofcontents

\chapter{Introduction}

Amenability is an important property of groups and their actions (and other objects) with many consequences in dynamics, harmonic analysis, geometric group theory, and elsewhere.  There is a good notion of amenability in the literature for an action of a locally compact group on a von Neumann algebra.  However, for actions on $C^*$-algebras, some natural definitions only work in the case where the acting group is discrete.  Our goals in this work are four-fold:
\begin{enumerate}
\item Introduce a notion of amenability for an action of a locally compact group on a $C^*$-algebra, and study its relationship to approximation properties for the action and to existing notions of amenability for actions on commutative $C^*$-algebras.
\item Study the connection of amenability to the weak containment problem of when the maximal and reduced crossed products coincide.
\item Study amenability in the case of actions on $X\rtimes G$-$C^*$-algebras for a regular action $G\curvearrowright X$, and for type I $G$-$C^*$-algebras.
\item Study the connection of amenability to various important regularity properties such as nuclearity, exactness, lifting properties, and equivariant versions of the WEP and injectivity.
\end{enumerate}

We now discuss each of these goals in turn.

\section{Amenable actions}

To explain our notion of amenable actions, we start by recalling the classical case of actions on von Neumann algebras.  This theory was initiated by Claire Anantharaman-Delaroche \cite{Anantharaman-Delaroche:1979aa}*{D\'{e}finition 3.4} over forty years ago and has been used to great effect by Anantharaman-Delaroche and many others in the intervening period.

To establish terminology, let $A$ be a $C^*$-algebra (or a von Neumann algebra) equipped with an action $\alpha:G\to \text{Aut}(A)$ of a locally compact group $G$.  Then $(A,\alpha)$ is a \emph{$G$-$C^*$-algebra} (respectively, \emph{$G$-von Neumann algebra}) if for all $a\in A$ the `orbit map' $G\to A$ defined by $g\mapsto \alpha_g(a)$ is norm (respectively, ultraweakly) continuous.  

Let $G$ be a locally compact group, and let $(M,\sigma)$ be a $G$-von Neumann algebra.  Equip the von Neumann algebra tensor product $L^\infty(G)\overline{\otimes} M$ with the tensor product of the canonical action $\tau$ on $L^\infty(G)$ induced by left translation, and the given action $\sigma$ on $M$.   The following definition is due to Anantharaman-Delaroche \cite{Anantharaman-Delaroche:1979aa}*{D\'{e}finition 3.4}.  Note that it reduces to one of the standard definitions of amenability of $G$ if $M=\C$.

\begin{definition}
A $G$-von Neumann algebra $(M,\sigma)$ is \emph{amenable} if there exists an equivariant conditional expectation $\Phi:L^\infty(G)\overline{\otimes} M\to M$, 
where we view $M$ as a von Neumann subalgebra of $L^\infty(G)\overline\otimes M$ via the canonical embedding $m\mapsto 1\otimes m$.
\end{definition}

If $(A,\alpha)$ is a $G$-$C^*$-algebra and $G$ is discrete, then its double dual $A^{**}$ is a $G$-von Neumann algebra with the canonically induced action $\alpha^{**}$, and Anantharaman-Delaroche \cite{Anantharaman-Delaroche:1987os}*{D\'{e}finition 4.1} defines the $G$-action on $A$ to be \emph{amenable} precisely when $(A^{**},\alpha^{**})$ is an amenable $G$-von Neumann algebra.  This definition works well in the discrete case, but not for general locally compact $G$: indeed, $A^{**}$ is typically not a $G$-von Neumann algebra when $G$ is not discrete.

One of the main ideas in this paper is to find an appropriate replacement for $A^{**}$ when $G$ is locally compact.  The key ingredient is given by the following theorem of Akio Ikunishi:

\begin{theorem}[{\cite{Ikunishi}*{Theorem 1.1}}]\label{cont dd}
Let $G$ be a locally compact group, and $(A,\alpha)$ a $G$-$C^*$-algebra.  Then there is a canonically associated $G$-von Neumann algebra $(A_\alpha'',\alpha'')$ with the following universal property: any equivariant $*$-homomorphism $A\to M$ from $A$ to a $G$-von Neumann algebra $M$ extends uniquely to an ultraweakly continuous equivariant $*$-homomorphism $A_\alpha''\to M$.  
\end{theorem}

It follows from the theorem that when $G$ is discrete, $(A_\alpha'',\alpha'')=(A^{**},\alpha^{**})$.  Thus $A_\alpha''$ is a natural replacement for $A^{**}$.  We show that it has similar functoriality properties to $A^{**}$, although it is often a proper quotient.

Replacing $A^{**}$ by the $G$-von Neumann algebra $A_\alpha''$ we define a $G$-$C^*$-algebra $(A,\alpha)$ to be 
{\em von Neumann amenable} if $(A_\alpha'',\alpha'')$ is an amenable $G$-von Neumann algebra. This clearly extends the notion of amenable action of a discrete group $G$.  

However, von Neumann amenability is not directly useful for studying properties of the associated crossed products.  Very early on in the theory, Anantharaman-Delaroche \cite{Anantharaman-Delaroche:1982aa}*{Corollaire 3.7} established the striking fact that amenability of a $G$\nb-von Neumann algebra $(M,\sigma)$ is equivalent to amenability of the induced action of $G$ on the centre $Z(M)$.  For discrete $G$, Anantharaman-Delaroche later used this in \cite{Anantharaman-Delaroche:1987os}*{Th\'{e}or\`{e}me 3.3} to characterize amenability in terms of an approximation property using functions of positive type on $G$ with values in the centre $Z(M)$ of $M$, that is, functions  $\theta:G\to Z(M)$ such that 
for all finite $F\subseteq G$ the `$F\times F$-matrix' 
$$
\big(\sigma_g\theta(g^{-1}h)\big)_{g,h\in F}\in M_F(Z(M))
$$
is positive. Specializing to $M=A^{**}$ gives characterizations of amenability for $G$-$C^*$-algebras. The following  definition is inspired by this characterization.

\begin{definition}\label{def-am}
Let $(A,\alpha)$ be a $G$-$C^*$-algebra.  We say that  $(A,\alpha)$ is  {\em amenable} if there exists a net $(\theta_i:G\to Z(A_\alpha''))$ of norm-continuous, compactly supported functions of positive type such that $\|\theta_i(e)\|\leq 1$ for all $i\in I$ and $\theta_i(g)\to 1_{A_\alpha''}$ ultraweakly and uniformly on compact subsets of $G$.
\end{definition}

Our notion of amenability is a complete analogue of the classical definition of Anantharaman-Delaroche \cite{Anantharaman-Delaroche:1987os}*{D\'{e}finition 4.1} for  discrete groups $G$. For actions of discrete groups it has been shown by Anantharaman-Delaroche in  \cite{Anantharaman-Delaroche:1987os}*{Th\'{e}or\`{e}me 3.3} that amenability and von Neumann amenability coincide.
In  this work we prove that amenability  and von Neumann amenability coincide for actions of {\em exact} locally compact groups $G$. Shortly after we posted a first draft of this work on the arXiv, Alex Bearden and Jason Crann showed in  \cite{Bearden-Crann}*{Theorem 3.6} with different methods  that both notions coincide also for non-exact groups.
Thus we can now state

\begin{theorem}[Bearden-Crann]\label{thm-Bearden-Crann}
A $G$-$C^*$-algebra $(A,\alpha)$ is amenable if and only if it is von Neumann amenable.
\end{theorem}

A drawback of these formulations of amenability is that the
enveloping $G$\nb-von Neumann algebra is usually a huge object which is not easy to understand.  
Motivated by this, we give a different characterization, which we call the
{\em weak quasi-central approximation property} (wQAP).  To state it, let us denote by $S(A)^c$ the set of states $\phi$ on a $G$-$C^*$-algebra $(A,\alpha)$ such that,  if $\alpha^*$ is the induced action of $G$ on $A^*$, then the map 
$$
G\to S(A);\quad g\mapsto \alpha_g^*(\phi)
$$
is norm-continuous.

\begin{definition}\label{def-wQAP intro}
Let $(A,\alpha)$ be a $G$-$C^*$-algebra.  We say that $(A,\alpha)$ has the {\em weak quasi-central approximation property} (wQAP) if there exists a net $(\xi_i)_{i\in I}$ of functions in $C_c(G,A)\subseteq L^2(G,A)$ such that:
\begin{enumerate}[(i)]
\item $\|\braket{\xi_i}{\xi_i}_A\|\leq 1$ for all $i\in I$;
\item for all $\phi\in S(A)^c$ we have $\phi( \braket{\xi_i}{\lambda_g^\alpha\xi_i}_A) \to 1$ uniformly for $g$ in compact subsets of $G$;
\item for all $\phi\in S(A)^c$ and all $a\in A$ we have $\phi(\braket{\xi_i a-a\xi_i}{\xi_i a-a\xi_i}_A)\to 0$.
\end{enumerate}
\end{definition}

The (wQAP) is a variant of the {\em quasi-central approximation property} (QAP) as introduced by the authors in \cite{Buss:2019}*{Section 3} to explain some work of Yuhei Suzuki \cite{Suzuki:2018qo}.  The (QAP) is analogous to the (wQAP), but one replaces the weak convergence conditions (ii) and (iii) by the analogous norm\footnote{One should use strict convergence in the multiplier algebra in the non-unital case.} convergence conditions. 
We show in Theorem \ref{thm-wQAP} that amenability is equivalent to the  (wQAP) and to the {\em weak approximation property} (wAP),
a weak version of the {\em approximation property} (AP) as introduced by Ruy Exel and Chi-Keung Ng  in  \cite{ExelNg:ApproximationProperty}*{Definition 3.6} for Fell bundles over $G$. In an earlier draft of this work we 
presented an argument that for discrete $G$ amenability is also equivalent to the (QAP)  and the (AP), and we asked whether such result could also hold for actions of general locally compact groups. While it turned out later that our argument contained a mistake\footnote{We are grateful to one of the referees for pointing this out to us!},
 our question has been answered in the positive by Narutaka Ozawa and Suzuki in 
\cite{Ozawa-Suzuki}*{Theorem 2.13 and Theorem 3.2}. Combining the results of Ozawa and Suzuki with ours we can now state

\begin{theorem}[Theorem \ref{thm-wQAP} and Ozawa-Suzuki]\label{thm-OzSuz}
For a $G$-$C^*$-algebra $(A, \alpha)$ the following are equivalent
\begin{enumerate}[(i)]
\item $(A,\alpha)$ is amenable.
\item $(A,\alpha)$ has the (wQAP).
\item $(A,\alpha)$ has the (wAP).
\item $(A,\alpha)$ has the (QAP).
\item $(A,\alpha)$ has the (AP).
\end{enumerate}
\end{theorem}

We note a striking aspect of the (wQAP) and/or the (QAP).   As discussed above, Anantharaman-Delaroche \cite{Anantharaman-Delaroche:1982aa}*{Corollaire 3.7} showed that amenability of a $G$-von Neumann algebra is equivalent to amenability of the induced action on its centre.  It had been suspected --  see for example \cite{Brown:2008qy}*{Definition 4.31} -- that amenability of a $G$-$C^*$-algebra $(A,\alpha)$ would also be equivalent to amenability of the action on the centre\footnote{In the non-unital case, one should use the centre $Z\M(A)$ of the multiplier algebra.} $Z(A)$.  However, in Definition \ref{def-wQAP intro}, one has only the ``quasi-centrality'' condition in (iii) above.  Thanks to work of Suzuki \cite{Suzuki:2018qo} (see also the discussion of Suzuki's work in \cite{Buss:2019}*{Section 3}), we know that this distinction between centrality and quasi-centrality in the $C^*$-algebra case is quite fundamental.  Indeed, for simple $C^*$-algebras $A$, amenability of the action on $Z\M(A)\cong\C$ always implies amenability of $G$, while it follows from Suzuki's constructions (see also the more recent work of  Ozawa and Suzuki in  \cite{Ozawa-Suzuki}) that there are many important examples of amenable actions on simple $C^*$-algebras by groups which are not amenable. It forms a theme of the current work that amenability for actions on $C^*$-algebras is a genuinely `noncommutative' property.

Using our characterizations of amenability, we can establish the following permanence properties.  We think of these as evidence that amenability is a natural condition.

\begin{theorem}
Amenability is inherited under taking equivariant: quotients; hereditary subalgebras (in particular ideals); extensions; inductive limits; Morita equivalent $G$-$C^*$-algebras.
\end{theorem}

 At this point, it seems we can be confident that the definition of amenability in this paper (equivalently, the (QAP)) is the `correct' notion of amenability for actions of locally compact groups on $C^*$\nb-algebras.

We use the equivalence of the (AP) and amenability  to show that nuclearity of the cross-sectional $C^*$-algebra of a Fell bundle $\Bd$ over a discrete group $G$ implies the (AP): this resolves a conjecture of Pere Ara, Exel, and Takeshi Katsura \cite{Ara-Exel-Katsura:Dynamical_systems}*{Remark~6.5}.

Let us complete this discussion of amenability with a result on measurewise amenability.  Having said that amenability is fundamentally a noncommutative property, our introduction of the enveloping $G$-von Neumann algebra also allows us to resolve problems concerning actions on commutative $C^*$-algebras.  Using the equivalence of von Neumann amenability and amenability\footnote{In an earlier version of this paper, we were able to establish a variant of Theorem \ref{mw amen intro} under the additional assumption that $G$ is exact; the work of Bearden and Crann allowed us to establish the current more general version.} as shown by Bearden and Crann \cite{Bearden-Crann}*{Theorem 3.6}, we get the following result.

\begin{theorem}\label{mw amen intro}
Let $(A,\alpha)$ be a $G$-$C^*$-algebra with $A=C_0(X)$ commutative, and both $X$ and $G$ second countable.  Then the following are equivalent:
\begin{enumerate}
\item the $G$-$C^*$-algebra $A$ is amenable; 
\item for every quasi-invariant Radon measure $\mu$ on $X$, the $G$-von Neumann algebra $L^\infty(X,\mu)$ is amenable.
\end{enumerate}
\end{theorem}

Condition (2) was introduced by Jean Renault in \cite{Renault-LNM}*{Definition II.3.6}, who called it  \emph{measurewise amenability} (see also \cite{ADR}*{Definition 3.3.1}).  In  \cite{Adams:1994wg}*{Theorem A} (see also \cite{ADR}*{Theorem 4.2.7}), Scot Adams, George Elliott and Thierry Giordano show that measurewise amenability is equivalent to requiring that for every quasi-invariant Radon measure $\mu$ on $X$, the measure space $(X,\mu)$ is amenable in the classical sense of Robert Zimmer \cite{Zimmer}*{Definition 1.4}.  Thus our notion of amenability interacts well with those of Zimmer and Renault.

Bearden and Crann \cite{Bearden-Crann}*{Corollary 4.14} have recently shown that amenability of an action $\alpha:G\to\Aut(C_0(X))$ is always equivalent to topological amenability of the underlying action $G\curvearrowright X$. 
Combining this with Theorem \ref{mw amen intro} gives a positive answer to the long standing open question whether topological amenability and measurewise amenability for an action $G\curvearrowright X$ coincide (for second countable $G$ and $X$).

\section{The weak containment problem}  A $G$-$C^*$-algebra $(A,\alpha)$ has the \emph{weak containment property} (WCP) if the canonical quotient map $A\rtimes_{\max} G \to A\rtimes_\red G$ is an isomorphism.  

In \cite{Anantharaman-Delaroche:1987os}*{Proposition 4.8}, Anantharaman-Delaroche used approximations by positive type functions to show that for discrete $G$, amenability of a $G$-$C^*$-algebra $(A,\alpha)$ implies the (WCP).   In this paper, we use our definition to extend this to general locally compact groups.  The {\em weak containment problem} asks whether the converse holds true as well, i.e.\ if $(A,\alpha)$ has the (WCP), is it also amenable?

For the class of exact groups $G$ and commutative $G$-$C^*$-algebras $A$, we can give a complete, positive answer to this problem.

\begin{theorem}\label{com wc}
Let $G$ be a locally compact and exact group, and let $A=C_0(X)$ be a commutative $G$-$C^*$-algebra.  Then the following are equivalent:
\begin{enumerate}[(i)]
\item the $G$-$C^*$-algebra $A$ is amenable;
\item the canonical quotient map $A\rtimes_{\max} G \to A\rtimes_\red G$ is an isomorphism.
\end{enumerate}
\end{theorem}

The class of exact groups was introduced by Eberhard Kirchberg and Simon Wassermann in \cite{KW-exact}.  It is very large, containing for example all almost connected groups \cite{KW}*{Corollary 6.9}.  The exactness assumption comes into play in our work via an important characterization of the property due to Jacek Brodzki, Chris Cave, and Kang Li \cite{Brodzki-Cave-Li:Exactness}*{Theorem 5.8} (in the second countable case) and Ozawa and Suzuki \cite{Ozawa-Suzuki}*{Proposition 2.5} (in general). 

In the case of discrete $G$ and unital commutative $A$, Theorem \ref{com wc} is due to Masayoshi Matsumura \cite{Matsumura:2012aa}*{Theorem 1.1}. Our proof is different to Matsumura's, relying heavily on ideas from our earlier work \cite{Buss:2019}*{Sections 4 and 5}.

For noncommutative $C^*$-algebras, we establish the following analogue of Theorem \ref{com wc}.

\begin{theorem}\label{gen wc 2}
Let $G$ be a locally compact exact group, and let $(A,\alpha)$ be a $G$-$C^*$-algebra.  Then the following are equivalent:
\begin{enumerate}[(i)]
\item $(A,\alpha)$ is amenable;
\item for every $G$-$C^*$-algebra $B$, the canonical quotient map $(A\otimes_{\max}B)\rtimes_{\max} G \to (A\otimes_{\max} B)\rtimes_\red G$ is an isomorphism;
\item the canonical quotient  map $(A\otimes_{\max}A^{\op})\rtimes_{\max} G \to (A\otimes_{\max} A^{\op})\rtimes_\red G$ is an isomorphism.
\end{enumerate}
\end{theorem}

Theorem \ref{gen wc 2} extends Matsumura's \cite{Matsumura:2012aa}*{Theorem 1.1} which covers the case where $G$ is discrete and $A$ is a unital nuclear $C^*$-algebra.  Again, our techniques are closer to our earlier work \cite{Buss:2019}*{Proposition 5.9} in the discrete case than to Matsumura's ideas.

It might at first seem odd that in Theorem \ref{gen wc 2} amenability is equivalent to the weak containment property for $A\otimes_{\max} A^\op$, and not for $A$ itself.  This is in fact quite necessary and is another manifestation of the idea that amenability for actions on $C^*$-algebras is a noncommutative phenomenon.  Indeed, we have the following striking example.

\begin{theorem}\label{no amen}
There is an action of $G=PSL(2,\C)$ on the compact operators $\mathcal{K}$ that is non-amenable, but such that the canonical quotient map $\mathcal{K}\rtimes_{\max}G \to \mathcal{K}\rtimes_\red G$ is an isomorphism.
\end{theorem}

Using these examples, we are able to answer a question of Anantharaman-Delaroche \cite{Anantharaman-Delaroche:2002ij}*{Question 9.2 (b)} on whether the weak containment property passes to the restriction of an action to a closed subgroup.

\begin{theorem}\label{no res}
There is an action of $G=PSL(2,\C)$ on the compact operators $\mathcal{K}$ and a closed subgroup $\Gamma$ of $G$ such that the canonical quotient map $\mathcal{K}\rtimes_{\max}G \to \mathcal{K}\rtimes_\red G$ is an isomorphism, but the canonical quotient map for the restricted action $\mathcal{K}\rtimes_{\max}\Gamma \to \mathcal{K}\rtimes_\red \Gamma$ is not injective.
\end{theorem}

The bad behaviour of the (WCP) under taking restrictions is in stark contrast 
to the behaviour of amenability: we show that for actions of exact groups, amenability passes to the restriction to a closed subgroup.  While this paper was under review, the exactness assumption was shown to be unnecessary by Ozawa and Suzuki \cite{Ozawa-Suzuki}*{Corollary 3.4}.

We should remark that a key technical tool in our investigation of the weak containment problem is a notion we call commutant amenability: see Definition \ref{com amen}.  The interaction of amenability and commutant amenability with each other and with exactness seems to be at the heart of the weak containment problem: see for example \cite{Willett} and \cite{Kranz:2020ug} for some results relating exactness, amenability, and the weak containment problem for groupoids.

\section{Regular $X\rtimes G$-algebras and type I $C^*$-algebras}

In the case that the $C^*$\nb-algebra being acted on has good structure, we are able to establish several interesting permanence properties and applications.  These results mainly seem to be new even in the discrete group case.

Let $X$ be a locally compact $G$-space.  An $X\rtimes G$-algebra  is a $C^*$-algebra that fibres over $X$ in a way that is compatible with the given actions.  Such $C^*$-algebras are important in the theory of induced representations, and in connection to the Baum-Connes conjecture (among other places).  

In the case that the $G$-action on $X$ is sufficiently well behaved (the technical condition needed is \emph{regularity} - see Definition \ref{def-regular}) we can use our results on weak containment to deduce that amenability for a regular $X\rtimes G$-algebra is determined by the actions on the fibres.

\begin{theorem}
Suppose that $G$ is an exact group and that $X$ is a regular locally compact $G$-space. 
Further let  $(A,\alpha)$ be an $X\rtimes G$-algebra. Then the following are equivalent:
\begin{enumerate}
\item $\alpha:G\to \Aut(A)$ is amenable.
\item For every $x\in X$, the  action $\alpha^x:G_x\to \Aut(A_x)$ on the fibre $A_x$ is  amenable.
\end{enumerate}
\end{theorem}

As a corollary, we get yet another permanence result: an induced action of an exact group from a closed subgroup is amenable if and only if the original action was.  This partly generalizes a result of Anantharaman-Delaroche from the discrete case \cite{Anantharaman-Delaroche:1987os}*{Th\'eor\`eme 4.6}.

Specializing to actions on type I $C^*$-algebras, we can also show that suitable amenable actions on type I algebras are determined by amenability of the actions on the point stabilizers.  

\begin{theorem}
Let $\alpha:G\to \Aut(A)$ be an action of a second countable, locally compact, exact group
on a separable, type I $C^*$-algebra $A$ such that the induced action on $\widehat{A}$ is regular.
Then $\alpha$ is amenable if and only if all stabilizers $G_\pi$ for the action of $G$ on $\widehat{A}$
are amenable.
\end{theorem}

In the case of Hausdorff spectrum, we get the following very satisfactory characterization of amenability.  Unlike the results above, the theorem below does not proceed via our results on weak containment, and so does not require exactness.

\begin{theorem}
Let $\alpha:G\to \Aut(A)$ be an action of a second countable, locally compact group
on a separable, type I $C^*$-algebra $A$ such that $X=\widehat{A}$ is Hausdorff (for example, if $A$ has continuous trace). Then  $\alpha$ is amenable if and only if the  induced action 
on $C_0(X)$  is amenable.
\end{theorem}

Combining our work with a result of Bearden and Crann \cite{Bearden-Crann}*{Corollary 4.14}, one sees that in the situation of the theorem above, $\alpha$ is amenable if and only if the action on $X$ is topologically (or measurewise) amenable.

\section{Regularity properties}

For discrete groups, it is a well-known philosophy that if $(A,\alpha)$ is an amenable $G$-$C^*$-algebra, then regularity properties such as nuclearity should be inherited by the crossed product $A\rtimes_\red G$.  For our notion of amenability for actions of locally compact groups, we get the following results.  

\begin{theorem}\label{pass intro}
Let $(A,\alpha)$ be an amenable $G$-$C^*$-algebra. Then $A$ is nuclear (respectively~is exact, has the WEP, has the LLP) if and only if $A\rtimes_\max G$ is nuclear (respectively~is exact, has the WEP, has the LLP). 
\end{theorem}  

We also get a similar result on the LP, although this is more subtle, in particular requiring separability assumptions.

In the discrete case, the results on nuclearity and exactness are well known: see for example \cite{Brown:2008qy}*{Theorem 4.1.8}.  
The result on the WEP for discrete groups is proved in \cite{Bhattacharya:2013sj} with a different proof, and under the stronger assumption that the $G$-$C^*$-algebra $(A,\alpha)$ is unital and amenable in the sense of Brown and Ozawa \cite{Brown:2008qy}*{Definition 4.3.1}; this is what we call \emph{strong amenability} (see Definition~\ref{def-amenable (SA)}) and it is the same as topological amenability of the action on the spectrum of the centre $Z(A)$ in case $A$ is unital. 

We turn now to $G$-injective $C^*$-algebras.  A $G$-$C^*$-algebra $A$ is \emph{injective} if for any commutative diagram
$$
\xymatrix{ C \ar@{-->}[dr] &  \\ B \ar[u] \ar[r] & A }
$$
where the solid horizontal arrow is an equivariant ccp map and the vertical arrow is an injective equivariant $*$-homomorphism, the diagonal arrow can be filled in by an equivariant ccp map.  The following theorem generalizes work of Brodzki, Cave, and Li \cite{Brodzki-Cave-Li:Exactness} and of Mehrdad Kalantar and Matthew Kennedy \cite{Kalantar:2014sp}*{Theorem 1.1} characterizing exactness in terms of actions on injective $G$-$C^*$-algebras. We refer to Definition \ref{def-amenable (SA)}  for the notion of a strongly amenable action.

\begin{theorem}\label{inj intro}
Let $G$ be a locally compact group. Then the following are equivalent:
\begin{enumerate}
\item $G$ is exact.
\item Every  $G$-injective $G$-$C^*$algebra $(A,\alpha)$ is strongly amenable.
\item There exists a strongly amenable  $G$-injective $G$-$C^*$-algebra  $(A,\alpha)$.
\end{enumerate}
\end{theorem}

Finally in the circle of ideas about regularity, we introduce a weakening of $G$\nb-injectivity called the \emph{continuous $G$-WEP}.  A $G$-$C^*$-algebra $(A,\alpha)$ has the continuous $G$-WEP if for any equivariant inclusion $B\into A$ of a $G$-$C^*$-algebra $(B,\beta)$ into $(A,\alpha)$, there is an equivariant ccp map $A\to B_\beta''$ such that the composition $B\to A \to B_\beta''$ is the canonical inclusion of $B$ into its enveloping $G$-von Neumann algebra.  For actions on suitably nice $C^*$-algebras, this property is closely related to amenability.  Indeed, one has 

\begin{theorem}\label{thm-contGWEP}
Let $(A,\alpha)$  be a $G$-$C^*$-algebra and assume that $A$ has the WEP (e.g.\ $A$ is nuclear) and $G$ is exact.  Then  $(A,\alpha)$  is amenable if and only if it has the continuous $G$\nb-WEP.
\end{theorem}

We conclude with the following result on a variant of the weak containment property.

\begin{theorem}\label{thm-contGWEP1}
Let  $(A,\alpha)$ be a commutative $G$-$C^*$-algebra. Then  $(A,\alpha)$ has the continuous $G$-WEP if and only if $A\rtimes_\max G=A\rtimes_\inj G$.
\end{theorem}

We will not explain all the terminology here, but just state that this is one of the most general forms of the equivalence between a variant of the weak containment property and a variant of an amenability property; indeed, it reduces to Theorem~\ref{com wc} when $G$ is exact.  Outside the exact case, weak containment unfortunately remains quite mysterious, and the above is currently the most general result we know in that direction.

\section{Outline of the paper} 
In Chapter~\ref{Chap:Enveloping-vN-algebra} we study one of the main actors in this paper: Ikunishi's enveloping $G$-von Neumann algebra $(A_\alpha'',\alpha'')$ of a $G$-$C^*$-algebra $(A,\alpha)$.
We show that the construction $(A,\alpha)\mapsto (A_\alpha'',\alpha'')$ has good functoriality  properties, similar to those of  the usual bidual construction $A\mapsto A^{**}$. Our construction is different from the one given by Ikunishi in \cite{Ikunishi}, but 
it follows from the universal properties of $(A_\alpha'',\alpha'')$ that both constructions coincide.  We also 
establish a one-to-one correspondence between quasi-equivalence classes of covariant representations of $(A,G,\alpha)$ 
and  $G$\nb-invariant central projections in $A_\alpha''$.

In Chapter~\ref{sec:Amenable-Actions} we introduce several notions of amenability
and study some relations between approximation properties by positive type functions and our notions of 
amenable actions.  We also prove some basic permanence properties of amenability. We then establish the equivalence of measurewise amenability and amenability.

In Chapter~\ref{chap:QAP} we prove the equivalence between amenability and the weak quasi-central approximation property (wQAP) and we relate this to the more recent results of Ozawa and Suzuki in \cite{Ozawa-Suzuki} which show that amenability is equivalent to the stronger quasi-central approximation property (QAP), as introduced by the authors in \cite{Buss:2019}*{Section 3}, and to the approximation property (AP) as introduced in \cite{ExelNg:ApproximationProperty}*{Definition 3.6}. As an application we solve a long-standing open question of Ara, Exel, and Katsura on the nuclearity of cross-sectional $C^*$-algebras of Fell bundles over $G$.

In Chapter~\ref{sec:Weak-containment} we  study the relation between amenability of  actions and the weak containment property.   Section~\ref{sec-commutative} proves the central positive results: Theorems \ref{com wc} and \ref{gen wc 2} above.  A key tool here is the use of the Haagerup standard form of a $G$-von Neumann algebra.  In Section~\ref{sec:example} we construct our examples of non-amenable actions $\alpha:G\to \Aut(\K)$  with $\K=\K(H)$ for some Hilbert space $H$, 
such that $\K\rtimes_\max G\cong \K\rtimes_\red G$; in other words, this gives a negative solution to the weak containment problem.  We also give our examples of actions with the weak containment property that have restrictions without that property.

In Chapter~\ref{chap:examples} we give applications to $C^*$-algebras with good structure.
In Section~\ref{sec:regular} we prove our results on  $X\rtimes G$-algebras in which $X$ is a {\em regular} $G$-space 
and for regular actions on type I $G$-$C^*$-algebras.  In Section~\ref{sec:continuous-trace} we study actions on  type I $C^*$-algebras $A$ with Hausdorff spectrum $\widehat{A}$.

In Chapter~\ref{sec:G-WEP} we study the relationship of amenability to regularity properties.  In Section \ref{sec:pass} we show that various regularity properties pass to crossed products  as in Theorem \ref{pass intro} above.  In Section \ref{sec:inj} we prove Theorem \ref{inj intro}, and  in Section \ref{sec:gwep} we introduce our equivariant version of the weak expectation property of Lance, the continuous $G$-WEP.

Finally, in Chapter~\ref{sec-questions} we discuss some recent developments and 
summarize some natural questions which arise from  the results in this paper. In particular we will discuss  the impact of the recent works of Bearden and Crann \cites{Bearden-Crann, Bearden-Crann1} and of Ozawa and Suzuki \cite{Ozawa-Suzuki} on the subject of this paper.

\section{Conventions and notation}\label{Sec:Prel}

In this paper, $\Bd(H)$ refers to the algebra of bounded operators on a Hilbert space $H$ and $\mathcal L(\mathcal H)$ refers to the algebra of adjointable operators on a Hilbert $C^*$\nb-module $\mathcal H$.  We will follow standard usage, and say that a net $(a_i)$ in $\Bd(H)$ converges \emph{weakly} if it converges in the weak operator topology (\emph{not} in the weak topology that $\Bd(H)$ inherits from its dual space), and similarly in $\Bd(H,H')$ if $H'$ is another Hilbert space.  The ultraweak topology on a von Neumann algebra $M$ will refer to the weak-$*$ topology coming from its unique predual; if $M\subseteq \Bd(H)$ is a concrete von Neumann algebra, then the weak operator topology inherited from $\Bd(H)$ agrees with the ultraweak topology on bounded sets (but not necessarily in general).  As we will usually be interested in convergence of bounded nets, we will sometimes elide the difference between weak and ultraweak convergence when we are dealing with a concrete von Neumann algebra.

Throughout, $G$ denotes a locally compact group equipped with a fixed Haar measure.  We will typically write $A$, $(A,G)$, $(A,\alpha)$ or $(A,G,\alpha)$ for a $G$-$C^*$-algebra depending on which data we want to emphasize.  If $A$ is a $G$-$C^*$-algebra (or $G$-von Neumann algebra) associated algebras and spaces such as the multiplier algebra $\M(A)$, the centre $Z(A)$, the dual $A^*$, and the double dual $A^{**}$ will be equipped with the canonically induced actions.  We warn the reader that even if the action of $G$ on $A$ is strongly continuous\footnote{Recall this means that for each $a\in A$, the map $G\to A$ defined by $g\mapsto \alpha_g(a)$ is norm continuous.}, the induced actions of $G$ on $A^*$ and $\M(A)$ will typically not be strongly continuous, and the induced action on $A^{**}$ will typically not be ultraweakly continuous.

Throughout, $L^\infty(G)$ denotes the von Neumann algebra of (equivalence classes of) bounded, measurable functions on $G$ equipped with the left translation action $\tau$ defined by $(\tau_gf)(h):=f(g^{-1}h)$. This makes $L^\infty(G)$ a $G$-von Neumann algebra.  If $A$ is a $C^*$-algebra (or von Neumann algebra) equipped with a not-necessarily-continuous action $\alpha:G\to \Aut(A)$, we will write $A_c$ for the collection of all $a\in A$ such that the map $g\mapsto \alpha_g(a)$ is norm continuous; note that norm continuity is used here, even if $A$ was originally a von Neumann algebra.  Note that  $A_c$ is a $G$-$C^*$-algebra with the naturally induced structures.   

If $(M, \sigma)$ is a $G$-von Neumann algebra, then $M$ becomes an $L^1(G)$-module via
\begin{equation}\label{eq-L1module}
f*m:=\int_G f(g)\sigma_g(m)\,dg,\quad f\in L^1(G), ~m\in M,
\end{equation}
where the integral converges ultraweakly.  For $f\in L^1(G)$ and $g\in G$ we define $\tau_g(f)(h):=f(g^{-1}h)$; direct computations then show that
\begin{equation}\label{eq-Mc}
\|f*m\|\leq \|f\|_1\|m\|\quad\text{and}\quad \sigma_g(f*m)=\tau_g(f)*m
\end{equation}
for all  $g\in G, f\in L^1(G)$ and $m\in M$. Since the translation action of $G$ on $L^1(G)$ is strongly continuous, it follows that 
$L^1(G)*M\subseteq M_c$. On the other hand, it is easily checked that if $(f_i)_{i\in I}$ is a standard approximate unit of $L^1(G)$ consisting of positive continuous functions with compact supports, then $f_i*m\to m$ {\em in norm} for every $m\in M_c$ (see \cite{Pedersen:1979zr}*{Lemma~7.5.1}).
Thus it follows from an application of the Cohen-Hewitt factorization  theorem \cite{Hewitt}*{Theorem (2.5)}, that 
\begin{equation}\label{eq:Mc=L1*M}
M_c=L^1(G)*M_c=L^1(G)*M.
\end{equation}
In particular,  if we define  
\begin{equation}\label{contub def}
\contub(G):=L^\infty(G)_c=L^1(G)*L^\infty(G),
\end{equation}
it follows that $\contub(G)$ consists of all bounded continuous functions $f:G\to \C$  that are uniformly continuous for the left-invariant uniform structure on $G$ (see  \cite{Hamana:2011}*{Proposition 3.3}).  Throughout, we equip $C_{ub}(G)$ with the restriction of the action $\tau$ on $L^\infty(G)$, i.e.\ $(\tau_gf)(h):=f(g^{-1}h)$.

Throughout, a \emph{$G$-map} always means a $G$-equivariant map between sets equipped with $G$-actions.  This terminology might be combined with others in what we hope is an obvious way: for example, $G$-embedding, ccp $G$-map, normal $G$-map etc. 

If $\mu$ is a positive Radon measure on a locally compact space $X$  and $M$ is a von Neumann algebra, then
$L^\infty(X,M)$ denotes the von Neumann tensor product $L^\infty(X,\mu)\bar{\otimes} M$. 
We refer to \cite{Tak}*{Chapter V, Theorem 7.17} for the relationship between $L^\infty(X,M)$ and the bounded ultraweakly measurable functions from $X$ to $M$.

For a $G$-$C^*$-algebra $A$ we regard $C_c(G,A)$, the space of compactly supported, norm continuous $A$-valued functions on $G$, as 
a $*$-algebra  with   convolution and involution given by
$$f_1*f_2(g):=\int_G f_1(h)\alpha_h(f_2(h^{-1}g))\,dh\quad\text{and}\quad f^*(g):=\Delta(g^{-1})\alpha_g(f(g^{-1}))^*$$
for $f_1, f_2,f\in C_c(G,A)$ and $g\in G$, where the integral is with respect to Haar measure on $G$ and $\Delta:G\to (0,\infty)$ 
denotes the modular function on $G$. A \emph{covariant homomorphism} $(\pi, u):(A,G)\to \M(D)$ for a $C^*$-algebra $D$
 consists of a $*$-homomorphism $\pi:A\to \M(D)$ together with a strictly continuous homomorphism 
 $u:G\to \U\M(D)$, denoted $g\mapsto u_g$,  into the group of unitaries of $\M(D)$ such that for all $a\in A$ and $g\in G$ we have
 $$\pi(\alpha_g(a))=u_g\pi(a)u_g^*.$$
 We say that $(\pi, u)$ is {\em nondegenerate} if $\pi:A\to \M(D)$ is nondegenerate in the sense that $\pi(A)D=D$. 
 If $D$ is the algebra $\K(H)$ of compact operators on a Hilbert space $H$, then $\M(D)=\Bd(H)$ and $(\pi,u)$ is a covariant representation on the Hilbert space $H$ in the usual sense. 
 Every covariant homomorphism $(\pi,u)$ of $(A,G,\alpha)$  into $\M(D)$ integrates to a $*$-homomorphism
 $$\pi\rtimes u:C_c(G,A)\to \M(D);\quad \pi\rtimes u(f):=\int_G \pi(f(g)) u_g \,dg.$$
 The {\em maximal crossed product} $A\rtimes_\max G$ is defined as the completion of $C_c(G,A)$ by the $C^*$-norm
 $$\|f\|_\max:=\sup_{(\pi,u)}\|\pi\rtimes u(f)\|.$$
Every integrated form $\pi\rtimes u:C_c(G,A)\to \M(D)$ extends uniquely to a $*$-homo\-mor\-phism out of $A\rtimes_\max G$,
and this implements a one-to-one correspondence between nondegenerate covariant homomorphisms of $(A,G,\alpha)$ and 
nondegenerate $*$\nb-homomorphisms of $A\rtimes_\max G$ -- the reverse process of $(\pi, u)\mapsto \pi\rtimes u$
 is given by sending a nondegenerate $*$-homomorphism $\Phi:A\rtimes_\max G\to \M(D)$ to the covariant homomorphism
 $(\Phi\circ i_A, \Phi\circ i_G)$,
   where $(i_A, i_G):(A,G)\to \M(A\rtimes_\max G)$ is the canonical covariant homomorphism\footnote{Note that we have  abused notation slightly: we use the same symbol $\Phi$ for the canonical extension of the original map $\Phi$ to $\M(A\rtimes_{\max}G)$.  Such abuses will be used throughout the paper when they seem unlikely to cause confusion.}. If $A\rtimes_\max G\to \Bd(H_u)$ denotes the universal representation of $A\rtimes_{\max} G$, i.e.\ the direct sum of all  representations which appear as GNS-constructions from the states of $A\rtimes_{\max}G$, then extending this representation (uniquely and faithfully) to $\M(A\rtimes_\max G)$, we will typically identify $(i_A, i_G)$ with the underlying covariant representation of $(A,G,\alpha)$ on $H_u$.

 The {\em regular representation} of $(A,G,\alpha)$ 
 is the covariant representation 
 $$(i_A^r, i_G^r):=\big((\id_A\otimes M)\circ \tilde\alpha, 1\otimes \lambda_G\big)$$ of 
$(A,G)$ to  $\M(A\otimes \K(L^2(G)))$
in which $\lambda_G:G\to \U(L^2(G))$ denotes the left regular representation, $M: C_0(G)\to \Bd(L^2(G))$ denotes the 
representation by multiplication operators, and 
$\tilde{\alpha}: A\to C_b(G,A)\subseteq \M(A\otimes C_0(G))$ is defined by $\big(\tilde{\alpha}(a)\big)(g)=\alpha_{g^{-1}}(a)$.
The {\em reduced crossed  product} $A\rtimes_\red G$ is defined as the image of the integrated form (also called the regular representation)
$$\Lambda_{(A,\alpha)}:A\rtimes_\max G\onto A\rtimes_\red G\subseteq  \M(A\otimes \K(L^2(G)))$$
of $(i_A^r, i_G^r)$. If $\sigma: A\to \Bd(H)$ is any nondegenerate $*$-representation of $A$, 
the extension of $\sigma\otimes\id_{\K}: A\otimes \K(L^2(G))\to \Bd\big(H\otimes L^2(G)\big)$ to $\M(A\otimes \K(L^2(G)))$ restricts to a 
representation
$$\tilde\sigma\rtimes\lambda: A\rtimes_\red G\to \Bd\big(H\otimes L^2(G)\big)$$
 which is faithful if and only if 
$\sigma:A\to \Bd(H)$ is faithful. It can be described more concretely as the integrated form 
of the covariant representation 
$(\tilde\sigma, \lambda)$ on $L^2(G,H)\cong H\otimes L^2(G)$ given by
\begin{equation}\label{eq-regular}
(\widetilde{\sigma}(a)\xi)(g):=\sigma\big(\alpha_{g^{-1}}(a)\big)\xi(g) \quad\text{and}\quad (\lambda_g\xi)(h):=\xi(g^{-1}h),
\end{equation}
for $ \xi\in L^2(G,H)$. We call $\tilde\sigma\rtimes \lambda$ the {\em regular representation} induced from $\sigma$.

Note that  the construction of $(\tilde\sigma, \lambda)$ as in (\ref{eq-regular}) also makes sense if we start with a $G$-von Neumann algebra $M$ and a normal $*$-representation $\sigma:M\to \Bd(H)$, in which case $\widetilde{\sigma}$ is faithful whenever $\sigma$ is.

One of the main topics of this work is the question of when the regular representation is an isomorphism.  When this happens, we say that the $G$-$C^*$-algebra $A$ (or action $\alpha$) has the \emph{weak containment property} (WCP) and usually write $A\rtimes_\max G\cong A\rtimes_\red G$ or just $A\rtimes_\max G=A\rtimes_\red G$.

\section{Acknowledgements} This work was funded by: the Deutsche Forschungsgemeinschaft (DFG, German Research Foundation) Project-ID 427320536 SFB 1442 and under Germany's Excellence Strategy EXC 2044  390685587, Mathematics M\"{u}nster: Dynamics, Geometry, Structure; CNPq/CAPES/Humboldt - Brazil; the US NSF (DMS 1401126, DMS 1564281, and DMS 1901522).

Part of this paper was written while the first author was visiting the second author at the University of M\"unster. The first author is grateful for the warm hospitality provided by the second author, and also to CAPES-Humboldt for granting the visit.

We are grateful to Alex Bearden and Jason Crann for keeping us updated on their paper \cite{Bearden-Crann} and 
for some other fruitful interactions.  The results of Bearden and Crann have important consequences for several statements of this paper. In particular,  \cite{Bearden-Crann}*{Theorem 3.6} allowed us to extend several results we originally only proved in the case of {\em exact} locally compact groups to the general case (for example, see Theorem \ref{Bearden-Crann} which has consequences to Corollary \ref{cor-amenable}, Theorem \ref{thm-amenable-all}, and 
Theorem \ref{thm-wQAP}).   We should also remark that we only became aware of Ikunishi's paper \cite{Ikunishi} through the references in \cite{Bearden-Crann}: using \cite{Ikunishi}, we were able to drop a lengthy appendix on the predual of  $A_\alpha''$ which appeared in the first preprint of this paper.

We also thank Timo Siebenand and Tim de Laat for some useful comments on the representation theory of $\SL(2,\C)$, and Yuhei Suzuki for helpful discussions centered around the results of his paper \cite{Suzuki:2020}.  
We want to thank Narutaka Ozawa and  Suzuki for informing us about their latest progress on the  relation between amenability and the (QAP) \cite{Ozawa-Suzuki}.

Finally, we want to thank the anonymous referees for a careful reading of the paper, and for many helpful comments.  One referee pointed out several errors and gaps in our previous expositions, in particular in our first attempt to  prove Theorem \ref{mw amen intro} and in an  attempt to prove the equivalence of amenability and the (QAP) for actions of discrete groups. The latter equivalence has now been settled for actions of all  locally compact groups by Ozawa and Suzuki in \cite{Ozawa-Suzuki}.  The amount of time the referees spent on this project is far beyond anything one can reasonably expect!  The referees' numerous  comments have significantly improved the quality of this paper, for which we are extremely grateful.

\chapter{The $G$-equivariant enveloping von Neumann algebra}\label{Chap:Enveloping-vN-algebra}

In this chapter, we introduce the enveloping $G$-von Neumann algebra of a $G$\nb-$C^*$\nb-algebra, and establish its basic properties.  For us, the point of the enveloping $G$-von Neumann algebra is that it plays the same role with respect to covariant representations of a $G$-$C^*$-algebra as the usual enveloping von Neumann algebra plays with respect to representations of a $C^*$-algebra.

In Section \ref{Sec:Enveloping-vN-algebra} we give our definition of the enveloping $G$-von Neumann algebra, and establish its universal property and functoriality properties.

In Section \ref{sec:ccover} we give a brief discussion of central covers of representations and how they interact with the enveloping $G$-von Neumann algebra $A_\alpha''$.  This material is a natural extension of the non-equivariant theory, and will be useful in the proof that measurewise amenability and amenability are equivalent.

Finally, in Section \ref{sec:predual}, we show that our enveloping von Neumann algebra is canonically isomorphic to the universal $W^*$-dynamical system of Ikunishi \cite{Ikunishi}, who seems to be the first to have studied this object.  We also give a brief summary (closely related to the work of Ikunishi \cite{Ikunishi}) of the predual of the enveloping $G$-von Neumann algebra.

\section{The enveloping von Neumann algebra of a $C^*$-action}\label{Sec:Enveloping-vN-algebra}

Let $(A,\alpha)$ be a $G$-$C^*$-algebra and let
$(i_A, i_G): (A,G)\to \Bd(H_u)$
be the universal representation of $(A,G)$, that is, the covariant representation corresponding to 
the direct sum of all GNS-representations of $A\rtimes_{\max}G$ as explained in Section \ref{Sec:Prel} above.  

\begin{definition}\label{gvna}
With notation as above, the \emph{enveloping $G$-von Neumann algebra} of $(A,\alpha)$ is defined to be
$$
A_{\alpha}'':=i_A(A)''\subseteq \Bd(H_u).
$$
\end{definition}
Note that $A_{\alpha}''$ is a $G$-von Neumann algebra with $G$-action given by $\alpha'':=\Ad i_G$. 

\begin{remark}\label{env iso 0}
The universal property of $A^{**}$ gives a $G$-equivariant, normal, surjective $*$-homomorphism
$$i_A^{**}:A^{**}\to i_A(A)''=A_\alpha''.$$
It can happen that $i_A^{**}$ is an isomorphism (see Remark \ref{env iso} below), but this is not  true in general. 
Indeed, if $A=C_0(G)$ equipped with the (left) translation action $\tau$, then $C_0(G)\rtimes_\max G\cong \mathcal K(L^2(G))$ via the 
integrated form $M\rtimes\lambda$ of the covariant pair $(M,\lambda)$, where $M$ is the multiplication action of $C_0(G)$, and $\lambda$ is the regular representation.
Therefore $C_0(G)_\tau''\cong L^{\infty}(G)$.  For non-discrete locally compact groups, the induced map $i_{C_0(G)}^{**}:C_0(G)^{**}\to L^\infty(G)$ is always a 
proper quotient, as the left hand side contains the characteristic function $\chi_{\{g\}}$ of any singleton as a non-zero element, and this is non-zero on the right if and only if $G$ is discrete.
\end{remark}

The algebra  $A_\alpha''$ enjoys the following universal property for covariant representations:

\begin{proposition}\label{prop-universal}
Let $(A,\alpha)$ be a $G$-$C^*$-algebra, and let $(\pi,u):(A,G)\to \Bd(H_\pi)$ be a nondegenerate covariant representation. Let
$\alpha^{\pi}=\Ad u$ denote the action of $G$ on $\pi(A)''$ given by conjugation with $u$. 
Then there exists a 
unique normal $\alpha''$-$\alpha^\pi$-equivariant surjective $*$-homomorphism
$$\pi'':A_\alpha''\to \pi(A)''$$
which extends $\pi$.
\end{proposition}
\begin{proof}
Let $(\pi\rtimes u)^{**}:(A\rtimes_\max G)^{**}\to \Bd(H_\pi)$ be the normal extension of the integrated form $\pi\rtimes u$, which exists by the universal property of $(A\rtimes_\max G)^{**}$. Viewing $A_\alpha''$ as a von Neumann subalgebra of $(A\rtimes_\max G)^{**}$, the homomorphism $(\pi\rtimes u)^{**}$ restricts to the desired  $G$-equivariant normal extension $\pi'':A_\alpha''\to \pi(A)''\subseteq \Bd(H_\pi)$.
\end{proof}

\begin{corollary}\label{cor-vN}
Suppose that $(M,\sigma)$ is a $G$-von Neumann algebra and let 
 $\varphi:A\to M$ be a  $G$-equivariant $*$-homomorphism.
 Then there is a unique normal $G$\nb-equivariant 
extension  $\varphi'':A_\alpha''\to M$.  

Moreover, this extension is surjective if (and only if) $\varphi(A)$ is ultraweakly dense in $M$.
\end{corollary}
\begin{proof} 
Using the regular representation associated to a faithful normal representation of $M$ as in line (\ref{eq-regular}), we may assume that $M\subseteq \mathcal{B}(H)$, and that there is a unitary representation $u$ on $H$ that implements $\sigma$.
Then $H'=\varphi(A)H$ is a $u$-invariant closed  subspace.  We can apply Proposition \ref{prop-universal} 
to the nondegenerate covariant representation $(\varphi, u)$ of $(A,G,\alpha)$ on $H'$, giving an extension $\varphi'':A_\alpha''\to \Bd(H')$.  The image of $\varphi''$ is contained in the ultraweak closure of $\varphi(A)$ in $\Bd(H')$, which equals the ultraweak closure of $\varphi(A)$ in $\Bd(H)$, and is therefore contained in $M$.   The surjectivity statement is clear from the construction.
\end{proof}

It follows from Corollary \ref{cor-vN} that $A''_\alpha$ is the `biggest' $G$-von Neumann algebra containing $A$ as an ultraweakly dense $G$-invariant $C^*$-subalgebra.  

\begin{remark}\label{env iso}
The canonical map $i_A^{**}:A^{**}\to A_\alpha''$ of Remark \ref{env iso 0} is an isomorphism if and only if $A^{**}$ is a $G$-von Neumann algebra: indeed, if 
$A^{**}$ is a $G$-von Neumann algebra, then Corollary \ref{cor-vN} gives an inverse to $i_A^{**}$.  In particular, $i_A^{**}$ is an isomorphism whenever $G$ is discrete.  
\end{remark}

Notice that $A_\alpha''$ contains (a copy of) $\M(A)$ as a unital $G$-invariant $C^*$-subalgebra. To see this, let $i_A:A\to \Bd(H_u)$ be the canonical representation of $A$ in the universal representation of $A\rtimes_\max G$.  As this representation is faithful and nondegenerate, $\M(A)$ identifies canonically with the idealizer of $i_A(A)$ in $\Bd(H_u)$ which lies in the bicommutant $A_\alpha''$ of $i_A(A)\subseteq \Bd(H_u)$.  From this, we see that the construction $A\mapsto A_\alpha''$ has good functoriality properties:

\begin{proposition}\label{evna func}
Let $(A,\alpha)$ and $(B,\beta)$ be $G$-$C^*$-algebras, and let $\phi\colon A\to \M(B)$ be a (possibly degenerate) $G$-equivariant $*$-homomorphism. Then there is a unique normal $G$-equivariant 
extension  $\phi''\colon A_\alpha''\to B_\beta''$ of $\pi$. 

Moreover, this correspondence gives a well-defined functor $(A,\alpha)\mapsto (A_\alpha'',\alpha'')$ from the category of $G$-$C^*$-algebras and equivariant $*$-homomorphisms to the category of $G$-von Neumann algebras and equivariant normal $*$-homomorphisms.
\end{proposition}

\begin{proof}
To construct $\phi''$, identify $\M(B)$ with a $C^*$-subalgebra of $M:=B_\beta''$, and apply Corollary \ref{cor-vN}.  The functoriality statement follows as the maps involved are normal extensions from ultraweakly dense subalgebras.
\end{proof}

Above, we identified $A_\alpha''$ with the image of $A^{**}$ in $(A\rtimes_\max G)^{**}$ under the normal extension of the canonical map $i_A:A\to (A\rtimes_\max G)^{**}$.  Using the reduced crossed product makes no difference.

\begin{lemma}\label{lem:A_alpha''-red}
Let $(A,\alpha)$ be a $G$-$C^*$-algebra. Then $A_\alpha''$ is isomorphic to the image of $A^{**}$ in $(A\rtimes_\red G)^{**}$ under the normal extension $(\iota_A^r)^{**}$ of the canonical map $i_A^r:A\to (A\rtimes_\red G)^{**}$.
\end{lemma}
\begin{proof}
Applying the regular representation construction from line \eqref{eq-regular} to the representation $i_A:A\to \Bd(H_u)$ and its normal extension $i:A_\alpha''\to (A\rtimes_\max G)^{**}\subseteq \Bd(H_u)$ we get representations 
$$
\widetilde{i_A}\rtimes \lambda\colon A\rtimes_\red G \to \Bd(L^2(G,H_u)) \quad \text{and}\quad  \widetilde{i} :A_\alpha''\to \Bd(L^2(G,H_u)).
$$
Consider the diagram
$$
\xymatrix{ A_\alpha'' \ar[r]^-{i} \ar[drr]_{\widetilde{i}}& (A\rtimes_\max G)^{**} \ar[r] & (A\rtimes_\red G)^{**} \ar[d]^-{(\widetilde{i_A}\rtimes \lambda)^{**}} \\ & & \Bd(L^2(G,H_u))},
$$
where the map $(A\rtimes_\max G)^{**} \to (A\rtimes_\red G)^{**}$ is the canonical quotient.  This commutes: indeed, it clearly commutes on $A$, and all the maps are normal.  The lemma states that the composition of the two horizontal maps is injective.  However, the diagonal map $\widetilde{i}$ is injective as $i$ is, so we are done.
\end{proof}

\begin{remark}
Note that the image of the homomorphism $(\widetilde{i_A}\rtimes \lambda)^{**}$ in the proof above equals
$$[\widetilde{i_A}(A)(1\otimes\lambda)(G)]''= [\tilde{i}(A_\alpha'')(1\otimes\lambda)(G)]''\sbe \Bd(H_u\otimes L^2(G))$$
which is the von Neumann algebra crossed product $A_\alpha''\bar\rtimes G$.
\end{remark}

In the remainder of this section, we show that the functor $A\mapsto A_\alpha''$ has good behaviour on injections, short exact sequences, and Morita equivalences.  

\begin{corollary}\label{cor:injectivity-double-prime}
If $(A,\alpha)$ and $(B,\beta)$ are $G$-$C^*$-algebras and
$\varphi\colon A\to B$ is an injective $G$-equivariant $*$-homomorphism, then the unique normal extension $\varphi''\colon A_\alpha''\to B_\beta''$ is still injective.
\end{corollary}
\begin{proof}
Consider the commutative diagram
$$
\xymatrix { A_\alpha'' \ar[rr]^-{\varphi''} \ar[d] && B_\beta''  \ar[d]  \\ (A\rtimes_\red G)^{{**}} \ar[rr]^-{(\varphi\rtimes G)^{**} } && (B\rtimes_\red G)^{**}, }
$$
where the vertical maps are the inclusions of Lemma \ref{lem:A_alpha''-red}.  The reduced crossed product functor and the double dual functor take injective $*$-homomorphisms to injective $*$-homomorphisms, whence $(\varphi\rtimes G)^{**}$ is an injection, so $\varphi''$ is too.
\end{proof}

The following result shows that the functor $(A,\alpha)\mapsto (A_\alpha'',\alpha'')$  converts short exact sequences into direct sums, just as for the usual double dual.

\begin{lemma}\label{lem-decom-exact}
Suppose that $I\into A\onto B$ is a short exact sequence of $G$-$C^*$-algebras, with actions on $I$, $A$, and $B$ called $\iota$, $\alpha$, and $\beta$, respectively.  Then there is a functorial direct sum decomposition $(A_\alpha'',\alpha'')=(I_\iota'',\iota'')\oplus (B_\beta'',\beta'')$.
\end{lemma}
\begin{proof}
As the maximal crossed product is an exact functor, the sequence 
$$
\xymatrix{I\rtimes_\max G \ar@{^(->}[r] & A\rtimes_\max G \ar@{->>}[r] &B\rtimes_{\max}G }
$$
is exact.  As taking double duals converts short exact sequences to direct sums (see for example \cite{Blackadar:2006eq}*{III.5.2.11}), we get a canonical isomorphism
$$
(A\rtimes_\max G)^{**}\cong  (I\rtimes_\max G)^{**} \oplus (B\rtimes_{\max}G)^{**}.
$$
It follows directly that if $p\in I_\alpha''$ is the unit (which is also the unit of $(I\rtimes_{\max} G)^{**}$), then $pA_\alpha''=I_\iota''$ and $(1-p)A_\alpha''=B_\beta''$.
\end{proof}

We give a definition of Morita equivalence of $G$-von Neumann algebras, which is based on \cite{Abadie-Buss-Ferraro:Morita_Fell}*{Definition 4.1}.

\begin{definition}\label{gvn mor}
Two $G$-von Neumann algebras $(M,\sigma)$ and $(N,\tau)$ are \emph{Morita equivalent} if they are Morita equivalent via some $M$-$N$ bimodule $X$ in the sense of \cite{Blecher:2004}*{8.5.12}\footnote{See also \cite{Blecher:2004}*{8.5.1} for the background definitions needed to understand 8.5.12.} that also satisfies the following equivariance condition: $X$ is equipped with a weak-$*$ continuous\footnote{$X$ has a canonical weak-$*$ topology by \cite{Blecher:2004}*{Lemma 8.5.4}.} $G$-action $\kappa$ that is compatible with the $M$- and $N$-valued inner products: this means that $_M\braket{\kappa_g(x)}{\kappa_g(y)}=\sigma_g(_M\braket{x}{y})$ and $\braket{\kappa_g(x)}{\kappa_g(y)}_N=\tau_g(\braket{x}{y}_N)$ for all $x,y\in X$ and $g\in G$.
\end{definition}

The next lemma relates the above to Morita equivalence of $G$-$C^*$-algebras: see for example \cite{Combes}*{Section 3, Definition 1} for the latter.  

\begin{lemma}\label{lem:Morita}
If $(A,\alpha)$ and $(B,\beta)$ are two Morita equivalent $G$-$C^*$-algebras, then $(A_\alpha'',\alpha'')$ and $(B_\beta'',\beta'')$ are Morita equivalent as $G$-von Neumann algebras.
\end{lemma}
\begin{proof} 
If $(A,\alpha)$ is Morita equivalent to $(B,\beta)$, then there is a (linking) $G$-$C^*$-algebra $(L,\delta)$ containing $A$ and $B$ as opposite full corners by $G$-invariant orthogonal full projections $p,q\in \M(L)$: this follows for example from the argument of \cite{CELY}*{Remark 2.5.3 (4)}.  Then the enveloping $G$-von Neumann algebra $A_\alpha''$ of the corner $A=pLp$ identifies with the corner $pL_\delta''p$ in $L_\delta''$, and similarly  $B_\beta''$ identifies with $qL_\delta''q$. Moreover, as $p$ is full in $L$, it is also full in $L''_\delta$ in the sense that the ideal of $L_\delta''$ generated by $p$ is ultraweakly dense, and similarly for $q$.  The result follows from this: compare \cite{Abadie-Buss-Ferraro:Morita_Fell}*{Remark 4.3}.
\end{proof}

\section{Central covers of covariant representations}\label{sec:ccover}

In this section we discuss central covers of representations and how they interact with the enveloping $G$-von Neumann algebra $A_\alpha''$.  This material seems interesting in its own right, and will be useful in the proof that measurewise amenability and amenability are equivalent.

Recall from \cite{Pedersen:1979zr}*{Theorem 3.8.2} that for every nondegenerate $*$-representation $\pi:A\to \Bd(H)$
there exists a unique central projection $c_\pi\in Z(A^{**})$ such that $c_\pi A^{**}\cong \pi(A)''$ via the normal extension 
$\pi^{**}:A^{**}\to \Bd(H)$.   Recall also from \cite{Pedersen:1979zr}*{Definition 3.3.6} that two representations $\pi_1,\pi_2$ are {\em quasi-equivalent} if there exists an isomorphism $\Phi:\pi_1(A)''\to \pi_2(A)''$ such that $\Phi\big(\pi_1(a)\big)=\pi_2(a)$ for all $a\in A$.  From \cite{Pedersen:1979zr}*{Theorem 3.8.2} again, $\pi_1$ and $\pi_2$ are quasi-equivalent if and only if $c_{\pi_1}=c_{\pi_2}$. 

Now, if $(\pi, u)$ is a nondegenerate covariant representation of $(A,G,\alpha)$ on a Hilbert space $H$, it follows from the equation $\pi\circ \alpha_g=\Ad u_g\circ \pi$ and the fact that $\Ad u_g\circ \pi$ is (quasi-)equivalent to $\pi$ that $c_\pi$ is 
$G$-invariant. 
Let $d_\pi\in Z(A_\alpha'')$
denote the image of $c_\pi$ under the canonical quotient map $A^{**}\to A_\alpha''$. 
Since $\pi^{**}:A^{**}\to \pi(A)''$ factors through 
the $\alpha''$-$\Ad u$-equivariant map $\pi'':A_\alpha''\to \pi(A)''$ of Proposition \ref{prop-universal}, it follows that $\pi''$ induces a $G$-isomorphism  $d_\pi A_\alpha''\cong \pi(A)''$. Recall that $A_\alpha''$ is defined as the double commutant $i_A(A)''\subseteq \Bd(H_u)$, where 
$H_u$ denotes the Hilbert space of the universal representation $i_A\rtimes i_G$ of $A\rtimes_{\max} G$. If $(\rho, v)$ denotes the restriction 
of $(i_A, i_G)$ to the subspace $d_\pi H_u$ of $H_u$ we get $\rho(A)''=d_\pi i_A(A)''=d_\pi A_\alpha''$.  Thus we see that $(\rho, v)$ is quasi-equivalent to $(\pi, u)$ as in

\begin{definition}\label{def-quasiequiv}
We say that two nondegenerate covariant representations $(\pi_1, u_1)$ and $(\pi_2,u_2)$ of $(A,G,\alpha)$ are 
{\em quasi-equivalent} if there exists an $\Ad u_1$-$\Ad u_2$-equivariant isomorphism $\Phi:\pi_1(A)''\to \pi_2(A)''$ such that 
$\Phi(\pi_1(a))=\pi_2(a)$ for all $a\in A$. 
\end{definition}

\begin{proposition}\label{lem-quasi} 
Let  $(\pi_1, u_1)$ and $(\pi_2,u_2)$ be nondegenerate covariant representations of $(A,G,\alpha)$. Then the following are equivalent:
\begin{enumerate}
\item \label{lem-quasi 1}$(\pi_1, u_1)$ and $(\pi_2,u_2)$
are quasi-equivalent.
\item \label{lem-quasi 2}$\pi_1$ and $\pi_2$ are quasi-equivalent as representations of $A$.
\item \label{lem-quasi 3}$c_{\pi_1}=c_{\pi_2}$ in $Z(A^{**})$.
\item \label{lem-quasi 4}$d_{\pi_1}=d_{\pi_2}$ in $Z(A_\alpha'')$.
\end{enumerate}
Moreover the correspondence $(\pi,u)\mapsto d_\pi$ sets up a bijection between quasi-equivalence classes of nondegenerate
 covariant representations and $G$-invariant central projections in $A_\alpha''$. 
\end{proposition}
\begin{proof} The implications \eqref{lem-quasi 1} $\Rightarrow$ \eqref{lem-quasi 2} and \eqref{lem-quasi 3} $\Rightarrow$ \eqref{lem-quasi 4} are  trivial, and
\eqref{lem-quasi 2} $\Rightarrow$ \eqref{lem-quasi 3} follows from \cite{Pedersen:1979zr}*{Theorem 3.8.2}.  The implication \eqref{lem-quasi 4} $\Rightarrow$ \eqref{lem-quasi 1} follows 
from the above observed fact that the $\alpha''$-$\Ad u_i$-equivariant maps $\pi_i'':A_\alpha''\to \pi_i(A)''$ of Proposition \ref{prop-universal}
induce $G$-equivariant isomorphisms  $d_{\pi_i} A_\alpha''\cong \pi_i(A)''$, $i=1,2$.

For the final statement, the discussion so far shows that the correspondence $(\pi,u)\mapsto d_\pi$ is well defined and injective.  To see that it is surjective, let $(i_A,i_G):(A,G)\to \Bd(H_u)$ be the universal representation.  If $p\in Z(A_\alpha'')$ is a $G$-invariant projection then this restricts to a representation $(\pi,u)$ on $pH_u$, and we have by definition that $d_\pi=p$.
\end{proof}

Thanks to Proposition \ref{lem-quasi}, the following extension of the classical terminology (compare \cite{Pedersen:1979zr}*{3.8.1 and 3.8.12}) seems reasonable.

\begin{definition}\label{def:ccover}
Let $(\pi,u)$ be a nondegenerate covariant representation of $(A,G)$.  Then the projection $d_\pi\in A_\alpha''$ is called the \emph{equivariant central cover} of $(\pi,u)$.
\end{definition}

Of course, if $A^{**}=A_\alpha''$ (for example if $G$ is discrete), then the equivariant central cover of $(\pi,u)$ is just the classical central cover of $\pi$.

\section{The predual of the enveloping $G$-von Neumann algebra}\label{sec:predual}
In this section, we show that our enveloping $G$-von Neumann algebra is the same as the one defined by Ikunishi in \cite{Ikunishi}*{page 2}.  We then discuss the $G$-continuous functionals on $A$ in a little more detail.  This material will be used in our treatment of the weak quasi-central approximation property (wQAP) in Section \ref{sec:wQAP}.

Let $A$ be a $G$-$C^*$-algebra, and let $A^{*,c}$  consist of all elements $\phi\in A^*$
such that the map $g\mapsto \alpha_g^*(\phi)$ defined by $\alpha_g^*(\phi)(a):=\phi(\alpha_{g^{-1}}(a))$ is norm continuous.  In \cite{Ikunishi}*{Theorem 1} Ikunishi constructed  the {\em universal $W^*$-dynamical system} $(M_\alpha, G, \bar\alpha)$ 
of a $G$-$C^*$-algebra $(A,\alpha)$ as the quotient of $A^{**}$ by the annihilator of $A^{*,c}$.  The following proposition, due to Ikunishi for $*$-homomorphisms (although the proof for ccp maps is the same), shows that $(M_\alpha,\bar\alpha)$ enjoys the same universal properties as $(A_\alpha'',\alpha'')$.

\begin{proposition}[see \cite{Ikunishi}*{Theorem 1}]\label{prop-Ikunishi}
Let $(N,\sigma)$ be any $G$-von Neumann algebra and let $\phi:A\to N$ be a $G$-equivariant $*$-homomorphism (resp. ccp map). 
Then there exists a unique normal $G$-equivariant $*$-homomorphic (resp. ccp) extension $\bar\phi:M_\alpha\to N$. 
In particular, the identity $\id:A\to A$ extends to an isomorphism $(M_\alpha, \bar{\alpha})\cong (A_\alpha'', \alpha'')$.
\end{proposition}
\begin{proof}
First note that the induced action on the predual $N_*$ of $N$ is norm continuous: this can be deduced by representing 
$N$ faithfully, equivariantly, and normally into some $\Bd(H)$  via some regular representation $(\pi, u)$ as in Section \ref{Sec:Prel}, and observing that
the induced action $\Ad u$ on the space of trace-class operators  is norm continuous. 
Thus we get a dual map $\phi^*: N_*\to A^{*,c}$ by $N_*\ni \psi\mapsto \psi\circ \phi$. 
The  extension $\bar\phi$ is then given by  $\bar{\phi}=\phi^{**}:M_\alpha=(A^{*,c})^*\to N$.
If $\phi$ is ccp, then so is $\bar{\phi}$ as $A$ is ultraweakly dense in $M_\alpha$. 
The last assertion follows by observing that $\bar{\id}:M_\alpha\to A_\alpha''$ is the inverse of the extension $\id'':A_\alpha''\to M_\alpha$ of 
Corollary \ref{cor-vN}.
\end{proof}

Identifying $M_\alpha$ with $A_\alpha''$ as above, we now immediately get:

\begin{corollary}\label{ccp fun}
The assignment $A\mapsto A_\alpha''$ functorially takes equivariant ccp maps to equivariant and normal ccp maps. \qed
\end{corollary}

In order to give a better description of $A^{*,c}$, note that $L^1(G)$ acts on $A$ and $A^*$ via the formulas 
\begin{equation}\label{dual con}
h*a:=\int_G h(g) \alpha_g(a)\,dg  \quad \text{and}\quad  h*\phi:=\int_G h(g) \alpha_g^*(\phi)\,dg
\end{equation}
(the former integral converges in the norm topology, and the latter in the weak-$*$ topology).   Moreover, for any $a\in A$ and $\phi\in A^*$, one has the formula
\begin{align*}
(h*\phi)(a) & =\int_Gh(g)(\alpha_g^*\phi)(a) \,dg=\int_G h(g)\phi(\alpha_{g^{-1}}(a))\,dg \\ & 
=\int_G h(g^{-1})\Delta(g^{-1})\phi(\alpha_g(a)) \,dg.
\end{align*}
Hence for any $a\in A$, $\phi\in A^*$ and $h\in L^1(G)$, if we define 
$ \check h(g):=h(g^{-1})\Delta(g^{-1})$, we get the formula
\begin{equation}\label{dual con form}
(h*\phi)(a) =\phi(\check h*a). 
\end{equation}

The following proposition is probably well known to experts. As pointed out in \cite{Bearden-Crann} it 
can be deduced from \cite{Hamana:2011}*{Proposition 3.4(i)} (or by a direct computation)
that the span of $L^1(G)*A^*:=\{f*\varphi: f\in L^1(G), \varphi\in A^*\}$ is dense in $A^{*,c}$, and the proposition then follows from an application 
of the Cohen-Hewitt factorization theorem  \cite{Hewitt}*{Theorem (2.5)}. 

\begin{proposition}\label{prop-cont-functionals}
For any $G$-$C^*$-algebra $(A,\alpha)$ we have 
$$A^{*,c}=L^1(G)*A^{*,c}=L^1(G)*A^*. \eqno\qed$$
\end{proposition}

\chapter{Amenable actions}\label{sec:Amenable-Actions}

In this chapter we introduce several a priori different notions of amenability for actions of locally compact groups on $C^*$-algebras.  These are inspired by the work of Anantharaman-Delaroche in the case of discrete groups \cite{Anantharaman-Delaroche:1987os}*{Definition 4.1}. We also  establish some permanence properties and relate amenability to measurewise amenability in the sense of Renault \cite{Renault-LNM}*{Definition II.3.6} (see also \cite{ADR}*{Definitions 3.2.8 and 3.3.1}).   Combining this with a recent result of Bearden and Crann \cite{Bearden-Crann}*{Corollary 4.14}, this affirmatively solves the long-standing open problem of whether measurewise amenability and topological amenability are equivalent for an action $G\curvearrowright X$ of a second countable locally compact group on a second countable locally compact space.

In Section \ref{sec:compare} we introduce our notions of amenability 
 in terms of the enveloping $G$-von Neumann algebra of a $G$-$C^*$-algebra.  A lot of the work in this section is to discuss different approximation properties of $G$-von Neumann algebras.  

In Section \ref{sec:amen pp} we discuss some permanence properties of amenable actions.  This was one of the main motivations for establishing the various equivalent formulations of amenability earlier in the chapter.

Finally, in Section \ref{sec:meas amen} we prove that for second countable $G$ and $X$ an action $G\curvearrowright X$ is measurewise amenable if and only if  the induced action on $C_0(X)$ is amenable.

\section{Amenable actions of locally compact groups}\label{sec:compare}
In this section, we introduce amenable actions of general locally compact groups on $C^*$-algebras, and establish the equivalence of amenability with several other conditions.  The main technical tool is Proposition \ref{prop:Amenability-conditions}, which establishes the equivalence of several approximation properties for a $G$-von Neumann algebra.

Recall that if $M$ is a $G$-von Neumann algebra, we denote by $L^\infty(G, M)$ the von Neumann algebra  tensor product 
$L^\infty(G)\overline{\otimes}M$ equipped with the tensor product action.  We refer to \cite{Tak}*{Chapter V, Theorem 7.17} for the relationship between $L^\infty(G,M)$ and the space of bounded, ultraweakly measurable functions from $G$ to $M$.  We identify $M$ with the subalgebra $ 1\otimes M$ of $L^\infty(G,M)$.  

Anantharaman-Delaroche \cite{Anantharaman-Delaroche:1979aa}*{D\'{e}finition 3.4} defined a continuous action of $G$ on a von Neumann algebra to be \emph{amenable} if there is an equivariant conditional expectation 
$$
P:L^\infty(G,M)\to M.
$$
Later in \cite{Anantharaman-Delaroche:1987os}*{D\'{e}finition 4.1}, Anantharaman-Delaroche defined an action of a {\em discrete} group on a $C^*$-algebra $A$ to be {\em amenable} if the induced action on the double dual $A^{**}$ is amenable in the sense above.  However, this does not make sense for general actions of locally compact groups, as $A^{**}$ is not a $G$-von Neumann algebra in general.  Replacing $A^{**}$ by $A_\alpha''$ in the case where $G$ is a general locally compact group, leads to

\begin{definition}\label{def-vonNeumann}
A $G$-$C^*$-algebra $(A,\alpha)$ (or just the action $\alpha$)\footnote{Throughout this paper, we will treat properties of $G$-$C^*$-algebras and of actions interchangeably: for example, ``Let $A$ be a von Neumann amenable $G$-$C^*$-algebra'' means the same thing as ``Let $\alpha:G\to \text{Aut}(A)$ be a von Neumann amenable action''.}  is called {\em von Neumann amenable},  if there exists a $G$-equivariant conditional expectation 
$$P: L^\infty(G, A_\alpha'')\to A_\alpha''.$$
\end{definition}

\begin{remark}\label{rem-vonNeumannam}
In \cite{Anantharaman-Delaroche:1982aa}*{Corollaire 3.7}, Anantharaman-Delaroche showed that a $G$-von Neumann algebra $M$ is amenable if and only if  $Z(M)$ is amenable, i.e., if and only if there exists a $G$-equivariant  conditional expectation 
$$P: L^\infty(G, Z(M))\to Z(M).$$
In the context of this work, it is important to note that amenability of a $G$\nb-$C^*$\nb-algebra $A$ turns out \emph{not} to be equivalent to amenability of $Z(A)$ in general.
\end{remark}

Although Definition \ref{def-vonNeumann} is a straightforward extension of the established definition of amenable actions for discrete groups,
it is not the most useful one for studying the behaviour of the associated crossed products. In the case of discrete groups
Anantharaman-Delaroche was able to characterize  amenability in terms of several approximation properties involving functions 
of positive type.  The following is based on \cite{Anantharaman-Delaroche:1987os}*{Definition 2.1}.

\begin{definition}\label{amen def}
Let $(A,\alpha)$ be a $G$-$C^*$-algebra, or $G$-von Neumann algebra.  A function $\theta:G\to A$ is of \emph{positive type} (with respect to $\alpha$) if for every finite subset $F\subseteq G$ the  matrix
$$
\big(\alpha_g(\theta(g^{-1}h))\big)_{g,h\in F}\in M_F(A)
$$
is positive.  
\end{definition}

If $A$ is a $G$-$C^*$-algebra, then the prototypical examples of positive type functions are given by $\theta(g)=\braket{\xi}{\gamma_g(\xi)}_{A}$, where $\xi$ is a vector in a $G$-equivariant Hilbert $A$\nb-module $(\mathcal{H},\gamma)$.  
Indeed, it is shown  in \cite{Anantharaman-Delaroche:1987os}*{Proposition~2.3}
that every continuous positive type function into a $G$-$C^*$-algebra is associated to some essentially unique triple $(\mathcal{H},\gamma,\xi)$ as above.

\begin{definition}\label{def-amenable (A)}
A $G$-$C^*$-algebra $(A,\alpha)$ is {\em amenable} if there exists a net 
of norm-continuous, compactly supported, positive type functions 
$\theta_i:G\to Z(A_\alpha'')$ such that
$\|\theta_i(e)\|\leq 1$ for all $i\in I$, and  $\theta_i(g)\to 1_{A_\alpha''}$ ultraweakly and uniformly for $g$ in compact subsets of $G$.
\end{definition}

\begin{definition}\label{def-amenable (SA)}
A $G$-$C^*$-algebra $(A,\alpha)$ is {\em strongly amenable}, if there exists a net 
 of norm-continuous, compactly supported, positive type functions $\theta_i:G\to Z\M(A)$
such that
$\|\theta_i(e)\|\leq 1$ for all $i\in I$, and  $\theta_i(g)\to 1_{Z\M(A)}$ strictly and uniformly for $g$ in compact subsets of $G$.
\end{definition}

The next several remarks record connections between these notions and the existing literature.

\begin{remark}
If $G$ is discrete then $A_\alpha''=A^{**}$.  It follows that for $G$ discrete, amenability as in Definition \ref{def-amenable (A)} for $(A,\alpha)$ is equivalent to von Neumann amenability as in Definition \ref{def-vonNeumann}: this was shown in  1987 by Anantharaman-Delaroche \cite{Anantharaman-Delaroche:1987os}*{Th\'{e}or\`{e}me 3.3}.  On the other hand, the equivalence between amenability and von Neumann amenability was established very recently for all actions of locally compact groups by Bearden and Crann in \cite{Bearden-Crann}*{Theorem 3.6}; we will discuss this in more detail below.
\end{remark}

\begin{remark}
If $A$ is unital, then strict convergence in $\M(A)=A$ coincides  with norm convergence.  It follows that strong amenability as in Definition \ref{def-amenable (SA)} extends the notion of amenability defined by Brown and Ozawa in \cite{Brown:2008qy}*{Definition 4.3.1}. \end{remark}

\begin{remark}\label{rem-SAtoA}
Since the inclusion $i_A:A\into (A\rtimes_\max G)^{**}$ extends to a unital inclusion of $\M(A)$ into 
$A_\alpha''$ such that $Z\M(A)$ is mapped into $Z(A_\alpha'')$, we see that strong amenability as in Definition \ref{def-amenable (SA)} implies amenability as in Definition \ref{def-amenable (A)}. 

If $A$ is commutative and $G$ is discrete, then Anantharaman-Delaroche showed that strong amenability is equivalent to amenability in \cite{Anantharaman-Delaroche:1987os}*{Th\'{e}or\`{e}me 4.9}. This result has been extended recently by Bearden and Crann in \cite{Bearden-Crann}*{Corollary 4.14} to actions of arbitrary locally compact groups on commutative $C^*$-algebras.  

At  the other extreme, if $A$ is simple, then $Z\M(A)\cong \C$ and so strong amenability for a $G$-action on $A$ implies amenability of $G$. This is in stark contrast to the fact that there are important examples,  given by Suzuki in \cite{Suzuki:2018qo} (see also \cite{Ozawa-Suzuki}), of amenable actions of non-amenable discrete exact groups on unital, simple and nuclear  $C^*$-algebras: see the discussion in \cite{Buss:2019}*{Section 3}. This implies in particular that strong amenability is strictly stronger than amenability and it indicates that strong amenability is not the `correct' notion of amenability for actions on noncommutative $C^*$-algebras.
\end{remark}

The following result, which is a consequence of \cite{Anantharaman-Delaroche:2002ij}*{Proposition 2.5}, relates strongly amenable (and hence amenable)  actions on commutative $C^*$-algebras to topologically amenable actions on spaces as in  
\cite{Anantharaman-Delaroche:2002ij}*{Definition 2.1}.

\begin{proposition}\label{rem-amenable} 
An action $G\curvearrowright X$ of a locally compact group $G$ on a locally compact space $X$ is topologically amenable if and only if the induced action $\alpha:G\to \Aut(C_0(X))$ is strongly  amenable.
\end{proposition}
\begin{proof}  It is shown in \cite{Anantharaman-Delaroche:2002ij}*{Proposition 2.5} that the action of $G$ on $X$ is topologically amenable if and only if there is a net $(h_i)$ in $C_c(X\times G)$ of positive type functions (in the sense of \cite{Anantharaman-Delaroche:2002ij}*{Definition 2.3}) that tend to one uniformly on compact subsets of $X\times G$.  Let $(h_i)$ be a net in $C_c(X\times G)$ implementing topological amenability as above.  Then the net $\theta_i:G\to C_0(X)$ defined by $\theta_i(g)(x):=h_i(x,g)$ satisfies the assumptions needed for strong amenability in Definition \ref{def-amenable (SA)}.  Conversely, assume $\theta_i:G\to \M(C_0(X))$ satisfies the conditions in Definition \ref{def-amenable (SA)}.   Let $J$ consist of ordered pairs $(K,\epsilon)$ where $K$ is a compact subset of $X\times G$, and where $\epsilon>0$.  Make $J$ a directed set by stipulating that $(K,\epsilon)\leq (K',\epsilon')$ if $K\subseteq K'$ and $\epsilon'<\epsilon$.  For each $j=(K,\epsilon)\in J$ choose $i_j$ such that $|\theta_{i_j}(g)(x)-1|<\epsilon$ for all $(x,g)\in K$.  For each $j=(K,\epsilon)$, let $K_X\subseteq X$ be the projection of $K$ to $X$, and choose $e_j\in C_c(X)$ with values in the unit interval $[0,1]$ such that $\alpha_g(e_j)(x)=1$ for all $(x,g)\in K\cup K_X\times \{e\}$. Define $h_{j}\in C_c(X\times G)$ by $h_{j}(x,g)=e_j(x)\theta_{i_j}(g)(x)\alpha_g(e_j)(x)$.  Then each $h_j$ is positive type, and the net $(h_{j})$ converges uniformly to one on compact sets, showing amenability of the action of $G$ on $X$ as in \cite{Anantharaman-Delaroche:2002ij}*{Definition 2.1}. 
\end{proof}

Our main goal in the rest of this section is to show that in general amenability implies von Neumann amenability, and 
moreover that both are equivalent if $G$ is exact.  We also relate these to several other approximation properties.   As already mentioned above, since we first circulated a draft of this paper, Bearden and Crann \cite{Bearden-Crann}*{Theorem 3.6} showed that amenability and von Neumann amenability are equivalent in general.   However, the results below still seem worthwhile: first, as the proof is different and simpler in the exact case; and second, as some of the other approximation properties we consider are equivalent to amenability when $G$ is exact, but \emph{not} in general.  

For the next lemma, recall from Section \ref{Sec:Prel} that if $\beta:G\to\Aut(B)$ is any homomorphism of $G$ into
the group of $*$-automorphisms of a $C^*$-algebra $B$, we write $B_c$ for the $C^*$-subalgebra 
of $B$ consisting of all elements $b\in B$ such that the map
$$
G\to B;\quad  g\mapsto\beta_g(b)
$$
is continuous in norm.  

\begin{lemma}\label{sub c gone}
Let $(M,\sigma)$ be a $G$-von Neumann algebra.  Then the following are equivalent:
\begin{enumerate}
\item \label{sub c gone 1} there exists a net of norm-continuous, compactly supported, positive type functions 
\mbox{$\theta_i:G\to M$}  such that
$\|\theta_i(e)\|\leq 1$ for all $i\in I$, and  $\theta_i(g)\to 1_{M}$ ultraweakly and uniformly for $g$ in compact subsets of $G$;
\item \label{sub c gone 2} there exists a net of norm-continuous, compactly supported, positive type functions 
\mbox{$\theta_i:G\to M_c$}  such that
$\|\theta_i(e)\|\leq 1$ for all $i\in I$, and  $\theta_i(g)\to 1_{M}$ ultraweakly and uniformly for $g$ in compact subsets of $G$.
\end{enumerate}
The equivalence also holds if we replace ``\,$\theta_i(g)\to 1_{M}$ ultraweakly'' with ``\,$\theta_i(g)\to 1_{M}$ in norm'' in both conditions.
\end{lemma}

Note that the lemma applied to $M=Z(A_\alpha'')$ implies  that we can replace the codomain in the definition of amenability (Definition \ref{def-amenable (A)}) with a $G$-$C^*$-algebra.

\begin{proof}[Proof of Lemma \ref{sub c gone}]
The implication \eqref{sub c gone 2} $\Rightarrow$ \eqref{sub c gone 1} is clear, so we just need to prove that if $(\theta_i)$ has the properties in \eqref{sub c gone 1}, then we can build a net with the properties in \eqref{sub c gone 2}.

Let $\{V_j: j\in J\}$ be a  
neighbourhood basis of $e$ in $G$ consisting of 
symmetric compact sets such that 
$V_j\subseteq V_{j'}$ if $j\geq j'$ and let $(f_j)_{j\in J}$ be an approximate unit of $L^1(G)$ consisting of continuous positive 
 functions with $\int_G f_j(g)\,dg=1$ and such that $\supp(f_j)\subseteq V_j$ for all $j\in J$.
Assume further that there exists a compact  neighbourhood $V_0=V_0^{-1}$
of $e$ such that 
 $V_j\subseteq V_0$ for all $j\in J$. 
 
For each pair $(i,j)\in I\times J$ we define
$$\theta_{i,j}(g):=\int_G f_j(k)\sigma_k(\theta_i(k^{-1}gk))\, dk.$$
We claim that each $\theta_{i,j}$ is norm-continuous, compactly supported, positive type, satisfies $\|\theta_{i,j}(e)\|\leq 1$, and takes values in $M_c$.  
Indeed, it is straightforward to check that the functions $\theta_{i,j}:G\to M$ are compactly supported, and satisfy $\|\theta_{i,j}(e)\|\leq 1$ for all $i,j$.  Moreover, each is norm-continuous as $\theta_i$ is norm-continuous and compactly supported, so uniformly norm-continuous.  To see that each $\theta_{i,j}$ is positive type, let $F$ be a finite subset of $G$.  Then 
\begin{align*}
\big(\sigma_g(\theta_{i,j}(g^{-1}h))\big)_{g,h\in F} & =\int_G f_j(k)\Big(\sigma_{gk}( \theta_i(k^{-1}g^{-1}hk))\Big)_{g,h\in F}\,dk.
\end{align*}
This is an ultraweakly convergent integral of positive matrices in $M_F(M)$, so positive.  Finally, to see that $\theta_{i,j}$ is valued in $M_c$, we note that for any $g,h\in G$, 
$$
\sigma_h(\theta_{i,j}(g))=\int_G f_j(k)\sigma_{hk}( \theta_i(k^{-1}gk))\,dk\stackrel{l=hk}{=}\int_G f_j(h^{-1}l)\sigma_{l}( \theta_i(l^{-1}hgh^{-1}l))\,dl.
$$
Norm-continuity of this in $h$ follows from uniform norm-continuity of the original $\theta_i$, plus uniform continuity of each $f_j$.

To complete the proof in the case of ultraweak convergence it suffices to show that for any compact $K\subseteq G$, any finite $F\subseteq M_*$ and any $\epsilon>0$ there exists a pair $(i,j)\in I\times J$ such that for all $g\in K$ and all $\psi\in F$ we have
$$|\psi(\theta_{i,j}(g)-1_M)|<\epsilon.$$
We now do this.

Since $\theta_i(g)\to 1_M$ ultraweakly and uniformly on compact sets in $G$, we can find an index 
$i\in I$ such that for all $\psi\in F$ and for all $g\in V_0KV_0$ we get
$$|\psi(\theta_i(g)-1_M)|<\frac{\epsilon}{3}.$$
Let $0\neq C\geq \max\{\|\psi\|: \psi\in F\}$.
Since $\theta_i:G\to M$ is norm-continuous, the image $\theta_i(V_0KV_0)$ is 
a norm-compact subset of $M$. We therefore find finitely many elements $g_1,\ldots, g_r\in V_0KV_0$ 
such that $\theta_i(V_0KV_0)\subseteq \bigcup_{l=1}^r B_{\frac{\epsilon}{3C}}(\theta_i(g_l))$.
Since $\sigma:G\to \Aut(M)$ is ultraweakly continuous, we can further find an index $j\in J$
such that for all $k\in V_j$ and for all $\psi\in F$
$$\big|\psi\big(\sigma_k(\theta_i(g_l))-\theta_i(g_l)\big)\big|<\frac{\epsilon}{3}.$$
Then  for all $k\in V_j$, $g\in K$ and $\psi\in F$ we find a suitable $l\in \{1,\ldots,r\}$ such that
\begin{align*}
&\big|\psi(\sigma_k(\theta_i(k^{-1}gk))-1_M)\big|\\
&\leq \left|\psi\big(\sigma_k\big(\theta_i(k^{-1}gk)-\theta_i(g_l)\big)\big)\right|
+\left|\psi\big(\sigma_k(\theta_i(g_l))-\theta_i(g_l)\big)\right|
+\left|\psi(\theta_i(g_l)-1_M)\right|\\
&\leq \frac{\epsilon}{3}+\frac{\epsilon}{3}+\frac{\epsilon}{3}=\epsilon.
\end{align*}
 We then conclude  for all $g\in K$ and $\psi\in F$, that
\begin{align*}
\big|\psi(\theta_{i,j}(g)-1_M)\big|
&=
\left|\int_G f_j(k)\psi\big(\sigma_k(\theta_i(k^{-1}gk)-1_M)\big)\, dk\right|\\
&\leq \int_G f_j(k) \left| \psi\big(\sigma_k(\theta_i(k^{-1}gk)-1_M)\big)\right|\,dk\\
&\leq \int_G f_j(k)\epsilon\,dk=\epsilon
\end{align*}
which completes the proof under the assumption that $\theta_i(g)\to 1_M$ ultraweakly.

Suppose instead that $\theta_i(g)\to 1_M$ in norm uniformly for $g$ in compact subsets of $G$.  Then if $K\subseteq G$  is compact and $\epsilon>0$, let $i_0\in I$ be such that 
$\|\theta_i(g)-1_M\|<\epsilon$ for all $g\in V_0KV_0$ and $i\geq i_0$. Then, for all $g\in K$,  $j\in J$, and $i\geq i_0$ we get
$$
\|\theta_{i,j}(g)-1_M\|\leq \int_G f_j(k)\big\|\sigma_k\big(\theta_i(k^{-1}gk)-1_M\big)\big\|\,dk
\leq \int_G f_j(k)\epsilon\,dk=\epsilon,
$$
which completes the proof.
\end{proof}

\begin{remark}\label{rem:ZM(A)_c}
Similarly to the proof of Lemma \ref{sub c gone}, one can replace the codomain of the maps $\theta_i$ in the definition of strong amenability (Definition \ref{def-amenable (SA)}) with $(Z\M(A))_c$.  We leave the details to the reader.
\end{remark}

We need some more preliminaries.  First, we recall a standard definition.

\begin{definition}\label{l2ga act}
Let $A$ be a $C^*$-algebra and let $X$ be a locally compact space equipped with a fixed Radon measure $\mu$\footnote{We will mainly be interested in the case that $X=G$ and $\mu$ is Haar measure}.  Define $L^2(X,A)$ to be the Hilbert $A$-module 
given as the completion of $C_c(X,A)$ with respect to the $A$-valued inner product
$$\braket{\xi}{\eta}_{A}=\int_X \xi(x)^*\eta(x)\,d\mu(x).$$
We write $\|\xi\|_2:=\sqrt{\|\braket{\xi}{\xi}_A\|}$ for the associated norm.  Note that $L^2(X,A)$ identifies canonically with the (exterior) Hilbert module tensor product $L^2(X)\otimes A$.

Let now $(A,\alpha)$ be a $G$-$C^*$-algebra.   There is a continuous action $\lambda^\alpha$ of $G$ on $L^2(G, A)$ determined by the formula 
$$
\lambda^\alpha_g(\xi)(h)=\alpha_g(\xi(g^{-1}h)).
$$
A direct computation shows this is \emph{compatible} with the given action of $G$ on $A$, meaning that $\braket{\lambda^\alpha_g \xi}{\lambda^\alpha_g \eta}_A=\alpha_g( \braket{\xi}{\eta}_{A})$ for all $g\in G$ and $\xi,\eta\in L^2(G,A)$.  We call $\lambda^\alpha$ 
the {\em $\alpha$-regular representation} of $G$ on $L^2(G,A)$.  
\end{definition}

We record  a well-known result of Anantharaman-Delaroche \cite{Anantharaman-Delaroche:1987os}*{Proposition~2.5} for later use.

\begin{lemma}\label{lem:pt cs}
Let $A$ be a $G$-$C^*$-algebra.  For any norm-continuous, compactly supported, positive type function $\theta:G\to A$, there exists $\xi\in L^2(G,A)$ such that 
$$\theta(g)=\braket{\xi}{\lambda^\alpha_g\xi}_{A}$$
for all $g\in G$. \qed
\end{lemma}

\begin{lemma}\label{prop-approx}
Let $A\subseteq \Bd(H)$ be a concrete $C^*$-algebra
and let $\alpha:G\to \Aut(A)$ be a homomorphism (not necessarily continuous in any sense). Then the following are equivalent:
\begin{enumerate}
\item \label{prop-approx 1} There exists a net $(\xi_i)_{i\in I}$ of vectors in the unit ball of $L^2(G,A_c)$ such that 
$$\braket{\xi_i}{\lambda^\alpha_g\xi_i}_{A_c}\to 1_H$$
weakly in $\Bd(H)$ and  uniformly for $g$ in compact subsets of $G$. 
\item \label{prop-approx 2} There exists a net $(\xi_i)_i$ in $C_c(G,A_c)$ with $\|\xi_i\|_2\leq 1$ for all $i\in I$
and $$\braket{\xi_i}{\lambda^\alpha_g\xi_i}_{A_c}\to 1_H$$
weakly in $\Bd(H)$ and uniformly for $g$ in compact subsets of $G$.
\item \label{prop-approx 3} There exists a net $(\theta_i)_{i\in I}$ in $C_c(G,A_c)$  of positive type functions such that
$\|\theta_i(e)\|\leq 1$ for all $i\in I$, and  $\theta_i(g)\to 1_H$  weakly and uniformly for $g$ in compact subsets of $G$.
\end{enumerate}

The same equivalences hold if we replace ``weakly'' everywhere it appears by either ``in norm'' or ``strictly in $\M(A)_c$''. 
\end{lemma}

\begin{proof} \eqref{prop-approx 2} $\Rightarrow$ \eqref{prop-approx 3} follows by defining $\theta_i(g)=\braket{\xi_i}{\lambda^\alpha_g\xi_i}_{A_c}$ for all $g\in G$ and 
\eqref{prop-approx 3} $\Rightarrow$ \eqref{prop-approx 1} follows from  Lemma \ref{lem:pt cs}.  In order to show \eqref{prop-approx 1} $\Rightarrow$ \eqref{prop-approx 2}, let $(\xi_i)_i$ be a net of unit vectors in $L^2(G,A_c)$ such that 
$\braket{\xi_i}{\lambda^\alpha_g\xi_i}_{A_c}\to 1_H$ weakly and uniformly for $g$ in compact subsets of $G$. 
Since $C_c(G,A_c)$ is dense in $L^2(G,A_c)$ we can find for each $i\in I$ and $n\in \N$
an element $\eta_{i,n}\in C_c(G,A)$ with $\|\eta_{i,n}\|_2\leq 1$ and $\|\xi_i-\eta_{i,n}\|_2\leq \frac1{n}$. 
Then, with the canonical  order on $I\times \N$ one easily checks that $(\eta_{i,n})_{(i,n)}$ 
satisfies the conditions in \eqref{prop-approx 2}. For norm or strict approximation the arguments are analogous.
\end{proof}

We need one more definition before getting to the main results.  
Let $M\sbe \Bd(H)$ be a $G$-von Neumann algebra with $G$-action $\sigma$.
For $\xi\in C_c(G,M)$ there is a bounded operator $\xi:H\to L^2(G,H)$ defined by $(\xi\cdot v)(g):=\xi(g)v$.  With notation as in Definition \ref{l2ga act}, one computes directly that $\|\xi\cdot v\|_{L^2(G,H)}^2=\langle v,\braket{\xi}{\xi}_M v\rangle_H$ and therefore the map $C_c(G,M)\to \Bd(H,L^2(G,H))$ extends uniquely to an isometric inclusion $L^2(G,M)\into \Bd(H,L^2(G,H))$.

\begin{definition}\label{l2wgm}
With notation as above, we define $L^2_w(G,M)$ to be the weak closure of $L^2(G,M)$ inside $\Bd(H,L^2(G,H))$.
\end{definition}

Note that $L^2_w(G,M)$ is still a Hilbert $M$-module: considering $L^2(G,M)$ as inside $\Bd(H,L^2(G,H))$, the inner product on $L^2(G,M)$ is given by operator composition via the formula $\braket{\xi}{\eta}=\xi^*\circ \eta$; one checks that this still gives an $M$-valued inner product for $\xi$ and $\eta$ in the weak closure of $L^2(G,M)$.  As any two normal, faithful, (unital) representations of a von Neumann algebra have unitarily equivalent amplifications (see for example \cite[III.2.2.8]{Blackadar:2006eq}), $L^2_w(G,M)$ does not depend on the choice of representation of $M$, up to canonical isomorphism\footnote{One can also see that $L^2_w(G,M)$ is a Hilbert $M$-module that does not depend on the chosen representation of $M$ by identifying it with the selfdual completion of the Hilbert $M$\nb-module $L^2(G,M)$ as a ternary ring of operators: see the discussion on page 357 of \cite{Blecher:2004} and also \cite{Zettl}.}.

In particular, replacing the Hilbert space $H$ if necessary, we may assume without loss of generality that the $G$-action on $M$ is implemented by a unitary representation $u\colon G\to \Bd(H)$, that is, $\sigma_g=\Ad{u_g}$.  Considering $L^2(G,M)$ as a subset of $\Bd(H,L^2(G,H))$, the action\footnote{It is not necessarily strongly continuous for the Hilbert module norm, but this does not matter.} $\lambda^\sigma$ on $L^2(G,M)$ of Definition \ref{l2ga act} is then induced by the formula
$$
(\lambda^\sigma_g)(x)=(\lambda_g\otimes u_g)\circ x \circ u_g^*,\quad x\in L^2(G,M).
$$
Clearly this formula extends $\lambda^\sigma$ to a weakly continuous action on $L^2_w(G,M)$.  Note finally that there is an equivariant action of $L^\infty(G,M)$ on $L^2_w(G,M)$ by adjointable operators that extends the natural tensor product action of $L^\infty(G)\otimes M$ on $L^2(G,M)=L^2(G)\otimes M$.
This can be realized concretely via the canonical representation of 
$L^\infty(G,M)=L^\infty(G)\overline{\otimes} M$ on $L^2(G,H)=L^2(G)\otimes H$: 
the action of $f\in L^\infty(G,M)$ on $x \in L^2_w(G,M)$ is realized by the operator composition $f\circ x:H \to L^2(G,H)$, which one checks belongs to $L^2_w(G,M)$.  We leave the algebraic checks that the action is equivariant and by adjointable operators to the reader.

The following result relates amenability of a $G$-von Neumann algebra to certain approximation properties.  It is the key technical result of this section.

\begin{proposition}\label{prop:Amenability-conditions}
Let $(M,\sigma)$ be a $G$-von Neumann algebra and consider the following assertions:
\begin{enumerate}
\item \label{amen 1} there is a net of compactly supported, norm-continuous, positive type functions 
 \mbox{$\theta_i:G\to Z(M)_c$} 
with  $\|\theta_i(e)\|\leq 1$ and $\theta_i(g)\to 1$ in norm  uniformly for $g$ in compact subsets of $G$;
\item \label{amen 2} there is a net of compactly supported, norm-continuous, positive type functions 
 \mbox{$\theta_i:G\to Z(M)_c$} 
with  $\|\theta_i(e)\|\leq 1$ and $\theta_i(g)\to 1$ ultraweakly and uniformly for $g$ in compact subsets of $G$;
\item \label{amen 3} there is a bounded net $(\xi_i)$ in $\contc(G,Z(M)_c)\sbe L^2(G,Z(M)_c)$ with 
$$\braket{\xi_i}{\lambda^\sigma_g(\xi_i)}_{Z(M)_c}\to 1$$ ultraweakly and uniformly for $g$ in compact subsets of $G$;
\item \label{amen 4} for any ultraweakly dense $*$-subalgebra $A$ of $M$, there is a bounded net $(\xi_i)$ in $L^2_w(G,M)$ such that 
$$\braket{\xi_i}{{a}\lambda^\sigma_g(\xi_i)}_{M}\to {a}$$ 
ultraweakly and pointwise for each $g\in G$ and each $a\in A$;
\item \label{amen 5} the $G$-von Neumann algebra $(M,\sigma)$  is amenable, i.e., there is a $G$\nb-equivariant projection 
$P\colon L^\infty(G,M)\to M$;
\item \label{amen 6} there is a ucp $G$-map $L^\infty(G)\to Z(M)$;
\item \label{amen 7} there is a ucp $G$-map $\contub(G)\to Z(M)_c$.
\end{enumerate}
Then
$$\eqref{amen 1} \Rightarrow\eqref{amen 2} \Leftrightarrow \eqref{amen 3}\Rightarrow\eqref{amen 4}\Rightarrow \eqref{amen 5}\Rightarrow\eqref{amen 6}\Leftrightarrow\eqref{amen 7}
$$
and all of these conditions are equivalent if $G$ is exact.
\end{proposition}

Note that by Lemma \ref{sub c gone} we may replace $Z(M)_c$ by $Z(M)$ in statements \eqref{amen 1} and \eqref{amen 2} above. 
 As mentioned before, Bearden and Crann showed (amongst other things) in
\cite{Bearden-Crann}*{Theorem 3.6}  that  \eqref{amen 5} implies \eqref{amen 2} even for non-exact groups $G$.  Using this and the implications from Proposition \ref{prop:Amenability-conditions}, we then get

\begin{theorem}[Bearden-Crann, plus Proposition \ref{prop:Amenability-conditions}]\label{Bearden-Crann}
Items \eqref{amen 2} -- \eqref{amen 5} in Proposition \ref{prop:Amenability-conditions} are all equivalent.\qed
\end{theorem}

\begin{proof}[Proof of Proposition \ref{prop:Amenability-conditions}]
The implication \eqref{amen 1} $\Rightarrow$ \eqref{amen 2} is clear,  and  \eqref{amen 2} $\Leftrightarrow$ \eqref{amen 3} follows from Lemma \ref{prop-approx}. 
The implication \eqref{amen 3} $\Rightarrow$ \eqref{amen 4} follows (with $A=M$) because, since $(\xi_i)$ takes values in the centre, we have 
$$\braket{\xi_i}{m\lambda^\sigma_g(\xi_i)}_{M}=\braket{\xi_i}{\lambda^\sigma_g(\xi_i)}_{M}m$$ for all $m\in M$, and as multiplication is separately ultraweakly continuous on bounded sets.
 The implication \eqref{amen 5} $\Rightarrow$ \eqref{amen 6} follows by restriction of $P$ to $L^\infty(G)$, and using that $M$ is in the multiplicative domain of $P$ to conclude that the image of the restriction is central. 
The implication \eqref{amen 6} $\Rightarrow$ \eqref{amen 7} follows by taking continuous parts and using that $L^\infty(G)_c=\contub(G)$ (see line \eqref{contub def} above).  
The implication \eqref{amen 7} $\Rightarrow$ \eqref{amen 1} for $G$ exact follows from \cite{Ozawa-Suzuki}*{Proposition 2.5}\footnote{See also \cite{Brodzki-Cave-Li:Exactness}*{Theorem 5.8}, which handles the case of second countable $G$.} which implies that the $G$-action on $\contub(G)$ is strongly amenable if (and only if) $G$ is exact. 

It therefore remains to establish the implications \eqref{amen 4} $\Rightarrow$ \eqref{amen 5} and \eqref{amen 7} $\Rightarrow$ \eqref{amen 6}. For \eqref{amen 4} $\Rightarrow$ \eqref{amen 5} we shall use the idea of the proof of \cite{Abadie:2019kc}*{Lemma~6.5}, which is exactly the implication \eqref{amen 4} $\Rightarrow$ \eqref{amen 5} we need for $G$ discrete.  

First notice that once \eqref{amen 4} holds for $a\in A$, then it also holds for $a$ in the norm closure of $A$ because the net $(\xi_i)$ is bounded. Hence we may assume without loss of generality that $A\sbe M$ is a \cstar{}subalgebra, and in particular, that there is an increasing approximate unit $(e_j)$ for $A$.  
Using the canonical left action of $L^\infty(G,M)$ by adjointable operators on the von Neumann Hilbert module $L^2_w(G,M)$ we define for each $i$ and for all $f\in L^\infty(G,M)$,
$$Q_{i}(f):=\braket{\xi_i}{f \xi_i},$$ 
where here (and below) we put $\braket{\cdot}{\cdot}:=\braket{\cdot}{\cdot}_M$.
Then $(Q_i)$ is a uniformly bounded net of completely positive linear maps $L^\infty(G,M)\to M$ with $\sup_{i}\|Q_i\|\leq \sup_{i}\|\xi_i\|^2_2<\infty$. 
The set of all completely positive linear maps $L^\infty(G,M)\to M$ that  are bounded by a fixed constant $C\geq 0$ is compact
 with respect to the pointwise ultraweak topology (this follows for example from \cite{Brown:2008qy}*{Theorem 1.3.7}).  Hence we may assume (after passing to a subnet if necessary) that  there is a cp linear map $Q\colon L^\infty(G,M)\to M$ with $Q(f)=\lim_i Q_i(f)$ ultraweakly for all $f\in L^\infty(G,M)$.

Recall that $(e_j)$ is an increasing approximate unit for $A$.  For each $j$ we define the cp map $P_j:L^\infty(G,M)\to M$ as
\begin{equation}\label{Pj}
P_j(f):=Q(e_j^{1/2}fe_j^{1/2})\quad f\in L^\infty(G,M),
\end{equation}
and for all pairs $(i,j)$ and $f\in L^\infty(G,M)$ we define 
$$P_{i,j}(f):=Q_i(e_j^{1/2}fe_j^{1/2})=\braket{\xi_i}{e_j^{1/2}fe_j^{1/2}\xi_i}.$$
Then for each $j$, $P_j(f)=\lim_iP_{i,j}(f)=\lim_i\braket{\xi_i}{e_j^{1/2}fe_j^{1/2}\xi_i}$ for all $f\in L^\infty(G,M)$.

 Let us consider the restriction of $P_j$ to the centre $Z(L^\infty(G,M))$, which equals $L^\infty(G,Z(M))$ by \cite{Tak}*{Chapter V, Corollary 5.11}. Hence we get a net $(P_j)$ of cp maps $P_j\colon L^\infty(G,Z(M))\to M$. We shall prove that this net converges pointwise-ultraweakly to a $G$-equivariant projection $P\colon L^\infty(G,Z(M))\to Z(M)$. As we already know that the existence of such a projection is equivalent to the von Neumann amenability of $M$ (see Remark \ref{rem-vonNeumannam}), this will complete the proof of \eqref{amen 4} $\Rightarrow$ \eqref{amen 5}.  

We first claim that the net $(P_j)$ converges pointwise ultraweakly to a cp map $P:L^\infty(G,Z(M))\to M$.  For this it suffices to show that for any positive $f\in L^\infty(G,Z(M))$, the net $(P_j(f))$ is increasing.  Notice first that $f$ commutes with all elements of $M\sbe L^\infty(G,M)$, so in particular it commutes with all elements of $A$.  Hence for any $j\leq k$ and any positive $f\in L^\infty(G,Z(M))$, we have 
$$
P_j(f)\leq Q(e_j^{1/2}fe_j^{1/2})=Q(f^{1/2}e_jf^{1/2})\leq Q(f^{1/2}e_kf^{1/2})=Q(e_k^{1/2}fe_k^{1/2})=P_k(f)
$$
as required.  Let $P\colon L^\infty(G,Z(M))\to M$ be the pointwise ultraweak limit of $(P_j)$.

We next claim that $P$ takes image in $Z(M)$, and is a projection.  It is enough to show that $P(f)a=aP(f)$ for every positive $f\in L^\infty(G,Z(M))$ and every positive $a\in A$.  For this it is enough to show that $P(f)a$ is positive. By construction,
$$P(f)a=\lim_j\lim_i \braket{e_j^{1/2}\xi_i}{fe_j^{1/2}\xi_i}a=\lim_j\lim_i \braket{e_j^{1/2}\xi_i}{fe_j^{1/2}\xi_ia}$$
ultraweakly.  Now, for any normal state $\phi$ on $M$ we  compute
\begin{align}\label{cs est}
\Big|\phi\Big(\braket{e_j^{1/2}\xi_i}{fe_j^{1/2}\xi_ia}&- \braket{e_j^{1/2}\xi_i}{fae_j^{1/2}\xi_i}\Big)\Big|^2=
\Big|\phi\left(\braket{e_j^{1/2}\xi_i}{f(e_j^{1/2}\xi_ia-ae_j^{1/2}\xi_i)}\right)\Big|^2 \nonumber \\
&\leq \|\xi_i\|_2^2\|f\|^2\phi\big(\braket{e_j^{1/2}\xi_ia-ae_j^{1/2}\xi_i}{e_j^{1/2}\xi_ia-ae_j^{1/2}\xi_i}\big)
\end{align}
and notice that
\begin{multline*}
\lim_j\lim_i\phi\big(\braket{e_j^{1/2}\xi_ia-ae_j^{1/2}\xi_i}{e_j^{1/2}\xi_ia-ae_j^{1/2}\xi_i}\big)\\
=\lim_j\lim_i\Big(\phi(\braket{e_j^{1/2}\xi_ia}{e_j^{1/2}\xi_ia})-\phi(\braket{e_j^{1/2}\xi_ia}{ae_j^{1/2}\xi_i})\\-
\phi(\braket{ae_j^{1/2}\xi_i}{e_j^{1/2}\xi_ia})+\phi(\braket{ae_j^{1/2}\xi_i}{ae_j^{1/2}\xi_i})\Big)=0
\end{multline*}
since all terms are converging to $\phi(a^*a)$ by the assumption on the net $(\xi_i)$. But since $fa\geq 0$ (as $f$ and $a$ are positive and commute), this and line \eqref{cs est} then proves
$$\phi\big(P(f)a\big)=\lim_j\lim_i\phi\big(\braket{e_j^{1/2}\xi_i}{fae_j^{1/2}\xi_i}\big)=\lim_j\lim_i \phi\big(P_{i,j}(fa)\big)\geq 0$$
as desired. Thus $P$ takes image into $Z(M)$ and can therefore be viewed as a cp map $P\colon L^\infty(G,Z(M))\to Z(M)$. It is a projection because for every $m\in Z(M)$, 
$$P(m)=\lim_j\lim_i\braket{e_j^{1/2}\xi_i}{me_j^{1/2}\xi_i}=\lim_j\lim_i\braket{\xi_i}{e_jm\xi_i}=\lim_j e_jm=m,$$
where all limits are with respect to the ultraweak topology (note that as $A$ is ultraweakly dense in $A$, $(e_j)$ converges ultraweakly to the identity of $M$).

It remains to verify that $P$ is $G$-equivariant.   
This is equivalent to the identity 
\begin{equation}\label{eq-Ginvariant}
\sigma_g\big(P((\tau\otimes \sigma)_{g^{-1}}(f))\big)=P(f)
\end{equation}
 for every positive $f\in L^\infty(G,Z(M))$ and $g\in G$.
Observe that for all $i$ and $j$ we have
$$
\sigma_g\big(P_{ij}((\tau\otimes \sigma)_{g^{-1}}(f))\big)=\sigma_g\big(\braket{e_j^{1/2}\xi_i}{\Ad \lambda^\sigma_{g^{-1}}(f)e_j^{1/2}\xi_i}\big)
=\braket{\lambda_g^\sigma(e_j^{1/2}\xi_i)}{f\lambda_g^{\sigma}(e_j^{1/2}\xi_i)}.
$$
Thus, in order to prove (\ref{eq-Ginvariant}) it suffices to show that
$$\Big|\phi\left(\braket{\lambda^\sigma_g(e_j^{1/2}\xi_i)}{f \lambda^\sigma_g(e_j^{1/2}\xi_i)}-\braket{e_j^{1/2}\xi_i}{fe_j^{1/2}\xi_i}\right)\Big|\to 0$$
 for every normal state $\phi$ of $M$. The triangle inequality implies that every semi-norm satisfies 
$$\big|\|x\|^2-\|y\|^2\big|=\big(\|x\|+\|y\|\big)\big|\|x\|-\|y\|\big|\leq \big(\|x\|+\|y\|\big)\|x-y\|;$$
applying this to the semi-norm $\|\cdot\|_\phi$ induced by the semi-inner product $\braket{\xi}{\eta}_\phi:=\phi(\braket{\xi}{\eta})$ on $L^2_w(G,M)$
we deduce that
\begin{multline*}
\Big|\phi\left(\braket{\lambda^\sigma_g(e_j^{1/2}\xi_i)}{f\cdot \lambda^\sigma_g(e_j^{1/2}\xi_i)}-\braket{e_j^{1/2}\xi_i}{fe_j^{1/2}\xi_i}\right)\Big|
\\
= \Big|\|f^{1/2}\lambda^\sigma_g(e_j^{1/2}\xi_i)\|^2_\phi-\|f^{1/2}e_j^{1/2}\xi_i\|^2_\phi\Big|\\
\leq 2\sup_{k}\|\xi_k\|_\phi\cdot\|f\|\cdot \|\lambda^\sigma_g(e_j^{1/2}\xi_i)-e_j^{1/2}\xi_i\|_\phi.
\end{multline*}
Finally notice that 
\begin{align*}
&\|\lambda^\sigma_g(e_j^{1/2}\xi_i)-e_j^{1/2}\xi_i\|^2_\phi\\
&=\phi\big(\braket{\lambda^\sigma_g(e_j^{1/2}\xi_i)-e_j^{1/2}\xi_i}{\lambda^\sigma_g(e_j^{1/2}\xi_i)-e_j^{1/2}\xi_i}\big)\\
&=\phi\Big(\sigma_g(\braket{e_j^{1/2}\xi_i}{e_j^{1/2}\xi_i})-\braket{\lambda^\sigma_g(e_j^{1/2}\xi_i)}{e_j^{1/2}\xi_i}
-\braket{e_j^{1/2}\xi_i}{\lambda^\sigma_g(e_j^{1/2}\xi_i)}+\braket{e_j^{1/2}\xi_i}{e_j^{1/2}\xi_i}\Big).
\end{align*}
Taking the limit in $i$ and using the properties of $(\xi_i)$, this converges to 
$$
\phi\Big(\sigma_g(e_j)-\sigma_g(e_j^{1/2})e_j^{1/2}-e_j^{1/2}\sigma_g(e_j^{1/2})+e_j\Big),
$$
and taking now the limit in $j$, using that $(e_j)$ is an approximate unit (and hence converges ultraweakly to $1$), this converges to zero.

Finally, it remains to show that \eqref{amen 7} $\Rightarrow$ \eqref{amen 6}.  We adapt ideas from the proof of 
\cite{Anantharaman-Delaroche:1987os}*{Lemme 2.1}.  Let $\Phi:\contub(G)\to Z(M)_c$ be a ucp $G$-map as in \eqref{amen 7}.
Then for every positive function $h\in C_c(G)$ with $\|h\|_1=1$ it follows from line (\ref{contub def}) that 
$$\Phi_h:L^\infty(G)\to Z(M)_c; \quad \Phi_h(f):=\Phi(h*f)$$
gives a well-defined completely positive map.   As it is clearly unital, it is thus ucp.

Let $I$ be a directed set and let $(U_i)_{i\in I}$ be a neighbourhood base of the identity in $G$ with $U_i\subseteq U_j\Leftrightarrow i\geq j$.
For each $i\in I$ let $h_i\in C_c(G)$ be a positive function with $\|h_i\|_1=1$ and $\supp(h_i)\subseteq U_i$; write $\Phi_i:=\Phi_{h_i}$. 
Then $(\Phi_i)$ is a net of ucp maps from $L^\infty(G)$ to $Z(M)$, and by \cite{Brown:2008qy}*{Theorem 1.3.7}
we may assume, after passing to a subnet if necessary, that there exists a ucp map $\tilde\Phi:L^\infty(G)\to Z(M)$ such that $\tilde\Phi(f)=\lim_i\Phi_i(f)$
for all $f\in L^\infty(G)$, where the limit is in the ultraweak topology of $Z(M)$.

It remains to show that $\tilde\Phi$ is $G$-equivariant. First observe that for each $h\in L^1(G)$ 
we have $\| h*h_i-h_i*h\|_1\to 0$ and therefore 
\begin{equation}\label{h hi com}
\|(h*h_i-h_i*h)*f\|_\infty\leq\| h*h_i-h_i*h\|_1\|f\|_\infty\to 0
\end{equation}
for all $f\in L^\infty(G)$. Recall from Section \ref{Sec:Prel} that we denote the translation action on $C_{ub}(G)$ by $\tau$.  Then for all $f\in C_{ub}(G)$ and $h\in L^1(G)$ it follows from the $G$-equivariance of $\Phi$ 
that
\begin{align*}
\Phi(h*f)&=\Phi\left(\int_G h(g)\tau_g(f)\,dg\right)
=\int_G h(g)\Phi(\tau_g(f))\,dg\\
&=\int_Gh(g)\sigma_g(\Phi(f))\,dg
=h*\Phi(f).
\end{align*}
Since convolution with $h$  is ultraweakly continuous on $Z(M)$, this and line \eqref{h hi com} imply that for all $h\in L^1(G)$ and $f\in L^\infty(G)$:
$$\tilde\Phi(h*f)=\lim_i\Phi(h_i*h*f)=\lim_i\Phi(h*h_i*f)=\lim_ih*\Phi(h_i*f)=h*\tilde\Phi(f).$$
Recall now from line (\ref{eq-Mc}) that 
$$\sigma_g(h*m)=\tau_g(h)*m\quad(\text{resp.}\quad\tau_g(h*f)=\tau_g(h)*f\;)$$
for every $g\in G$, $h\in L^1(G)$, and $m\in M$ (resp.  $f\in L^\infty(G)$). On the other hand, for all $g\in G$, $h\in L^1(G)$ and 
$m\in M$, we get
$$h*\sigma_g(m)=\int_Gh(s)\sigma_{sg}(m)\, ds\stackrel{s\mapsto sg^{-1}}{=} \int_G\Delta(g^{-1})h(sg^{-1})\sigma_s(m)\, ds
=\rho_{g^{-1}}(h)*m,$$
where we define $\rho_{g}(h)(s):=\Delta(g)h(sg)$ for $g\in G$ and $h\in L^1(G)$.  Similarly, we have $h*\tau_g(f)=\rho_{g^{-1}}(h)*f$ for  $g\in G$, $h\in L^1(G)$ and $f\in L^\infty(G)$. 

Now, using the fact that $h_i*m\to m$ (resp. $h_i*f\to f$) ultraweakly for all $m\in Z(M)$ (resp. $f\in L^\infty(G)$) and in norm for $m\in Z(M)_c$ (resp. $f\in C_{ub}(G)$), we get
for all $g\in G$ and $f\in L^\infty(G)$:
\begin{align*}
\tilde\Phi(\tau_g(f))&= \lim_i\Phi(h_i*\tau_g(f))=\lim_i\Phi(\rho_{g^{-1}}(h_i)*f)\\
&=\lim_i\lim_j\Phi(h_j*\rho_{g^{-1}}(h_i)*f)=\lim_i\lim_j\Phi(\rho_{g^{-1}}(h_i)*h_j*f)\\
&=\lim_i\lim_j\rho_{g^{-1}}(h_i)*\Phi(h_j*f)=\lim_i\lim_jh_i*\sigma_g(\Phi(h_j*f))\\
&=\lim_ih_i*\sigma_g(\tilde\Phi(f))=\sigma_g(\tilde\Phi(f)),
\end{align*}
which completes the proof.
\end{proof}

\begin{remark}\label{no equiv}
In \cite{Anantharaman-Delaroche:2002ij}*{Theorem 7.2}, Anantharaman-Delaroche shows (amongst other things) that if a group admits a topologically amenable action on a compact space, then it is exact.  Applying this to the spectrum of the $C^*$-algebra $Z(M)_c$, we see that condition \eqref{amen 1} of Proposition \ref{prop:Amenability-conditions} implies exactness of $G$.

On the other hand, condition \eqref{amen 2} is automatic for $M=L^\infty(G)$, so \eqref{amen 2} does not imply \eqref{amen 1} in general.
This follows from the  fact that the translation action of $G$ on itself is topologically amenable (e.g., see \cite{Anantharaman-Delaroche:2002ij}*{Examples 2.7 (3)}), hence $\tau:G\to\Aut(C_0(G))$ is strongly amenable by Proposition \ref{rem-amenable} and hence 
amenable by Remark \ref{rem-SAtoA}. But then \eqref{amen 2} follows from the fact that $C_0(G)_\tau''=L^\infty(G)$ (see Remark \ref{env iso 0}).

We also remark that \eqref{amen 7} is automatic for $M=\contub(G)_\tau''$: one can just consider the canonical inclusion $C_{ub}(G)\to C_{ub}(G)_\tau''$.  
On the other hand, if  $M=\contub(G)_\tau''$ satisfies condition \eqref{amen 5}, then the action on $\contub(G)$ is amenable by Theorem~\ref{Bearden-Crann}, and therefore the action of $G$ on the spectrum of $\contub(G)$ is topologically amenable by \cite{Bearden-Crann}*{Corollary 4.14} and Proposition \ref{rem-amenable}. Hence $G$ is exact by \cite{Anantharaman-Delaroche:2002ij}*{Theorem 7.2} again, and we see that \eqref{amen 7} $\Rightarrow$ \eqref{amen 5} cannot hold in general.\footnote{
We are grateful to Jason Crann for pointing this out to us.}.

To summarize, the conditions in Proposition \ref{prop:Amenability-conditions} satisfy
$$
\eqref{amen 1} \Rightarrow\eqref{amen 2} \Leftrightarrow \eqref{amen 3}\Leftrightarrow\eqref{amen 4}\Leftrightarrow \eqref{amen 5}\Rightarrow\eqref{amen 6}\Leftrightarrow\eqref{amen 7}
$$
 in general, and all are equivalent when $G$ is exact. But neither of the one-way implications are reversible when $G$ is not exact.
\end{remark}

As our main interest is in $G$-$C^*$-algebras, it is convenient to record the consequences of Proposition \ref{prop:Amenability-conditions} above and Theorem \ref{Bearden-Crann} of Bearden and Crann in this case.  

\begin{corollary}\label{cor-amenable}
For a $G$-$C^*$-algebra $(A,\alpha)$ consider the following statements:
\begin{enumerate}
\item \label{cor-amenable 1} the induced action $\alpha'':G\to \Aut(Z(A_\alpha'')_c)$ is strongly amenable;
\item \label{cor-amenable 2} $\alpha$ is amenable;
\item \label{cor-amenable 3} $\alpha$ is von Neumann amenable;
\item \label{cor-amenable 4} there is a ucp $G$-map $L^\infty(G)\to Z(A_\alpha'')$;
\item \label{cor-amenable 5} there is a ucp $G$-map $\contub(G)\to Z(A_\alpha'')_c$.
\end{enumerate}
Then 
$$\eqref{cor-amenable 1} \Rightarrow \eqref{cor-amenable 2} \Leftrightarrow \eqref{cor-amenable 3} \Rightarrow \eqref{cor-amenable 4} \Leftrightarrow \eqref{cor-amenable 5} 
$$
and all of these conditions are equivalent if $G$ is exact.
\end{corollary}
\begin{proof} Apply Proposition \ref{prop:Amenability-conditions} and Theorem \ref{Bearden-Crann} to  $M=A_\alpha''$.
\end{proof}

\section{Permanence properties of amenability}\label{sec:amen pp}

In this section, we record some permanence properties of amenable actions.

\begin{proposition}\label{pro:Morita}
Amenability is preserved by Morita equivalence: precisely, if $(A,\alpha)$ and $(B,\beta)$ are Morita equivalent $G$-$C^*$-algebras, then $A$ is amenable if and only if $B$ is amenable.
\end{proposition}
\begin{proof}
If $(A,\alpha)$ and $(B,\beta)$ are Morita equivalent, then so are the corresponding $G$-von Neumann algebras $A_\alpha''$ and $B_\beta''$ by (the proof of) Lemma \ref{lem:Morita}.  This implies that their centres $Z(A_\alpha'')$ and $Z(B_\beta'')$ are isomorphic as $G$-von Neumann algebras (see \cite{Abadie:2019kc}*{Proposition~4.6}) and this clearly implies the result. 
\end{proof}

\begin{lemma}\label{lem:funct-strong-amenability}
Let $(A,\alpha)$ and $(B,\beta)$ be $G$-$C^*$-algebras and let $\Phi\colon A\to \M(B)$ be a nondegenerate $G$-equivariant $*$-homomorphism. If the  normal extension $\Phi''\colon A_\alpha''\to B_\beta''$ maps $Z(A_\alpha'')$ into $Z(B_\beta'')$ and if $(A,\alpha)$ is amenable, then so is $(B,\beta)$.
Similarly, if the strictly continuous extension $\bar\Phi\colon \M(A)\to \M(B)$ maps $Z\M(A)$ into $Z\M(B)$, then strong amenability passes from $(A,\alpha)$ to $(B,\beta)$. 
\end{lemma}
\begin{proof}
Since $\Phi$  is nondegenerate, $\Phi''\colon A_\alpha''\to B_\beta''$ is normal and unital. 
Thus, if $(\theta_i)$ is a net  of positive type functions implementing amenability of $(A,\alpha)$, then  $(\Phi''\circ\theta_i)$ implements amenability of $(B,\beta)$.
A similar argument works for strong amenability. 
\end{proof}

\begin{remark}
The above  lemma is not true in general without the assumptions that $Z(A_\alpha'')$ (respectively, $Z\M(A)$) is mapped to $Z(B_\beta'')$ (respectively, $Z\M(B)$). 
To see  counterexamples, let $G$ be a non-amenable group acting amenably on a non-zero $C^*$-algebra $A$ and let $B=A\rtimes_\red G$ 
equipped with the action $\beta=\Ad i_G$, where $i_G:G\to \U\M(A\rtimes_\red G)$ denotes the canonical homomorphism.    
As the inner action $\beta$ on $A\rtimes_\red G$ induces the trivial action on $Z((A\rtimes_\red G)_\beta'')$, we see that $\beta$ is amenable if and only if $G$ is amenable.  On the other hand, the canonical map $i_A:A\to \M(A\rtimes_\red G)$ is a nondegenerate $G$-equivariant $*$-homomorphism (which has image in $B$ if 
$G$ is discrete). We refer to  \cite{Suzuki:2018qo}  for more involved examples where $A$ and $B$  are both simple and unital.
\end{remark}

\begin{proposition}\label{prop-ext-amenable}
Let  $\alpha:G\to\Aut(A)$ be  a continuous action and let $I\subseteq A$ be a $G$-invariant ideal. 
Then $\alpha$ is amenable if and only if the induced actions $\alpha^I$ and $\alpha^{A/I}$ on $I$ and $A/I$, respectively, are 
both amenable.
\end{proposition}
\begin{proof} The decomposition $A_\alpha''=I_{\alpha^I}''\oplus (A/I)_{\alpha^{A/I}}''$ 
of Lemma \ref{lem-decom-exact} induces a decomposition $Z(A_\alpha'')=Z(I_{\alpha^I}'')\oplus Z((A/I)_{\alpha^{A/I}}'')$.   The result follows directly from this.
\end{proof}

As a consequence of the previous results we also get permanence properties for hereditary subalgebras:

\begin{corollary}\label{her perm}
Let  $\alpha:G\to\Aut(A)$ be a continuous amenable action and let $B\subseteq A$ be a $G$-invariant hereditary $C^*$-subalgebra.  Then the restricted action on $B$ is amenable. 
\end{corollary}

\begin{proof}
The right ideal $BA$ is an imprimitivity bimodule implementing a Morita equivalence between $B=BAB$ and the ideal $I=\overline{ABA}$.  The result follows from Propositions~\ref{prop-ext-amenable} and~\ref{pro:Morita}.
\end{proof}

We shall  see in Proposition \ref{prop-inductive} below that amenability is also preserved under taking  inductive limits of amenable $G$-$C^*$-algebras.

\section{The equivalence of amenability and measurewise amenability}\label{sec:meas amen}

In this section, we establish the equivalence of amenability and measurewise amenability for actions on commutative $C^*$-algebras $A=C_0(X)$.  Measurewise amenability was introduced by Renault in \cite{Renault-LNM}*{Definition II.3.6}.  In \cite{Adams:1994wg}*{Theorem A}, Adams, Elliott, and Giordano show that the original definition is equivalent to the following.

\begin{definition}\label{defn-mwa}
Let $C_0(X)$ be a commutative $G$-$C^*$-algebra, where both $X$ and $G$ are second countable.  Then the underlying action of $G$ on $X$ is called \emph{measurewise amenable} if for every quasi-invariant Radon measure $\mu$ on $X$, the $G$-von Neumann algebra $L^\infty(X,\mu)$ is amenable. 
\end{definition}

Using  Theorem \ref{Bearden-Crann} of Bearden and Crann we obtain the following characterization of measurewise amenability. 

\begin{proposition}\label{prop-measure}
Let $C_0(X)$ be a commutative $G$-$C^*$-algebra, where both $X$ and $G$ are second countable.  Then the underlying action $G\curvearrowright X$ is measurewise amenable if and only if for every quasi-invariant Radon measure $\mu$ on $X$ there exists a net of compactly supported, positive type, norm-continuous functions $\theta_i:G\to L^\infty(X,\mu)$ with $\theta_i(e)\leq 1$ for all $i$ 
and such that $\theta_i(g)\to 1$ ultraweakly  
and uniformly on compact subsets of $G$. \qed
\end{proposition}

Our main goal in the rest of this section is the following

\begin{theorem}\label{thm-amenable-all}
Let $(C_0(X),\alpha)$ be a commutative $G$-$C^*$-algebra, and assume that $X$ and $G$ are second countable.  Then the following are equivalent:
\begin{enumerate}
\item \label{thm-amenable-all 1} $\alpha$ is amenable;
\item \label{thm-amenable-all 2} the underlying action $G\curvearrowright X$ is measurewise amenable.
\end{enumerate}
\end{theorem}

For the proof, we will need a technical lemma.  For the statement, let us say that a covariant representation $(\pi,u)$ of a $G$-$C^*$-algebra $A$ is \emph{cyclic} if the integrated form $\pi\rtimes u$ is cyclic as a representation of $A\rtimes_{\max} G$.

\begin{lemma}\label{prop:amenability-representations}
Let $G$ be a  locally compact group, and let $(A,\alpha)$ be a  $G$-$C^*$-algebra.  Then the following are equivalent:
\begin{enumerate}[(1)]
\item \label{prop:amenability-representations 1} $\alpha$ is  amenable;
\item \label{prop:amenability-representations 2} for every nondegenerate covariant representation $(\pi, u)$ of $(A,G,\alpha)$ the action 
$\Ad u:G\to \Aut(\pi(A)'')$ is 
amenable;
\item \label{prop:amenability-representations 3} for every cyclic covariant representation $(\pi,u)$ of $(A,G,\alpha)$ the action 
$\Ad u:G\to \Aut(\pi(A)'')$ is amenable. 
\end{enumerate}
\end{lemma}

\begin{proof} The implications \eqref{prop:amenability-representations 1} $\Rightarrow$ \eqref{prop:amenability-representations 2} $\Rightarrow$ \eqref{prop:amenability-representations 3} are straightforward, so we only need to check \eqref{prop:amenability-representations 3} $\Rightarrow$ \eqref{prop:amenability-representations 1}.  For this, using Theorem \ref{Bearden-Crann}, it suffices to check condition \eqref{amen 2} from Proposition \ref{prop:Amenability-conditions}, i.e.\ to show that for any finite set $\{\phi_1,...,\phi_n\}$ of states from the predual of $A_\alpha''$, any $\epsilon\in (0,1)$, and any compact subset $K$ of $G$ there exists a continuous compactly supported positive type function $\theta:G\to Z(A_\alpha'')$ such that $\|\theta(e)\|\leq 1$, and such that 
\begin{equation}\label{phii est}
|\phi_i(\theta(g))-1|<\epsilon  \text{ for all }i\in \{1,...,n\} \text{ and } g\in K.
\end{equation}
Define $\psi:=\frac{1}{n}\sum_{i=1}^n \phi_i$, which is also a state in the predual of $A_\alpha''$.  As $A_\alpha''$ is a (unital) von Neumann subalgebra of $(A\rtimes_{\max} G)^{**}$, we may extend $\psi$ to a normal state, say $\widetilde{\psi}$, on $(A\rtimes_{\max} G)^{**}$ using \cite{Blackadar:2006eq}*{Corollary III.2.1.10}.  Considering $\widetilde{\psi}$ as a state on $A\rtimes_{\max}G$, let $(\pi,u)$ be its GNS representation, considered as a covariant pair for $(A,G)$ with associated cyclic vector $\xi$. 

Now, as the representation $(\pi,u)$ is cyclic, assumption \eqref{prop:amenability-representations 3} gives a continuous compactly supported positive type function $\theta_0:G \to Z(\pi(A)'')$ such that $\|\theta_0(e)\|\leq 1$, and such that if $\psi_0$ is the (normal) vector state on $\pi(A)''$ associated to $\xi$, then  $|\psi_0(\theta_0(g))-1|<\epsilon^2/2n$ for all $g\in K$.  Let $d_\pi\in Z(A_\alpha'')$ be the central cover of $(\pi,u)$ as in Definition \ref{def:ccover}, so there is a canonical equivariant normal isomorphism $\pi(A)''\cong d_\pi A_\alpha''$.  This restricts to an equivariant isomorphism of centres $Z(\pi(A)'')\cong Z(d_\pi A_\alpha'')$.  Let $\iota:Z(\pi(A)'') \into Z(A_\alpha'')$ be the composition of this isomorphism and the canonical inclusion $Z(d_\pi A_\alpha'')\to Z(A_\alpha'')$. We claim that $\theta:=\iota\circ \theta_0$ has the right properties.  Indeed, it is clearly positive type, satisfies $\|\theta(e)\|\leq 1$ and is compactly supported, so it remains to show the condition in line \eqref{phii est}.

For this, we note that $\psi((1-d_\pi)A_\alpha'')=0$, whence $\psi_0\circ \theta_0=\psi\circ \theta$, and so for any $g\in K$, 
$$
\frac{\epsilon^2}{2n}>|\psi_0(\theta_0(g))-1|=|\psi(\theta(g))-1|=\Bigg|\frac{1}{n}\sum_{i=1}^n \phi_i(\theta(g))-1\Bigg|
$$
whence 
$$
\frac{\epsilon^2}{2}>\Bigg|\sum_{i=1}^n (\phi_i(\theta(g))-1)\Bigg|.
$$
For each $i\in \{1,...,n\}$, write $x_i$ and $y_i$ for the real and imaginary parts, respectively,  of $\phi_i(\theta(g))$, so taking real parts of this inequality gives 
\begin{equation}\label{phii est 2}
-\frac{\epsilon^2}{2} <\sum_{i=1}^n (x_i-1)<\frac{\epsilon^2}{2}.
\end{equation}
Note that, as $\theta$ is positive type and $\|\theta(e)\|\leq 1$, we have $\|\theta(g)\|\leq 1$ for all $g\in G$, whence $x_i+\sqrt{-1}y_i$ is in the unit ball of $\C$ for all $i\in \{1,...,n\}$.  Hence line \eqref{phii est 2} implies that $1-\epsilon^2/2 <x_i\leq 1$ for each $i$.  As $x_i^2+y_i^2\leq 1$, this implies that $y_i^2\leq \epsilon^2-\epsilon^4/4$ for each $i$.  Hence for each $i$ and each $g\in K$,
$$
|\phi_i(\theta(g))-1|^2= (x_i-1)^2+y_i^2 \leq \frac{\epsilon^4}{4}+\epsilon^2-\frac{\epsilon^4}{4}=\epsilon^2
$$
and we are done.
\end{proof}

\begin{proof}[Proof of Theorem \ref{thm-amenable-all}]
For the implication \eqref{thm-amenable-all 1} $\Rightarrow$ \eqref{thm-amenable-all 2}, assume that $C_0(X)$ is amenable.  Note that if  $\mu$ is a quasi-invariant Borel measure on $X$ we can construct a covariant representation $(M_\mu, u)$ of $(C_0(X), G, \alpha)$ on $L^2(X,\mu)$ 
given by
$$\big(M_\mu(f)\xi\big)(x):=f(x) \xi(x)\quad \text{and}\quad \big(u_g\xi\big)(x):=  \Big(\frac{d\mu(g^{-1} x)}{d\mu(x)}\Big)^{1/2}\xi(g^{-1}x),$$
where  $\frac{d\mu(g^{-1} x)}{d\mu(x)}$ denotes the Radon-Nikodym derivative. 

Now, $M_\mu$ gives an equivariant $*$-homomorphism $C_0(X)\to L^\infty(X,\mu)$ with ultraweakly dense image, whence Corollary \ref{cor-vN} gives a canonical, equivariant, surjective extension $M_\mu'':C_0(X)_\alpha''\to L^\infty(X,\mu)$ to the enveloping $G$-von Neumann algebra $C_0(X)_\alpha''$ of $C_0(X)$.  Let $(\theta_i:G\to C_0(X)_\alpha'')$ be a net satisfying the conditions as in \eqref{amen 2} of Proposition \ref{prop:Amenability-conditions}.  Then $(M_\mu''\circ \theta_i)$ also satisfies the same conditions with respect to $L^\infty(X,\mu)$.  Hence by Proposition \ref{prop-measure}, $L^\infty(X,\mu)$ is an amenable $G$-von Neumann algebra.

To see \eqref{thm-amenable-all 2} $\Rightarrow$ \eqref{thm-amenable-all 1}, note that by Lemma~\ref{prop:amenability-representations}, it is enough to show that for every cyclic covariant representation  $(\pi,u)\colon (C_0(X),G)\to \Bd(H_\pi)$ the 
action $\Ad u:G\to \Aut(\pi(C_0(X))'')$ is amenable. Since $C_0(X)\rtimes_\max G$ is separable, it follows that the Hilbert space 
$H_\pi$ is separable as well. 
It follows then from Renault's disintegration theorem \cite{Renault}*{Th\'eor\`eme 4.1}, that there exists a 
quasi-invariant measure $\mu$ on $X$ such that 
$\pi$ has a direct integral 
decomposition $\int_X^\oplus \pi_x \dd \mu(x)$ with respect to a measurable field of Hilbert spaces over $(X,\mu)$, and that $G$ acts compatibly on this field.
This implies that there exists an ultraweakly continuous, unital, equivariant, $*$-homomorphism $\Phi_\pi:L^{\infty}(X,\mu)\to Z(\pi(C_0(X))'')$.
Composing $\Phi_\pi$ with a net $(\theta_i)$  of positive type functions in $C_c(G, L^\infty(X,\mu))$ with the properties in Proposition \ref{prop-measure}, we obtain a net which establishes 
amenability of $(\pi(C_0(X))'',\Ad u)$.
\end{proof}

In \cite{Bearden-Crann}*{Corollary 4.14}, Bearden and Crann show that amenability for a commutative $G$\nb-$C^*$\nb-algebra $C_0(X)$ is equivalent to strong amenability.  As strong amenability of $C_0(X)$ is the same as topological amenability of $G\curvearrowright X$ (see Proposition \ref{rem-amenable}), we thus get the following corollary.  As noted in the introduction to this chapter, this solves a long-standing open question.

\begin{corollary}\label{cor-commutative}
Suppose that $G\curvearrowright X$ is a continuous action of a second countable locally compact group on a second countable locally compact space $X$. Then the following are equivalent:
\begin{enumerate}
\item $G\curvearrowright X$ is measurewise amenable;
\item $G\curvearrowright X$ is topologically amenable. \qed
\end{enumerate}
\end{corollary}

\chapter{The quasi-central approximation property}\label{chap:QAP}

In earlier work \cite{Buss:2019}*{Section 3}, we introduced the quasi-central approximation property, or (QAP), for an action 
$\alpha:G\to \Aut(A)$ of a discrete group $G$ on a $C^*$-algebra $A$.  We showed there that the (QAP) implies amenability of the action in the sense of Anantharaman-Delaroche \cite{Anantharaman-Delaroche:1987os}*{D\'{e}finition 4.1} (equivalently, in the sense of Definition \ref{def-amenable (A)}), but left the converse open. Since we introduced it, the (QAP) played an important role in recent work of Suzuki \cite{Suzuki:2020} on the classification of $G$-$C^*$-algebras; this (and the naturalness of the question) motivated us to revisit the converse in this chapter.

In Section \ref{sec:wQAP} we show that amenability of an  action of a locally compact group $G$ is always equivalent to a weak version (wQAP) of the (QAP). This  provides a description of  amenability in terms of $A$-valued positive type functions which avoids the use of the enveloping $G$-von Neumann algebra $A_\alpha''$. As a consequence, we  show that amenability passes to inductive limits of $G$-$C^*$-algebras. 

In the meantime Ozawa and Suzuki showed in \cite{Ozawa-Suzuki}*{Theorem 3.2}  that amenability (and hence the (wQAP)) and the (QAP) are  equivalent and, in \cite{Ozawa-Suzuki}*{Theorem 2.13},  that the (QAP) is also equivalent to the approximation property (AP), which was introduced by Ruy Exel \cite{Exel:Amenability}*{Definition 4.4} in the context of Fell bundles over discrete groups, and by Exel and Chi-Keung Ng \cite{ExelNg:ApproximationProperty}*{Definition 3.6} for Fell bundles over general locally compact groups.

In Section \ref{sec-alternative}, we  deduce consequences for $C^*$-algebras associated to Fell bundles.   
The equivalence of amenability and the (AP) for group actions on $C^*$-algebras gives a hint that Exel's definition is the ``correct'' extension of the notion of amenability to the setting of Fell bundles.  In  this context, we use the connection between amenability and the (AP)   to solve a conjecture of Ara, Exel, and Katsura relating the (AP) and nuclearity of cross-sectional $C^*$-algebras.

\section{The weak quasi-central approximation property (wQAP)}\label{sec:wQAP}

In this section we want to give some characterizations of amenability in terms of $A$-valued positive type functions on $G$.  This is preferable for certain purposes, as the properties use $A$ itself, rather than the much larger enveloping $G$-von Neumann algebra $A_\alpha''$ of the action.

Recall from Section \ref{sec:predual} above that Ikunishi \cite{Ikunishi}*{Theorem 1} showed that the predual of $A_\alpha''$ canonically identifies with the closed subspace $A^{*,c}$  of all $\phi\in A^*$ such that the map
$$G\to A^*; \quad g\mapsto \alpha_g^*(\phi)$$  
is  norm continuous.  Every such functional can be written as a linear combination of at most four states  
in $S(A)^c:=S(A)\cap A^{*,c}$, which then identifies with the set of normal states on $A_\alpha''$. 

In what follows,  if $\mathcal H$ is a Hilbert $A$-module (such as $L^2(G,A)$ from Definition \ref{l2ga act}) we shall write $\|\xi\|_2:=\|\braket{\xi}{\xi}_A\|^{\frac12}$ for all $\xi\in \mathcal H$. Moreover, if $\phi$ is a positive linear functional on $A$, then we shall write
\begin{equation}\label{eq}
\braket{\xi}{\eta}_{\phi}:={\phi}(\braket{\xi}{\eta}_A)\quad\text{and}\quad \|\xi\|_{\phi}:=\braket{\xi}{\xi}_{\phi}^{\frac12}
\end{equation}
for all $\xi,\eta\in \mathcal H$. Note that $\braket{\cdot}{\cdot}_{\phi}$ is a semi-inner product on $\mathcal H$ and therefore satisfies the 
Cauchy-Schwarz inequality
\begin{equation}\label{CS}
|\braket{\xi}{\eta}_{\phi}|\leq \|\xi\|_{\phi}\|\eta\|_{\phi}
\end{equation}
for all $\xi,\eta\in \mathcal H$, whence in particular each $\|\cdot \|_{\phi}$ is a seminorm.  Following Ananthara\-man-Delaroche \cite{Anantharaman-Delaroche:1987os}*{Section 1}, we write $\tau_s$ for the topology on $\mathcal H$ defined by this family of semi-norms.

In what follows  we  equip $L^2(G,A)$ with the $G$-action $\lambda^\alpha$ as in Definition \ref{l2ga act} above, and the left and right $A$-module actions determined respectively by  
$$
(a\xi)(g):=a\xi(g) \quad \text{and}\quad (\xi a)(g):=\xi(g)a
$$ 
for all $\xi\in C_c(G,A)$.

\begin{definition}\label{def-wQAP}
An action  $\alpha:G\to \Aut(A)$ of a locally compact group $G$ on  a $C^*$-algebra $A$ satisfies  the 
{\em weak quasi-central approximation property} (wQAP) if there exists a net $(\xi_i)_{i\in I}$ of functions 
$\xi_i\in C_c(G,A)\subseteq L^2(G,A)$ 
such that
\begin{enumerate}
\item \label{def-wQAP 1} $\|\xi_i\|_2\leq 1$ for all $i\in I$;
\item \label{def-wQAP 2} for all ${\phi}\in S(A)^c$ we have $ \braket{\xi_i}{\lambda_g^\alpha\xi_i}_{\phi}\to 1$ uniformly  on compact subsets of $G$;
\item \label{def-wQAP 3} for all ${\phi}\in S(A)^c$ and all $a\in A$ we have $\|\xi_i a-a\xi_i\|_{\phi}\to 0$.
\end{enumerate}
\end{definition}

\begin{remark}\label{rem-QAP}
The (wQAP) is a variant of the {\em quasi-central approximation property} (QAP) as introduced for actions of discrete groups in \cite{Buss:2019}*{Section 3}.  The (QAP) has a natural extension to locally compact groups where conditions \eqref{def-wQAP 2} and \eqref{def-wQAP 3} in Definition \ref{def-wQAP} are replaced by the conditions
\begin{enumerate}
\item[(2$'$)] $\braket{\xi_i}{\lambda_g^\alpha\xi_i}_A\to 1$ in the strict topology of $\M(A)$ uniformly on compact subsets of $G$;
\item[(3$'$)] $\|\xi_i a-a\xi_i\|_2\to 0$  for all $a\in A$.
\end{enumerate}
\end{remark}

Our main goal in this section is to show that the (wQAP), and some related properties, are equivalent to amenability.  While this paper was under review, Ozawa and Suzuki established that the (QAP) is  also equivalent to amenability 
(hence  the (QAP) and (wQAP) are always equivalent by Theorem~\ref{thm-wQAP} below).\footnote{In a previous version of this paper we had a proof of this fact for discrete groups $G$, but, as was pointed out to us by \mcomment{one of the referees}, there was a mistake in our proof.}

Another interesting approximation property is the Exel-Ng approximation property (AP).  This was introduced in \cite{ExelNg:ApproximationProperty}*{Definition 3.6} in the setting of 
Fell bundles over locally compact groups. In the special case of actions $\alpha:G\to \Aut(A)$ this 
translates into

\begin{definition}\label{def-wAP}
An action $\alpha:G\to \Aut(A)$ satisfies the {\em approximation property}
(AP)  if there exists a bounded net $(\xi_i)$ in $C_c(G,A)\subseteq L^2(G,A)$ such that for all $a\in A$
$$\braket{\xi_i}{a\lambda_g^\alpha\xi_i}_A \to a$$
in norm, uniformly on compact subsets of $G$. 

An action $\alpha:G\to \Aut(A)$ satisfies the {\em weak approximation property} (wAP) if there exists a $\|\cdot\|_2$-bounded net
$(\xi_i)$ in $C_c(G,A)$
such that for all $\phi\in S(A)^c$ and for all $a\in A$
$$\phi\big(\braket{\xi_i}{a\lambda_g^\alpha\xi_i}_A - a\big)\to 0$$ 
uniformly on compact subsets of $G$.
\end{definition}

It is clear that the (AP) implies the (wAP). The main result of this section is the following

\begin{theorem}\label{thm-wQAP}
Let $\alpha:G\to\Aut(A)$ be an action of the locally compact group $G$ on the $C^*$-algebra $A$. Then the following 
are equivalent:
\begin{enumerate}
\item \label{wqap 1}$\alpha$ is amenable.
\item \label{wqap 2}$\alpha$ satisfies the (wQAP).
\item \label{wqap 3}$\alpha$ satisfies the (wAP).
\end{enumerate}
\end{theorem}

The following lemma is well known to experts: it is a special case 
of a version of Kaplansky's density theorem for Hilbert modules due to 
Zettl (see \cite{Zettl}*{Lemma 2.4 and Corollary 2.7}). For completeness, we give an elementary direct proof below:

\begin{lemma}\label{lem-dense}
Suppose that $M$ is a von Neumann algebra and $A\subseteq M$ is an ultraweakly dense $C^*$-subalgebra.  Then, if $H$ is a Hilbert space, the unit ball  of the Hilbert $A$-module $H\otimes A$
 is dense in the unit ball of the Hilbert $M$-module $H\otimes M$ with respect to the locally convex 
topology $\tau_s$ generated by the seminorms $\xi\mapsto \|\xi\|_\phi={\phi}(\braket{ \xi}{\xi}_M)^{1/2}$, where ${\phi}$ runs through the set of normal states of $M$.
\end{lemma}
\begin{proof} 
Let us assume without loss of generality that $A$ is represented faithfully and nondegenerately on the Hilbert space $K$ 
such that $M=A''\subseteq \Bd(K)$.  We  represent the linking algebra 
$L(H\otimes A):=\left(\begin{matrix} \K(H)\otimes A& H\otimes A\\H^*\otimes A& A\end{matrix}\right)$ faithfully on the 
Hilbert-space direct sum  $(H\otimes K)\oplus K$ via 
$$\left(\begin{matrix} k\otimes a& \xi\otimes b\\ \eta^*\otimes c& d\end{matrix}\right)\left(\begin{matrix} \zeta\otimes v\\ w\end{matrix}\right)=\left(\begin{matrix} k\zeta\otimes av +\xi\otimes bw\\ \braket{ \eta}{\zeta}_{\C} cv+dw\end{matrix}\right).$$
Let $p$ and $q=1-p$ denote the orthogonal projections from $(H\otimes K)\oplus K$ onto $H\otimes K$ and $K$, respectively. 
The  isometric inclusion of $H\otimes A$ into the upper right corner $\Bd(K, H\otimes K)=p\Bd((H\otimes K)\oplus K)q$  extends  to an isometric inclusion of $H\otimes M=H\otimes A''$.  Moreover,  $H\otimes M$ lies in the strong closure of $H\otimes A$ inside $\Bd(K, H\otimes K)$, as if $\xi\otimes m$ is any elementary tensor in $H\otimes M$ and $(a_i)$ is a net in $A$ which 
converges strongly to $m$, then for every vector $w\in K$ we get
$$(\xi\otimes a_i)w =\xi\otimes a_iw\to \xi\otimes mw=(\xi\otimes m)w.$$
It follows then from the Kaplansky density theorem applied to $L(H\otimes A)\subseteq \Bd((H\otimes K)\oplus K)$, that the unit ball of  $H\otimes A$ is strongly dense in the unit ball of $H\otimes M$. 

Let then $\xi$ be an element of the unit ball of $H\otimes M$ and let $(\xi_i)$ be a net in the unit ball of $H\otimes A$ such that $\xi_i\to \xi$ strongly; we claim that $\xi_i\to \xi$ in the $\tau_s$ topology, which will suffice to complete the proof.  We have $\xi_iv\to\xi v$ in $K$ for all $v\in K$, whence for all $v,w\in K$ we have
$$\braket{ \braket{\xi_i-\xi}{\xi_i-\xi}_M v}{ w}_{\C}=\braket{ (\xi_i-\xi)^*(\xi_i-\xi)v}{ w}_{\C}=\braket{ (\xi_i-\xi)v}{ (\xi_i-\xi)w}_{\C}\to 0$$
as $i\to\infty$.  Hence the bounded net $(\braket{\xi_i-\xi}{\xi_i-\xi}_M)_i$ weakly converges to $0$.
Since the ultraweak topology coincides with the weak topology on bounded sets, it follows that
$\|\xi_i-\xi\|_{\phi}\to 0$ for every normal state ${\phi}$ of $M$.
Thus $\xi_i\to \xi$ in the topology $\tau_s$.
\end{proof}

\begin{corollary}\label{cor-dense}
Let $\alpha:G\to \Aut(A)$ be an action  and  $V\subseteq G$ an open subset of $G$.
Then for every $\xi\in L^2(V,A_\alpha'')$ with $\|\xi\|_2\leq 1$ there exists a net $(\xi_i)\in C_c(V,A)$ such  that 
$\|\xi_i\|_2\leq 1$ for all $i\in I$ and $\xi_i\to \xi$ in the topology $\tau_s$ generated 
by the seminorms  $\xi\mapsto\|\xi\|_{\phi}$, ${\phi}\in S(A)^c$.
\end{corollary}
\begin{proof}
This follows from Lemma \ref{lem-dense} with $H=L^2(V)$ (with respect to the Haar measure of $G$ restricted to $V$) together with the fact that 
the unit ball of $C_c(V,A)$ is norm-dense in the unit ball of $L^2(V,A)$.
\end{proof}

For what follows we also need the following well-known fact.  We leave the straightforward proof to the reader.

\begin{lemma}\label{lem-compact}
Suppose that $X$ is a normed space, $({\phi}_i)$ is a norm-bounded 
net in $X^*$, and ${\phi}\in X^*$. Then the following are equivalent:
\begin{enumerate}
\item $({\phi}_i)$ converges to ${\phi}$ in the weak*-topology.
\item $({\phi}_i)$ converges to ${\phi}$ uniformly on every compact subset $K$  of $X$. \qed
\end{enumerate}
\end{lemma}

We are now ready for the

\begin{proof}[Proof of Theorem \ref{thm-wQAP}]
For \eqref{wqap 1} $\Rightarrow$ \eqref{wqap 2} assume  that $\alpha$ is amenable and let
 $(\xi_i)$ be a net  in  $C_c(G, Z(A_\alpha'')_c)$   with $\braket{\xi_i}{\xi_i}_{A_\alpha''}\leq 1$ such that
\begin{equation}\label{i lim}
\braket{\xi_i}{\lambda_g^{\alpha''}\xi_i}_{\phi}\to 1 \quad \text{as}\quad i\to\infty
\end{equation}
for  all ${\phi}\in S(A)^c$ uniformly on compact subsets  $C\subseteq G$.
 For each $i$, let $V_i\subseteq G$ be a relatively compact open subset of $G$ such that $\supp\xi_i\subseteq V_i$. 
It  follows from Corollary \ref{cor-dense} 
that for each $\xi_i$ we can find a net $(\xi_{ij})_j$ in $C_c(V_i,A)$ such that
 $\|\xi_{ij}\|_2\leq 1$ and 
 \begin{equation}\label{j lim 0}
 \|\xi_{ij}-\xi_i\|_{\phi}\to 0 \quad \text{as} \quad  j\to\infty
 \end{equation}
 for each state ${\phi}\in S(A)^c$.  With notations as in (\ref{eq}), we compute for each ${\phi}\in S(A)^c$:
\begin{align*}
|\braket{\xi_{ij}}{\lambda_g^\alpha(\xi_{ij})}_{\phi} - \braket{\xi_{i}}{\lambda_g^{\alpha''}(\xi_{i})}_{\phi}| 
&\leq |\braket{ \xi_{ij}-\xi_i}{\lambda_g^{\alpha}(\xi_{ij})}_{\phi}| + |\braket{\xi_i}{\lambda_g^{\alpha''}(\xi_{ij}-\xi_i)}_{\phi}| \\
&\leq \|\xi_{ij}\|_2\|\xi_{ij}-\xi_i\|_{\phi} + \|\xi_i\|_2 \| \xi_{ij}-\xi_i\|_{\alpha^*_{g^{-1}}({\phi})}\\
&\leq  \|\xi_{ij}-\xi_i\|_{\phi} +  \| \xi_{ij}-\xi_i\|_{\alpha^*_{g^{-1}}({\phi})}.
\end{align*}
Using Lemma \ref{lem-compact} (applied to $A_\alpha''$ as the dual space of $A^{*,c}$) we conclude that the sum in the last line converges to 
$0$ as $j\to\infty$ uniformly for $g$ in the compact closure $C_i$ of  $V_iV_i^{-1}$ in $G$.
Since $\braket{\xi_i}{\lambda_g^{\alpha''}\xi_i}_{A_\alpha''}=0$ for all $g\notin C_i$ (and similarly with $\xi_i$ replaced by  $\xi_{ij}$), we conclude that for each $i$, 
\begin{equation}\label{j lim}
|\braket{\xi_{ij}}{\lambda_g^\alpha(\xi_{ij})}_{\phi} - \braket{\xi_{i}}{\lambda_g^{\alpha''}(\xi_{i})}_{\phi}|\to 0 \quad \text{as} \quad  j\to\infty
\end{equation}
uniformly on $g\in G$. 

Let now $F$ be a finite subset of $S(A)^c$, let $\mathcal{A}$ be a finite subset of $A$, let $C$ be a compact subset of $G$, and let $\epsilon>0$.  To complete the proof, it suffices to find $\eta\in C_c(G,A)$ such that $\|\eta\|_2\leq 1$, such that $|\braket{\eta}{\lambda_g^{\alpha}\eta}_{\phi}-1|<\epsilon$ for all $g\in C$ and $\phi$ in $F$, and such that $\|\eta a-a\eta\|_{\phi}<\epsilon$ for all $a\in \mathcal{A}$ and $\phi\in F$.  Using line \eqref{i lim}, we may find $i$ so that $\xi_i$ satisfies 
\begin{equation}\label{xii1}
|\braket{\xi_i}{\lambda_g^{\alpha''}\xi_i}_{\phi}-1|<\epsilon/2
\end{equation}
for all $\phi\in F$ and $g\in C$.  Let $P(A)^c:=P(A)\cap A^{*,c}$, where $P(A)$ denotes the set of positive functionals on $A$.  Then the limits in lines \eqref{j lim 0} and \eqref{j lim} clearly still hold for all $\phi\in P(A)^c$.  Using this and the limit in line \eqref{j lim 0} we may find $j$ such that $\xi_{ij}$ satisfies 
\begin{equation}\label{xiij 0}
\|\xi_{ij}-\xi_{i}\|_\phi<\frac{\epsilon}{\kappa}\quad\text{with}\;  \kappa:=2(1+\max_{a\in \mathcal A}\|a\|)
\end{equation}
for all $\phi$ in $F\cup \{\psi(a\cdot a^*):\psi\in F,a\in\mathcal{A}\}$, and so that 
\begin{equation}\label{xiij1}
|\braket{\xi_{ij}}{\lambda_g^\alpha(\xi_{ij})}_{\phi} - \braket{\xi_{i}}{\lambda_g^{\alpha''}(\xi_{i})}_{\phi}|<\epsilon /2
\end{equation}
for all $\phi \in F$ and $g\in G$.  We claim that $\eta:=\xi_{ij}$ works.

Indeed, $\eta$ is in $C_c(G,A)$ and $\|\eta\|_2\leq 1$ as all elements of the net $(\xi_{ij})$ have these properties.  Lines \eqref{xii1} and \eqref{xiij1} imply that $|\braket{\eta}{\lambda_g^{\alpha}\eta}_{\phi}-1|<\epsilon$ for all $g\in C$ and $\phi \in F$.  Finally, we note that for any $\phi\in F$ and any $a\in \mathcal{A}$, using that $\xi_i$ takes values in the center of $A_\alpha''$ we have
\begin{align*}
\| \eta a-a\eta\|_{\phi}&=\| \eta a-\xi_{i} a+a \xi_{i} - a\eta\|_{\phi}\\
&\leq \| (\eta-\xi_i)a\|_{\phi}+\| a(\eta-\xi_{i})\|_{\phi}\\
&\leq \phi(a^*\braket{\eta-\xi_{i}}{\eta-\xi_{i}}_{A_\alpha''}a)^{1/2}+\phi(\|a\|^2\braket{\eta-\xi_i}{\eta-\xi_{i}}_{A_\alpha''})^{1/2} \\
&=\|\eta-\xi_{i}\|_{{\phi}(a^*\cdot a)}+\|a\|\|\eta-\xi_{i}\|_{\phi}.
\end{align*}
This is less than $\epsilon$ by line \eqref{xiij 0}, so we are done with \eqref{wqap 1} $\Rightarrow$ \eqref{wqap 2}.

For the proof of   \eqref{wqap 2} $\Rightarrow$ \eqref{wqap 3} let 
 $(\xi_i)$ be a net in $C_c(G,A)$ as in Definition \ref{def-wQAP}. Then, for all $a\in A$ and ${\phi}\in S(A)^c$ we get:
\begin{align*}
 \big|{\phi}(\braket{\xi_i}{ a\lambda_g^{\alpha}\xi_i}_A-a\braket{ \xi_i}{ \lambda_g^{\alpha}\xi_i}_A)\big|
&=\big|\braket{ a^*\xi_i}{\lambda_g^{\alpha}\xi_i}_{\phi}-\braket{\xi_ia^*}{\lambda_g^{\alpha}\xi_i}_{\phi}\big|\\
&=\big|\braket{a^*\xi_i- \xi_i a^*}{ \lambda_g^{\alpha}\xi_i}_{\phi}\big|\\
&\leq \| a^*\xi_i- \xi_i a^*\|_{\phi}\|\lambda_g^{\alpha}\xi_i\|_2\to 0
\end{align*}
uniformly  on $G$. Since $\phi(\braket{ \xi_i}{ \lambda_g^{\alpha}\xi_i} - 1)\to 0$ uniformly on compact subsets of $G$, 
it follows from this and separate ultraweak continuity of multiplication on bounded sets that ${\phi}(\braket{\xi_i}{ a\lambda_g^{\alpha}\xi_i}_A-a)\to 0$ uniformly on compact subsets of $G$ as well.

The implication \eqref{wqap 3} $\Rightarrow$ \eqref{wqap 1} follows from Theorem \ref{Bearden-Crann}.  Indeed, our assumption \eqref{wqap 3} implies that condition \eqref{amen 4} from Proposition~\ref{prop:Amenability-conditions} holds with respect to $M=A_\alpha''$ and the canonical inclusion $A\into A_\alpha''$; hence by Theorem \ref{Bearden-Crann}, condition \eqref{amen 2} from Proposition~\ref{prop:Amenability-conditions} holds, which is the desired conclusion.
\end{proof}

We now give a useful application of Theorem \ref{thm-wQAP} towards
 inductive limits of  amenable $G$-$C^*$-algebras. 

\begin{proposition}\label{prop-inductive}
Suppose that $(A_m, G,\alpha_m)$ is a directed system of $G$-$C^*$-algebras and let $\alpha:G\to\Aut(A)$ denote the induced action on the inductive limit $A:=\lim_mA_m$. If all $\alpha_m$ are amenable, then so is $\alpha$.
\end{proposition}

\begin{proof}
We follow the ideas of \cite{Buss:2019}*{Proposition 3.3}.  Proposition \ref{prop-ext-amenable} implies that amenability passes to quotients.  Hence we may assume that the canonical maps $A_m\to A$ are $G$-embeddings, and we therefore may regard $A$ as the closed union of the $A_m$.  For each $m$, let $(\xi_{i}^{(m)})$ be a net in $C_c(G,A_m)\subseteq C_c(G,A)$ which implements the (wQAP) for $\alpha_m$, which exists by Theorem \ref{thm-wQAP}.  These nets can be used to construct a net that establishes the (wQAP) for $A$: we leave the elementary details to the reader.  Theorem \ref{thm-wQAP} now shows that $A$ is amenable.
\end{proof}

The  theorem below records the equivalences between the family of properties around the (QAP) and (AP), combining Theorem \ref{thm-wQAP}
and some of the main results of Ozawa-Suzuki \cite{Ozawa-Suzuki} for the reader's convenience.

\begin{theorem}\label{thm-weak}
Let $\alpha:G\to \Aut(A)$ be an action of a locally compact group $G$. Then the following are equivalent:
\begin{enumerate}
\item \label{thm-weak 1} $\alpha$ is amenable;
\item \label{thm-weak 2} $\alpha$ satisfies the (QAP);
\item \label{thm-weak 3} $\alpha$ satisfies the (wQAP);
\item \label{thm-weak 4} $\alpha$ satisfies the Exel-Ng (AP);
\item \label{thm-weak 5} $\alpha$ satisfies the (wAP).
\end{enumerate}
\end{theorem}

\begin{proof} 
The equivalences \eqref{thm-weak 1} $\Leftrightarrow$ \eqref{thm-weak 2} and \eqref{thm-weak 2} $\Leftrightarrow$ \eqref{thm-weak 4} are due to Ozawa and Suzuki: see \cite{Ozawa-Suzuki}*{Theorem 3.2} and \cite{Ozawa-Suzuki}*{Theorem 2.13}, respectively.  On the other hand,  \eqref{thm-weak 1}, \eqref{thm-weak 3}, and \eqref{thm-weak 5} were shown to be equivalent in Theorem \ref{thm-wQAP}.
\end{proof}

\begin{remark}\label{rem-BedCont} 
There is a variant of the approximation property due to  B\'edos and Conti (see \cite{BedCont}*{Definition 5.7}) for actions of discrete  groups $G$, which is extended to actions of locally compact $G$ by Bearden and Crann in \cite{Bearden-Crann}*{Section 4}.   Parallel to our work, Bearden and Crann showed in \cite{Bearden-Crann}*{Theorem 4.2} that amenability is equivalent 
to a variant of the B\'edos-Conti approximation property.  The B\'{e}dos-Conti property is in a similar spirit as the (AP) of Ng and Exel but it is a priori weaker than the (AP).
\end{remark}

\section{Consequences for Fell bundles}\label{sec-alternative}

As mentioned at the start of this chapter, the approximation property of Exel was originally defined in the context of Fell bundles (see for example \cite{Exel:Partial_dynamical} for background on Fell bundles).  In \cite{Ara-Exel-Katsura:Dynamical_systems}*{Remark~6.5} it is conjectured that nuclearity of the cross-sectional $C^*$-algebra $C^*_\red(\Bd)$ of a Fell bundle $\Bd$ should imply the (AP) for $\Bd$. Using Theorem \ref{thm-weak} we give a positive answer to this conjecture:

\begin{corollary}\label{cor:fb nuc}
Let $\Bd=(\Bd_g)_{g\in G}$ be a Fell bundle over the discrete group $G$ and assume that its reduced cross-sectional \cstar{}algebra $C^*_\red(\Bd)$ is nuclear. Then $\Bd$ has the (AP) of Exel.
\end{corollary}
\begin{proof}
Consider the dual $G$-action $\alpha:=\widehat{\delta}_\Bd$ on the crossed product $A:=C^*_\red(\Bd)\rtimes_{\delta_{\Bd}}G$ by the canonical (dual) coaction $\delta_\Bd$ of $G$ on the cross-sectional $C^*$\nb-algebra $C^*_\red(\Bd)$.   Using \cite{Abadie-Buss-Ferraro:Morita_Fell}*{Remark 3.3 and Theorem 3.4}, we see that $\Bd$ is weakly equivalent to $(A,\alpha)$ as Fell bundles in the sense of \cite{Abadie-Buss-Ferraro:Morita_Fell}*{Definition 2.6}.  If $C^*_\red(\Bd)$ is nuclear, then so is $A\rtimes_{\alpha,\red} G$ as it is isomorphic to $C^*_\red(\Bd)\otimes \K(\ell^2(G))$ (see \cite{Abadie-Buss-Ferraro:Morita_Fell}*{Remark 3.3} again). Using \cite{Anantharaman-Delaroche:1987os}*{Th\'eor\`eme~4.5}, this implies that the $G$-action $\alpha$ on $A$ is amenable and hence by Theorem~\ref{thm-weak}, $\alpha$ also has the (AP).  The (AP) passes through weak equivalences of Fell bundles by \cite{Abadie:2019kc}*{Corollary~5.24 and Theorem~6.12}, from which it follows that $\Bd$ has the (AP).
\end{proof}

The converse of the above corollary is known: if $\Bd$ has the (AP) and $\Bd_e$ is nuclear, then $C^*(\Bd)=C^*_\red(\Bd)$ is nuclear (see \cite{Exel:Partial_dynamical}*{Proposition~25.10}).  Note that this really is the precise converse of Corollary \ref{cor:fb nuc}, as nuclearity of $C^*_\red(\Bd)$ implies nuclearity of $\Bd_e$ due to the  existence of a conditional expectation $E:C^*_\red(\Bd)\to \Bd_e$ (see \cite{Exel:Partial_dynamical}*{Definition 19.2 and Proposition 19.3}).  

The (AP) is therefore a way to characterize nuclearity of cross-sectional $C^*$-algebras of Fell bundles, in particular, crossed products. We should remark that this problem of characterizing nuclearity of Fell bundle $C^*$-algebras or crossed products has been studied recently in \cite{MSTL:Herz-Schur,He:Herz-Schur} using the theory of Herz-Schur multipliers. In \cite{MSTL:Herz-Schur}*{Definition 4.1} a notion of `nuclearity' for a $G$-$C^*$-algebra $A$ is introduced that is shown to be equivalent to nuclearity of $A\rtimes_{\red} G$ in \cite{MSTL:Herz-Schur}*{Theorem 4.3}.  A similar `nuclearity' notion is introduced for Fell bundles in \cite{He:Herz-Schur}*{Definition 6.3}, and again this is used to characterize nuclearity of the cross-sectional $C^*$-algebra in \cite{He:Herz-Schur}*{Theorem 6.8}. Combining all this together, we can rephrase nuclearity of $C^*$-dynamical systems in terms of the ((Q)AP) as follows:

\begin{corollary}
A $G$-$C^*$-algebra $A$ is nuclear in the sense of \cite{MSTL:Herz-Schur} if and only if $A$ is a nuclear $C^*$-algebra and $\alpha$ has the (QAP).
Similarly,  a Fell bundle $\Bd$ is nuclear in the sense of \cite{He:Herz-Schur} if and only if $\Bd_e$ is a nuclear $C^*$-algebra
and  $\Bd$ has the (AP) of Exel.   \qed
\end{corollary}

\chapter{The weak containment property and commutant amenability }\label{sec:Weak-containment}

In this chapter we study the weak containment property as in the following definition.

\begin{definition}\label{def-weak-containment}
A $G$-$C^*$-algebra $(A,\alpha)$ (or just the action $\alpha$) is said to have the {\em weak containment property} (WCP) if the regular
representation $\Lambda_{(A,\alpha)}:A\rtimes_{\max}G\to A\rtimes_{\red}G$ is an isomorphism. When this is the case, to shorten the notation, we usually write $A\rtimes_\max G\cong A\rtimes_\red G$ or even $A\rtimes_\max G=A\rtimes_\red G$.
\end{definition}

For discrete groups, Anantharaman-Delaroche showed in \cite{Anantharaman-Delaroche:1987os}*{Proposition 4.8} that amenability implies the (WCP), but the converse has been a long-standing open problem.  We are interested here in elucidating the relationship between amenability and the (WCP).  

In Section \ref{sec:inj ca}, we build on our earlier work \cite{Buss:2019}*{Sections 4 and 5} for discrete groups.  We study injectivity properties of covariant representations, and use this to show that the (WCP) for $(A,G,\alpha)$ with $G$ exact is equivalent to a property we call \emph{commutant amenability}.  Roughly, $A$ is commutant amenable if the commutant $\pi(A)'$ of $A$ in any covariant representation $(\pi, u)$ (or just in the universal representation) has an amenability-like approximation property.  These results also show that (commutant) amenability always implies the (WCP), whether or not the acting group is exact.

In Section \ref{sec-commutative} we bring into play the \emph{Haagerup standard form}, which is a covariant representation of a $G$-von Neumann algebra with particularly good properties.  The results here extend work of Matsumura \cite{Matsumura:2012aa}, who first had the idea of using the Haagerup standard form to study amenability (our proofs are different to Matsumura's, however, following instead the lines of our earlier work \cite{Buss:2019}*{Section 5} on discrete groups).  The Haagerup standard form lets us translate commutant amenability into more intrinsic properties: for example, amenability of $A$ turns out to be equivalent to commutant amenability of $A\otimes_{\max} A^{\op}$.  We also show that commutant amenability is equivalent to amenability for actions on commutative $C^*$-algebras.  Thus for actions of exact locally compact groups on commutative $C^*$-algebras, amenability is indeed equivalent to the (WCP).  

In sharp contrast, in Section \ref{sec:example}, we show that the (WCP) is \emph{not} equivalent to amenability for actions on noncommutative $C^*$-algebras.  Our counterexamples are quite concrete: for example, we show that actions of $PSL(2,
\C)$ on the compact operators can satisfy the (WCP), yet be non-amenable.  We use the same examples to show that the (WCP) does not pass through the restriction of an action to a closed subgroup, which answers a question of Anantharaman-Delaroche 
\cite{Anantharaman-Delaroche:2002ij}*{Question 9.2 (b)} in the negative.

\section{Injective representations and commutant amenability}\label{sec:inj ca}

In this section, we use ideas developed by the authors in the case of discrete groups in \cite{Buss:2019}*{Section 4} to provide necessary and sufficient conditions for a $G$-$C^*$-algebra $(A,\alpha)$ to have 
the weak containment property (WCP) from Definition \ref{def-weak-containment}, at least when $G$ is exact.

Recall from \cite{Buss:2018nm} that the {\em injective crossed product functor} $(A,\alpha)\mapsto A\rtimes_{\inj}G$ is 
the largest crossed product functor which is injective in the sense that every $G$-equivariant inclusion $\varphi: A\into B$ 
from a $G$-$C^*$-algebra $A$ into a $G$-$C^*$-algebra $B$ descends to an inclusion $\varphi\rtimes_\inj G: A\rtimes_\inj G\into B\rtimes_\inj G$.
For a given $G$-$C^*$-algebra $(A,\alpha)$, the crossed product $A\rtimes_\inj G$ can be realized as the completion 
of the convolution algebra $C_c(G,A)$ with respect to the norm
$$\|f\|_{\inj}:=\inf\{\|\varphi\circ f\|_{B\rtimes_\max G}: \varphi:A\into B\;\text{ is $G$-equivariant}\}.$$
It was observed in \cite{Buss:2018nm}*{Proposition 4.2} that $A\rtimes_\inj G=A\rtimes_\red G$ whenever $G$ is exact\footnote{This can fail if $G$ is not exact: see \cite{Buss:2018nm}*{Lemma 4.7}.},
so that for exact groups the (WCP) is equivalent to $A\rtimes_\max G=A\rtimes_\inj G$, that is, to the statement that the canonical 
quotient map $q:A\rtimes_{\max}G\to A\rtimes_\inj G$ is an isomorphism.

Our first goal is to characterize when $q:A\rtimes_{\max}G\to A\rtimes_\inj G$ is an isomorphism in terms of injective representations as in the next definition.

\begin{definition}\label{def-injective-rep}
Let $(A,\alpha)$ and $(B,\beta)$ be $G$-$C^*$-algebras and let $\iota:A\into B$ be a $G$-equivariant inclusion.  Then a covariant representation 
$$
(\pi, u):(A,G)\to \Bd(H)
$$
is called {\em $G$-injective with respect to $\iota$} if the dashed arrow below 
$$
\xymatrix{ B \ar@{-->}[dr]^-\sigma & \\ A\ar[u]^-\iota \ar[r]^-\pi & \Bd(H)}
$$
can be filled in with a ccp $G$-map $\sigma$ (here $\Bd(H)$ is equipped with the $G$-action $\text{Ad}u$).

We say that $(\pi, u)$ is {\em $G$-injective} if it is $G$-injective with respect to any $G$\nb-inclusion $\iota:A\into B$ of $A$  into a $G$-$C^*$-algebra $B$.
\end{definition}

We need the following generalization of \cite{Buss:2019}*{Lemma 4.8}.

\begin{lemma}\label{int lem}
Let $A$ be a $G$-$C^*$-algebra, and $(\sigma,u):(A,G)\to \Bd(H)$ be a pair consisting of a ccp map $\sigma$ and a unitary representation $u$ satisfying the \emph{covariance relation}
$$
\sigma(\alpha_g(a))=u_g\sigma(a)u_g^*
$$
for all $a\in A$ and $g\in G$. Then the integrated form 
$$
\sigma\rtimes u :C_c(G,A)\to \Bd(H);\quad f\mapsto \int_G \sigma(f(g))u_g\,dg
$$
extends to a ccp map $\sigma\rtimes u:A\rtimes_{\max} G \to\Bd(H)$.
\end{lemma}
\begin{proof}  The ccp $G$-map $\sigma:A\to \Bd(H)$ necessarily takes image in $\Bd(H)_c$, and by \cite{Buss:2014aa}*{Theorem 4.9, (5) $\Rightarrow$ (6)} it then descends to a ccp map $\sigma\rtimes_\max G: A\rtimes_\max G\to \Bd(H)_c\rtimes_{\max}G$,
with respect to the action $\Ad u:G\to \Aut(\Bd(H)_c)$ given by conjugation with $u$. Composing $\sigma\rtimes_\max G$ with the integrated form 
$\id\rtimes u: \Bd(H)_c\rtimes_\max G\to \Bd(H)$ of the covariant pair $(\id, u)$  gives  the ccp map $\sigma\rtimes u$.
\end{proof}

The following proposition extends \cite{Buss:2019}*{Theorem 4.9}, where the same result has been shown for actions of discrete groups. 
Although the proof is almost the same as in the discrete case, we include it  for completeness.
\begin{proposition}\label{prop-inj-cov}
Let $\iota:A\into B$ be an inclusion of $G$-$C^*$-algebras $(A,\alpha)$ and $(B,\beta)$. 
 Then the following are equivalent:
\begin{enumerate}
\item \label{prop-inj-cov 1}the descent $\iota\rtimes_\max G: A\rtimes_\max G\to B\rtimes_\max G$ is injective;
\item \label{prop-inj-cov 2}every nondegenerate covariant representation $(\pi, u)$ of $(A,G,\alpha)$ on some Hilbert space $H$ is
$G$-injective with respect to the inclusion $\iota:A\into B$;
\item \label{prop-inj-cov 3}there exists a nondegenerate covariant representation $(\pi,u)$ of $(A,G,\alpha)$ which is $G$-injective with respect to
$\iota:A\into B$ and which has a faithful integrated form
$\pi\rtimes u:A\rtimes_\max G\to \Bd(H)$.
\end{enumerate}
\end{proposition}
\begin{proof}
Assume \eqref{prop-inj-cov 1}, and let $(\pi,u):(A,G)\to \Bd(H)$ be a covariant representation.  We must show that the dashed arrow below can be filled in with a ccp $G$-map
\begin{equation}\label{desideratum}
\xymatrix{ B \ar@{-->}[dr] & \\ A \ar[u]^-\iota \ar[r]^-\pi & \Bd(H)~. }
\end{equation}
Let $\widetilde{A}$ and $\widetilde{B}$ be the unitizations of $A$ and $B$ and let $\widetilde{\pi}:\widetilde{A}\to \Bd(H)$ and $\widetilde{\iota}:\widetilde{A}\to \widetilde{B}$ be the canonical (equivariant) unital extensions.  It will suffice to prove that the dashed arrow below 
$$
\xymatrix{ \widetilde{B} \ar@{-->}[dr] & \\ \widetilde{A} \ar[u]^-{\widetilde{\iota}} \ar[r]^-{\widetilde{\pi}} & \Bd(H) }
$$
can be filled in with an equivariant ccp map; indeed, if we can do this, then the restriction of the resulting equivariant ccp map $\widetilde{B}\to \Bd(H)$ to $B$ will have the desired property.

Since the descent $\iota\rtimes G:A\rtimes_{\max} G \to B\rtimes_{\max} G$ of $\iota$ is injective by assumption,  it follows from this and the commutative diagram
$$
\xymatrix{ 0\ar[r] & B\rtimes_{\max} G \ar[r] & \widetilde{B}\rtimes_{\max} G \ar[r] & \C\rtimes_{\max} G \ar[r] & 0 \\
0\ar[r] & A\rtimes_{\max} G \ar[r] \ar[u] & \widetilde{A}\rtimes_{\max} G \ar[r] \ar[u] & \C\rtimes_{\max} G \ar@{=}[u] \ar[r] & 0 }
$$
of short exact sequences that the map 
$$
\widetilde{\iota}\rtimes G:\widetilde{A}\rtimes_{\max} G\to \widetilde{B}\rtimes_{\max} G
$$
is injective as well.  From now on, to avoid cluttered notation, we will assume that $A$, $B$, $\pi$ and $\iota$ are unital, and that the map $\iota\rtimes G:A\rtimes_{\max}G\to B\rtimes_{\max}G$  is injective; our goal is to fill in the dashed arrow in line \eqref{desideratum} under these new assumptions. 

It  follows from unitality of $\iota$ that the map $\iota\rtimes G:A\rtimes_{\max} G \to B\rtimes_{\max} G$
uniquely extends to an injective $*$-homomorphism $\M(A\rtimes_{\max} G) \to \M(B\rtimes_{\max} G)$.  Hence in the diagram below
$$
\xymatrix{ \M(B\rtimes_{\max} G) \ar@{-->}[dr]^-{\widetilde{\pi\rtimes u}} & \\ \M(A\rtimes_{\max}G) \ar[r]^-{\pi \rtimes u} \ar[u]^-{\iota\rtimes G} & \Bd(H) }
$$
we may use injectivity of $\Bd(H)$ (i.e.\ Arveson's extension theorem as, for example, in \cite{Brown:2008qy}*{Theorem 1.6.1}) to show that the dashed arrow can be filled in with a ucp map.  Any element of the form $i_G^B(g)\in \M(B\rtimes_\max G)$ is in the multiplicative domain of $\widetilde{\pi\rtimes u}$, from which it follows that the restriction $\phi$ of $\widetilde{\pi\rtimes u}$ to 
$B\cong i_B(B)\subseteq \M(B\rtimes_\max G)$ is equivariant.  This restriction $\phi$ is the desired map.

The implication \eqref{prop-inj-cov 2} $\Rightarrow$ \eqref{prop-inj-cov 3} is clear, so it remains to show \eqref{prop-inj-cov 3} $\Rightarrow$ \eqref{prop-inj-cov 1}.  Let $\pi\rtimes u:A\rtimes_{\max} G \to \Bd(H)$ be a nondegenerate  faithful representation such that there is a ccp $G$-map $\phi:B\to \Bd(H)$ that extends $\pi$ as in \eqref{prop-inj-cov 3}.  Lemma \ref{int lem} implies that this ccp map integrates to a nondegenerate ccp map $\phi\rtimes u:B\rtimes_{\max} G \to \Bd(H)$.  As the diagram 
$$
\xymatrix{B\rtimes_{\max} G \ar[dr]^{\phi\rtimes u} & \\ A\rtimes_{\max} G \ar[r]^-{\pi\rtimes u} \ar[u]^-{\iota\rtimes G} & \Bd(H) }
$$
commutes and the horizontal map is injective, the vertical map is injective too.
\end{proof}

Notice that $A\rtimes_\max G=A\rtimes_\inj G$ if and only if every $G$-embedding $\iota\colon A\into B$ 
satisfies the equivalent conditions in Proposition~\ref{prop-inj-cov}. Hence we get the following immediate consequence, which we record for ease of future reference:

\begin{corollary}[{cf. \cite{Buss:2019}*{Corollary 4.11}}] \label{inj lem}
For a $G$-$C^*$-algebra $A$, the following are equivalent:
\begin{enumerate}
\item $A\rtimes_\max G=A\rtimes_\inj G$;
\item every nondegenerate covariant representation $(\pi,u)$ is $G$-injective;
\item there is a $G$-injective nondegenerate covariant representation that integrates to a faithful representation of $A\rtimes_\max G$.
\end{enumerate}
Moreover, if $G$ is exact, $\rtimes_{\inj}$ may be replaced by $\rtimes_{\red}$ in the above. \qed
\end{corollary}

In order to connect the above observations to amenability, we need the following extension 
of \cite{Buss:2019}*{Lemma 4.12} from discrete to locally compact groups. 

\begin{lemma}\label{unital lem}
Let $A$ be a $G$-$C^*$-algebra and let $(\pi,u)\colon (A,G)\to \Bd(H)$ be a nondegenerate $G$-injective covariant pair. Then for any  unital $G$-$C^*$-algebra $C$ there exists a ucp $G$-map $\phi:C\to \pi(A)'\sbe \Bd(H)$.
\end{lemma}
\begin{proof}
Let $G$ act diagonally on $C\otimes A$ and consider the canonical $G$-embedding 
$$
i\colon A\into \M(C\otimes A)_c;\quad a\mapsto 1\otimes a.
$$
Then $G$-injectivity of $\pi$ yields a ccp $G$-map $\varphi\colon \M(C\otimes A)_c\to \Bd(H)$ with $\varphi\circ i=\pi$.  Say now $(h_k)$ is an increasing approximate unit for $A$.  As $\varphi$ is ccp, we see that for any $k$
$$
1\geq \varphi(1) \geq \varphi(1\otimes h_k)=\varphi(i(h_k))=\pi(h_k).
$$ 
As $\pi(h_k)\to 1$ strongly, the above inequalities force $\varphi(1)=1$, so $\varphi$ is ucp.  We now consider the canonical $G$-embedding $j\colon C\to \M(C\otimes A)_c$, $c\mapsto c\otimes 1$, and then define $\phi\colon C\to \Bd(H)$ by $\phi(c):= \varphi(j(c))$. It remains to show that $\phi(C)\sbe \pi(A)'$. However, since $\varphi\circ i=\pi$ is a homomorphism, the image of $i$ lies in the multiplicative domain of $\varphi$, so that
$$\phi(c)\pi(a) 
=\varphi(j(c)i(a)) =\varphi(i(a)j(c))=\pi(a)\phi(c). \eqno\qedhere$$
\end{proof}

We are now going to introduce a notion of amenability which, at least for exact groups $G$,
 characterizes  weak containment:
 
\begin{definition}\label{com amen}
Let $(A,\alpha)$ be a $G$-$C^*$-algebra.  For a covariant representation $(\pi,u):(A,G)\to \Bd(H)$, equip the commutant $\pi(A)'$ with the action $\Ad u$.   
Then $(\pi,u)$ is \emph{commutant amenable} if  there  exists a net $(\theta_i:G\to \pi(A)')$ of norm-continuous, compactly supported, 
positive type functions such that
 $\|\theta_i(e)\|\leq 1$ for all $i$ and 
$\theta_i(g)\to 1$ ultraweakly and uniformly for $g$ in compact subsets of $G$.

The $G$-$C^*$-algebra $(A,\alpha) $ is \emph{commutant amenable}  if every\footnote{We will see in Proposition \ref{ca uni rep} below that it suffices for the universal representation to be commutant amenable.} nondegenerate covariant pair is commutant amenable.
\end{definition}

\begin{remark}\label{com amen sub c}
Lemma \ref{sub c gone} implies that one can require instead a net taking values in $\pi(A)'_c$ (but that otherwise has the same properties) without changing the definition of commutant amenability.
\end{remark}

In what follows, we say that a family  $\{(\pi_j, u_j): j\in J\}$  of covariant representations of $(A,G,\alpha)$
 is {\em faithful} if their direct sum integrates to  a faithful representation of $A\rtimes_\max G$.

\begin{lemma}\label{lem-CA}
Suppose that $(\pi,u)$ is a commutant amenable covariant representation 
of $(A,G,\alpha)$ on some Hilbert space $H$. Then for each $f\in C_c(G,A)$ we have
$$\|\pi\rtimes u(f)\|\leq \|f\|_{A\rtimes_\red G}.$$
In particular, if there exists a faithful family $\{(\pi_j, u_j): j\in J\}$ of commutant amenable covariant representations
of $(A,G,\alpha)$, then $A\rtimes_\max G= A\rtimes_\red G$.
\end{lemma}

\begin{proof}
Let $(\pi,u):(A,G)\to \Bd(H)$ be a commutant amenable covariant representation.  Using Remark \ref{com amen sub c} and Lemma \ref{prop-approx}, there is a net $(\xi_i)$ in $C_c(G, \pi(A)_c')\subseteq L^2(G, \pi(A)_c')$
with the properties from item \eqref{prop-approx 2} of  Lemma \ref{prop-approx}.
For each $i$, define 
$$
T_i:H\to L^2(G,H),\quad v\mapsto [g\mapsto \xi_i(g)v].
$$
A direct computation  shows that  for all $i\in I$ we have
\begin{align*}
\|T_iv\|^2 & =\braket{ v}{\braket{ \xi_i}{\xi_i}_{\pi(A)'}v}_\C\leq \|v\|^2.
\end{align*}
Thus $\|T_i\|\leq 1$ for all $i\in I$.  More direct computations show that the adjoint of $T_i$ is given by 
$$T_i^*(\eta)=\int_G\xi_i(g)^*\eta(g)\dd g$$
for $\eta\in \contc(G,H)$.

Now, via Fell's trick, the covariant pair $(\pi\otimes 1,u\otimes \lambda):(A,G)\to \Bd(L^2(G,H))$ integrates to $A\rtimes_{\red} G$.  Consider the 
net of contractive completely positive  maps
$$
\phi_i:\Bd(L^2(G,H))\to\Bd(H);\quad b\mapsto T_i^*bT_i.
$$
For $f\in \contc(G,A)\subseteq A\rtimes_\max G$ and $v\in H$ we compute:
\begin{align*}
\phi_i \big((\pi\otimes 1)\rtimes (u\otimes \lambda)(f)\big)v&=T_i^*((\pi\otimes 1)\rtimes (u\otimes \lambda)(f))T_iv\\
&=\left(\int_G\int_G\xi_i(h)^*\pi(f(g))u_g\xi_i(g^{-1}h)\dd g \dd h\right)v.
\end{align*}
Let $\Ad u$ be the action on $\pi(A)'$, and $\lambda^{\Ad u}$ the induced action on $L^2(G,\pi(A)')$ from Definition \ref{l2ga act}.  Using that $\xi_i$ takes values in $\pi(A)'$, we get
\begin{align*}
\phi_i \big((\pi\otimes 1)\rtimes (u\otimes \lambda)(f)\big)&=
\int_G\pi(f(g))\left(\int_G\xi_i(h)^*u_g\xi_i(g^{-1}h)u_g^*\dd h\right)u_g\dd g\\
&=\int_G\pi(f(g))\braket{ \xi_i}{\lambda^{\Ad u}_g\xi_i}_{\pi(A)'} u_g\dd g.
\end{align*}
As $\braket{ \xi_i}{\lambda^{\Ad u}_g\xi_i}_{\pi(A)'} $ converges weakly to $1$ uniformly for all $g$ in compact subsets of $G$, and as multiplication is separately weakly continuous, we get weak convergence 
$$
\phi_i \big((\pi\otimes 1)\rtimes (u\otimes \lambda)(f)\big)\to (\pi\rtimes u)(f).
$$
As weak limits do not increase norms and as each $\phi_i$ is contractive, this implies
\begin{align*}
\|(\pi\rtimes u)(f)\|&\leq \limsup_{i\to\infty}\|\phi_i\big( (\pi\otimes 1)\rtimes (u\otimes \lambda)(f)\big)\|\\
&\leq \|(\pi\otimes 1)\rtimes (u\otimes \lambda)(f)\| \leq \|f\|_{A\rtimes_\red G},
\end{align*} 
which completes the proof.
\end{proof}

Here is the main result of this section.  Note in particular that commutant amenability characterizes weak containment for actions of exact groups.

\begin{proposition}\label{prop-general}
Let $(A,\alpha)$ be a $G$-$C^*$-algebra.  Consider the 
following statements:
\begin{enumerate}
\item \label{prop-general 1}$\alpha$ is amenable;
\item \label{prop-general 2}$\alpha$ is commutant amenable;
\item \label{prop-general 3}there exists a commutant amenable, covariant representation $(\pi,u)$ of $(A,G,\alpha)$ such that 
$\pi\rtimes u:A\rtimes_\max G\to \Bd(H_\pi)$ is faithful;
\item \label{prop-general 4}$\alpha$ has the weak containment property.
\end{enumerate}
Then  
$$
\eqref{prop-general 1}\Rightarrow\eqref{prop-general 2}\Rightarrow\eqref{prop-general 3}\Rightarrow\eqref{prop-general 4}.
$$
Moreover, if $G$ is exact, we also have \eqref{prop-general 4} $\Rightarrow$ \eqref{prop-general 2}.
 \end{proposition}

 \begin{proof} 
Suppose that \eqref{prop-general 1} holds and let $(\pi, u)$ be a nondegenerate covariant representation of $(A,G,\alpha)$. 
 By Proposition \ref{prop-universal} there exists a normal surjective $G$\nb-equivariant $*$-homomorphism
$\Phi: A_\alpha''\onto \pi(A)''$.
Surjectivity implies that $\Phi$ restricts to a unital $*$-homomorphism 
$\Phi_Z: Z(A_\alpha'')\to  Z(\pi(A)'')\subseteq \pi(A)'$.
Thus if $(\theta_i:G\to Z(A_\alpha''))$ is a net of compactly supported positive type functions 
implementing amenability of  $\alpha$ it follows  that 
$(\Phi_Z\circ \theta_i)$ implements commutant amenability of $(\pi,u)$, hence \eqref{prop-general 2}. 
The implication \eqref{prop-general 2} $\Rightarrow$ \eqref{prop-general 3} is trivial and
\eqref{prop-general 3} $\Rightarrow$ \eqref{prop-general 4} follows from Lemma \ref{lem-CA} above.

Assume now that $G$ is exact.  Then \eqref{prop-general 4} and Corollary \ref{inj lem} imply that every nondegenerate covariant representation $(\pi,u)$ is $G$-injective, so Lemma \ref{unital lem} implies that there exists a $G$-equivariant ucp map
$\Phi:\contub(G)\to \pi(A)'$.  As $G$ is exact, \cite{Ozawa-Suzuki}*{Proposition 2.5} implies that the action on $\contub(G)$
is strongly amenable as in Definition \ref{def-amenable (SA)}, i.e., there exists a net $(\eta_i:G\to \contub(G))$ of norm-continuous, compactly supported, positive type functions 
such that $\|\eta_i(e)\|\leq 1$ for all $i$ and $\eta_i(g)\to 1_{\contub(G)}$ in norm and uniformly on compact 
subsets of $G$. But then $(\Phi\circ \eta_i)$ is a net of $\pi(A)'$-valued positive type functions that implements 
commutant amenability of $(\pi,u)$.
\end{proof}

\begin{remark}
We have seen in Section \ref{sec:amen pp} that our notion of amenability enjoys some desirable permanence properties like passage to ideals and quotients. It is also not difficult to see from the definition that commutant amenability passes to quotients and ideals. In particular, since commutant amenability implies the (WCP) by Proposition \ref{prop-general} above, it follows that commutant amenability of a $G$-$C^*$-algebra $A$ implies its \emph{inner exactness} in the sense that the sequence 
$$0\to I\rtimes_\red G\to A\rtimes_\red G\to (A/I)\rtimes_\red G \to 0$$
is exact for every $G$-invariant ideal $I\sbe A$.

On the other hand we will see in Theorem \ref{thm-Mats-2} below that if $A$ is amenable then $A\otimes_{\max} B$ is amenable for any $B$; this property fails for commutant amenability, as will follow from our examples in Section~\ref{sec:example} below.
\end{remark}

We conclude this section by showing that in order to check commutant amenability one does not have to study  all representations: indeed it suffices to just look at cyclic representations, or just at the universal representation.  We were motivated to show this by a question of Ruy Exel.

First, we record a basic lemma about permanence properties.

\begin{lemma}\label{ca perm}
Let $(A,\alpha)$ be a $G$-$C^*$-algebra.  Then the class of commutant amenable covariant representations of $(A,\alpha)$ is closed under taking subrepresentations, and under taking arbitrary direct sums.
\end{lemma}

\begin{proof}
Let $(\sigma,v)$ be a subrepresentation of $(\pi,u)$, where the latter acts on some Hilbert space $H$.   Assume that $(\pi,u)$ is commutant amenable, and let $(\theta_i:G\to \pi(A)')_{i\in I}$ be a net implementing its commutant amenability.  As $(\sigma,v)$
 is a subrepresentation of $(\pi,u)$, there is a $G$-invariant projection $p\in \pi(A)'$ such that $(\sigma,v)$ is the restriction of $(\pi,u)$ to $pH$.  In particular, the corner $p\pi(A)'p$ identifies canonically with $\sigma(A)'$.  Using this identification, the functions defined by $g\mapsto p\theta_i(g)p$ implement commutant amenability of $(\sigma,v)$.

Let now $(\pi_j,u_j)_{j\in J}$ be an arbitrary family of representations of $(A,G)$, all of which are commutant amenable, and let $(\pi,u)$ be their direct sum.  For each $j\in J$, let $(\theta_{i}^{(j)}:G\to \pi_i(A)')_{i\in I_j}$ be a net implementing the commutant amenability of $(\pi_j,u_j)$.  For each finite subset $F$ of $I$ and each tuple $\mathbf{i}:=(i_j)_{j\in F}\in \prod_{j\in F}I_j$, define $\theta_{F,\mathbf{i}}:G\to \pi(A)'$ by taking  the direct sum 
$$
\bigoplus_{j\in F} \theta^{(j)}_{i_j}:G \to \bigoplus_{j\in F} \pi_j(A)'
$$
and composing with the canonical inclusion 
$$
\bigoplus_{j\in F} \pi_j(A)'\into \pi(A)'.
$$
Let $\Lambda$ be the collection of pairs $\lambda=(F,\mathbf{i})$ where $F$ is a finite subset of $J$, and $\mathbf{i}\in \prod_{j\in F}I_j$.  Define a (directed) partial order on $\Lambda$ by stipulating that $(F,\mathbf{i})\leq (F',\mathbf{i}')$ if $F\subseteq F'$, and if for all $j\in F$, the $j^\text{th}$ entry of $\mathbf{i}$ is at most the $j^\text{th}$ entry of $\mathbf{i}'$. Then the net $(\theta_\lambda)_{\lambda\in \Lambda}$ implements commutant amenability of the direct sum representation; we  leave the direct checks to the reader.
\end{proof}

For the next result, recall from Section \ref{sec:meas amen} that we call a covariant representation $(\pi,u)$ of $(A,G)$ \emph{cyclic} if the integrated form $\pi\rtimes u$ is cyclic as a representation of $A\rtimes_{\max}G$.

\begin{proposition}\label{ca uni rep}
Let $(A,\alpha)$ be a $G$-$C^*$-algebra.  The following are equivalent:
\begin{enumerate}
\item \label{ca uni rep 1}$(A,\alpha)$ is commutant amenable;
\item \label{ca uni rep 2}all cyclic covariant representations of $(A,\alpha)$ are commutant amenable;
\item \label{ca uni rep 3}the universal  covariant representation of $(A,\alpha)$ is commutant amenable.
\end{enumerate}
\end{proposition}

\begin{proof}
The implications \eqref{ca uni rep 1} $\Rightarrow$ \eqref{ca uni rep 2} and \eqref{ca uni rep 1} $\Rightarrow$ \eqref{ca uni rep 3} hold by definition.  Any representation is a direct sum of cyclic representations, so \eqref{ca uni rep 2} $\Rightarrow$ \eqref{ca uni rep 1} by Lemma \ref{ca perm}.   Finally, \eqref{ca uni rep 3} $\Rightarrow$ \eqref{ca uni rep 2} by Lemma \ref{ca perm}, as any cyclic representation is (unitarily equivalent to) a subrepresentation of the universal representation by definition of the latter.
\end{proof}

\section{Applications of the Haagerup standard form}\label{sec-commutative}

A von Neumann algebra $M$ admits an essentially unique \emph{Haagerup standard form}: this is a representation of $M$ such that the relationship  between $M$ and $M'$ has particularly good properties, and that is covariant for a unitary representation of the group of automorphisms of $M$ with the point-ultraweak topology.  Crucially for us, a Haagerup standard form exists even for possibly non-countably decomposable von Neumann algebras such as $A_\alpha''$.  In this section, we exploit the Haagerup standard form to improve the results of the previous section.  In particular, we will see that commutant amenability and amenability coincide for actions on commutative $C^*$-algebras, and we will also show that amenability is characterized by the weak containment property of $A\otimes_{\max}A^{\op}$ for actions of exact groups.

The idea to use  Haagerup standard forms for the study of 
amenable  actions is due to Matsumura \cite{Matsumura:2012aa}, and was also exploited by the current authors in \cite[Section 5]{Buss:2019} for actions of discrete groups. 

The following theorem records the consequences of the Haagerup standard form that we will need: see \cite{Haagerup:1975xh}*{Theorem 2.3 and Corollary 3.6}.  For the statement, for a $G$-$C^*$-algebra $A$, let $(A^\op,\alpha^\op)$ denote its opposite $C^*$-algebra equipped with the $G$-action $\alpha^\op$ that agrees with $\alpha$ as an action on the underlying set.

\begin{theorem}[Haagerup]\label{thm-Haagerup}
Let $(A,\alpha)$ be a $G$-$C^*$-algebra.  There exist faithful normal representations $\pi$ of $A_\alpha''$ and $\pi^\op$ of $(A^\op)_{\alpha^\op}''$ on the same Hilbert space $H$
together with a strongly continuous unitary representation $u:G\to \U(H)$ such that the following are satisfied:
\begin{enumerate}
\item $(\pi, u)$  is covariant for $(A_\alpha'',G, \alpha'')$ and $(\pi^\op, u)$ is covariant for
\linebreak  $\big((A^\op)_{\alpha^\op}'', G, (\alpha^\op)''\big)$;
\item $\pi(A)'=\pi^{\op}((A^\op)_{\alpha^\op}'')$ and $\pi^{\op}(A^{\op})'=\pi(A_\alpha'')$.
\end{enumerate}
Moreover, if $A$ is commutative, we have $\pi(A)'=\pi(A)''\cong A_{\alpha}''$. \qed
\end{theorem}

The equivalence of \eqref{thm-exact-commutative 3} and \eqref{thm-exact-commutative 5} as in the next theorem is due to Matsumura \cite{Matsumura:2012aa}*{Theorem 1.1} for $G$ discrete and exact and $A=C(X)$ commutative and unital.  Matsumura's result has been extended by the authors to actions of discrete, exact $G$ on possibly non-unital $A=C_0(X)$ in \cite{Buss:2019}*{Theorem 5.2}. Here we give a version which works for actions of general locally compact groups.

\begin{theorem}\label{thm-exact-commutative}
Let $(A,G,\alpha)$ be a $G$-$C^*$-algebra with $A=C_0(X)$ commutative and consider 
the following statements:
\begin{enumerate}
\item \label{thm-exact-commutative 1}the underlying action $G\curvearrowright X$ is topologically amenable;
\item \label{thm-exact-commutative 2}$\alpha$ is strongly amenable;
\item \label{thm-exact-commutative 3}$\alpha$ is amenable;
\item \label{thm-exact-commutative 4}$\alpha$ is commutant amenable;
\item \label{thm-exact-commutative 5}$\alpha$ has the weak containment property.
\end{enumerate}
Then 
$$
\eqref{thm-exact-commutative 1} \Leftrightarrow \eqref{thm-exact-commutative 2} \Leftrightarrow \eqref{thm-exact-commutative 3} \Leftrightarrow \eqref{thm-exact-commutative 4}\Rightarrow \eqref{thm-exact-commutative 5}. 
$$
If, in addition, $G$ is exact,
all five properties are equivalent.
\end{theorem}

Only the relationships between the last three properties are new: \eqref{thm-exact-commutative 1} $\Leftrightarrow$ \eqref{thm-exact-commutative 2} is well-known (see Proposition \ref{rem-amenable} above), and the equivalence \eqref{thm-exact-commutative 2} $\Leftrightarrow$ \eqref{thm-exact-commutative 3} is a deep recent result of  Bearden and Crann \cite{Bearden-Crann}*{Corollary 4.14}.  We include \eqref{thm-exact-commutative 1} and \eqref{thm-exact-commutative 2} only for ease of reference.

\begin{proof} As discussed above, we only need to discuss the implications between \eqref{thm-exact-commutative 3}, \eqref{thm-exact-commutative 4}, and \eqref{thm-exact-commutative 5}. Proposition \ref{prop-general} gives \eqref{thm-exact-commutative 3} $\Rightarrow$ \eqref{thm-exact-commutative 4} $\Rightarrow$ \eqref{thm-exact-commutative 5}.  
To see \eqref{thm-exact-commutative 4} $\Rightarrow$ \eqref{thm-exact-commutative 3}  observe that  \eqref{thm-exact-commutative 4} implies in particular 
 that the Haagerup standard form representation $(\pi,u)$ is 
commutant amenable.  For $A$ commutative, we have $\pi(A)'=A_\alpha''=Z(A_\alpha'')$,
hence commutant amenability of $(\pi,u)$ implies amenability of $\alpha$.
If $G$ is exact, Proposition~\ref{prop-general} gives \eqref{thm-exact-commutative 5} $\Rightarrow$ \eqref{thm-exact-commutative 4}.
 \end{proof}

We now move on to noncommutative $G$-$C^*$-algebras.  For arbitrary $G$\nb-$C^*$\nb-algebras we get the following application of the Haagerup standard form:  this result extends Matsumura's \cite{Matsumura:2012aa}*{Theorem 1.1}, where the equivalence of \eqref{thm-Mats-2 1} and \eqref{thm-Mats-2 5} is  shown for  $G$ discrete and $A$ unital and nuclear.

\begin{theorem}\label{thm-Mats-2}
Let $(A,\alpha)$ be a $G$-$C^*$-algebra and consider the following statements:
\begin{enumerate}
\item \label{thm-Mats-2 1}$(A,\alpha)$ is amenable;
\item \label{thm-Mats-2 2}for every $G$-$C^*$-algebra $(B,\beta)$ the diagonal action  $(A\otimes_\max B, \alpha\otimes \beta)$ is  amenable;
\item \label{thm-Mats-2 3}for every $G$-$C^*$-algebra $(B,\beta)$ the diagonal action  $(A\otimes_\max B, \alpha\otimes \beta)$ is  commutant amenable;
\item \label{thm-Mats-2 4}the diagonal action $\alpha\otimes\alpha^{\op}:G\to \Aut(A\otimes_\max A^{\op})$ is commutant amenable;
\item \label{thm-Mats-2 5}$\alpha\otimes\alpha^\op$ satisfies the weak containment property.
\end{enumerate}
Then 
$$
\eqref{thm-Mats-2 1}\Leftrightarrow \eqref{thm-Mats-2 2}\Leftrightarrow\eqref{thm-Mats-2 3}\Leftrightarrow\eqref{thm-Mats-2 4}\Rightarrow \eqref{thm-Mats-2 5}
$$
and if $G$ is exact, all these properties are equivalent.
\end{theorem}
\begin{proof} The map 
$$
\Phi:A\to \M(A\otimes_{\max}B); \quad \Phi(a)=a\otimes 1
$$ 
is a nondegenerate $G$-equivariant $*$-homomorphism whose extension to enveloping 
$G$-von Neumann algebras (see Proposition \ref{evna func}) preserves centres. Thus it follows from Lemma \ref{lem:funct-strong-amenability} that 
$\alpha$ amenable implies $\alpha\otimes\beta$ amenable, whence  \eqref{thm-Mats-2 1} $\Rightarrow$ \eqref{thm-Mats-2 2}.  Proposition \ref{prop-general} shows that amenability always implies commutant amenability, whence \eqref{thm-Mats-2 2} $\Rightarrow$ \eqref{thm-Mats-2 3}. The implication \eqref{thm-Mats-2 3} $\Rightarrow$ \eqref{thm-Mats-2 4} is trivial.  Moreover, \eqref{thm-Mats-2 4} $\Rightarrow$ \eqref{thm-Mats-2 5}, and also 
\eqref{thm-Mats-2 5} $\Rightarrow$ \eqref{thm-Mats-2 4} in the case that $G$ is exact, follow from Proposition \ref{prop-general}.

Thus to complete the proof, we  need to show \eqref{thm-Mats-2 4} $\Rightarrow$ \eqref{thm-Mats-2 1}.
If $(\pi,u)$ and $\pi^{\op}$ are as in the Haagerup standard  form, we obtain a covariant 
representation $(\pi\times \pi^{\op}, u)$ of $(A\otimes_\max A^{\op}, G,\alpha\otimes\alpha^{\op})$ into 
$\Bd(H)$ such that $(\pi\times\pi^{\op})(a\otimes b)= \pi(a)\pi^{\op}(b)$ for all $a\in A$ and $b\in A^{\op}$.
Then \eqref{thm-Mats-2 4} implies that $(\pi\times \pi^{\op}, u)$ is commutant amenable.  
It  follows then from the properties of $\pi$ and $\pi^{\op}$ that 
$$(\pi\times\pi^{\op})(A\otimes_\max A^{\op})'=\pi(A)'\cap\pi(A^{\op})'=\pi(A)'\cap\pi(A)''=Z(\pi(A)'')\cong Z(A_\alpha''),$$ 
hence commutant amenability 
of $(\pi\times \pi^{\op}, u)$  implies amenability of $\alpha$ as required.
\end{proof}

As we will want to refer back to it below, we record the following corollary of Theorem \ref{thm-Mats-2} and Proposition \ref{prop-general}.

\begin{corollary}\label{cor-weakam}
Let $(A,\alpha)$ be a $G$-$C^*$-algebra with $G$ exact. Then the following are equivalent:
\begin{enumerate}
\item $(A,\alpha)$ is amenable.
\item For every $G$-$C^*$-algebra $(B,\beta)$, the diagonal action $\alpha\otimes\beta$ on $A\otimes_\max B$ has the weak containment property.\qed
\end{enumerate}
\end{corollary}

\begin{remark}\label{ad weak amen rem}
In  \cite{Anantharaman-Delaroche:2002ij} Anantharaman-Delaroche defined an action  
$(A,\alpha)$ of $G$ to be {\em weakly amenable}, if for every $G$-$C^*$-algebra $(B,\beta)$
$$(A\otimes_{\min} B)\rtimes_{\alpha\otimes \beta,\max}G\cong (A\otimes_{\min} B)\rtimes_{\alpha\otimes \beta,\red}G.$$
Clearly this is closely related to Corollary \ref{cor-weakam}, especially as Anantharaman-Delaroche also mentioned that there was no particular reason to choose the minimal tensor product rather than the maximal one in her definition.
\end{remark}

As an application of Theorem  \ref{thm-Mats-2} and Corollary \ref{cor-weakam}, we conclude this section by showing that amenability of an action $(A,\alpha)$  of $G$  passes to the restriction of the action 
to  an {\em exact} closed subgroup $H$ of $G$. At first sight, this statement looks trivial, since if $(\theta_i:G\to Z(A_\alpha''))$
is a net of continuous compactly supported positive type  functions which implement amenability of $(A,\alpha)$, then the net $(\theta_i|_H:H\to Z(A_\alpha''))$ certainly implements amenability 
of the action of $H$ on $A_\alpha''$, which implies  amenability of $\alpha|_H$ \emph{as long as} we know
 $A_{\alpha|_H}''=A_\alpha''$. This is  true if $H$ is \emph{open} in $G$, since then the 
 $H$-continuous states  of $A$ coincide with the $G$-continuous ones, hence both algebras have the same predual.

In general, it follows from the universal property of $A_{\alpha|_H}''$ that the identity of $A$ extends 
to a normal surjective $*$-homomorphism $q_H:A_{\alpha|_H}''\to A_\alpha''$, but this map is not always injective.
 For example, if $G$ acts on $A=C_0(G)$ by the translation action $\tau$, then $C_0(G)_\tau''=L^\infty(G)$ as already observed before. But if $H=\{e\}$, we get $C_0(G)_{\tau|_{\{e\}}}''=C_0(G)^{**}$
which, as observed before, differs from $L^\infty(G)$ if 
$G$ is not discrete. However, using Theorem~\ref{thm-Mats-2} we can show the following.

\begin{proposition}\label{prop-restrict}
Suppose that $\alpha:G\to \Aut(A)$ is an amenable action and that $H$ is an {\em open}, or {\em exact} and closed, subgroup of $G$.
 Then the restriction $\alpha|_H:H\to \Aut(A)$ is 
amenable as well. 
\end{proposition}

While this paper was under review, Ozawa and Suzuki \cite{Ozawa-Suzuki}*{Corollary 3.4} showed that the result holds true even without the exactness or openness assumptions on $H$ using a different argument.

\begin{proof} By the above discussion, we may assume that $H$ is exact.
By Theorem \ref{thm-Mats-2} it suffices to show that
\begin{equation}\label{eq-morH}
(A\otimes_\max A^{\op})\rtimes_\max H\cong(A\otimes_\max A^{\op})\rtimes_\red H
\end{equation}
 via the regular representation.
To see this we first observe that amenability of $\alpha$ implies that the diagonal action 
$\alpha\otimes\alpha^{\op}\otimes \tau$ of $G$ on $A\otimes_\max A^{\op}\otimes C_0(G/H)$ is amenable as well,
where $\tau:G\to \Aut(C_0(G/H))$ is given by left translation. As a consequence, we have
\begin{equation}\label{eq-mor}
\big(A\otimes_\max A^{\op}\otimes C_0(G/H)\big)\rtimes_\max G\cong \big(A\otimes_\max A^{\op}\otimes C_0(G/H)\big)\rtimes_\red G
\end{equation}
via the regular representation. 
Now by Green's imprimitivity theorem (e.g., see the discussion after \cite{CELY}*{Theorem 2.6.4}) there is a  canonical equivalence bimodule $X_H^G(A\otimes_{\max}A^{\op})$ 
which implements a Morita equivalence
$$(A\otimes_\max A^{\op}\otimes C_0(G/H))\rtimes_\max G\sim_M (A\otimes_\max A^{\op})\rtimes_\max H$$
and which factors through an equivalence bimodule $X_H^G(A\otimes_\max A^{\op})_\red$ which  gives a Morita equivalence 
for the reduced crossed products 
$$(A\otimes_\max A^{\op}\otimes C_0(G/H))\rtimes_\red G\sim_M (A\otimes_\max A^{\op})\rtimes_\red H.$$
The isomorphism (\ref{eq-morH}) then follows  from the isomorphism  (\ref{eq-mor})  and the Rieffel correspondence between ideals in 
Morita equivalent $C^*$-algebras (see \cite{Rieffel}*{Theorem 3.1}).
\end{proof}

\section{Weak containment does not imply amenability}\label{sec:example}

In this section, we present an example of a non-amenable action $\alpha:G\to \Aut(A)$ of a locally compact group $G$
on the $C^*$-algebra $A=\K(H)$ of compact operators such that \mbox{$A\rtimes_\max G$} $=A\rtimes_{\red}G$.  Thus the weak containment property (WCP) is not equivalent to amenability in general.  
The groups involved in our construction are concrete: for example, one could use $G=PSL(2,\C)$.  As this group is exact, and as for actions of exact groups the (WCP) is equivalent to commutant amenability, our construction also shows that commutant amenability is strictly weaker than 
amenability.  Our example will also show that the (WCP) (or commutant amenability) for a $G$\nb-$C^*$\nb-algebra $(A,G,\alpha)$ does not generally pass to the 
restriction $(A,H,\alpha|_H)$ to a closed subgroup $H$ of $G$, answering a question of Anantharaman-Delaroche \cite{Anantharaman-Delaroche:2002ij}*{Question 9.2 (b)}.

We do not have an example of a non-amenable action with the weak containment property where the acting group is  discrete.  We shall also see below that our construction is unlikely to produce anything interesting in that case, so the discrete case remains quite open.

In order to prepare our example, we need to recall some basic facts on circle-valued 
 Borel $2$-cocycles $\om:G\times G\to \T$, the corresponding maximal and reduced twisted  group algebras 
 $C^*_\max(G,\om)$ and $C^*_\red(G,\om)$, and their relations to actions of $G$ on the compact operators $\mathcal K(H)$
 and their crossed products. As references for more background and details, we suggest \cite{CKRW} and \cite{CELY}*{Section 2.8.6}.
 Throughout this section we assume that our groups are second countable and that $H$ is a separable Hilbert space.
 
 Recall that a {\em circle-valued Borel $2$-cocycle} on $G$ is a Borel map $\om:G\times G\to \T$ such that
 \begin{equation}\label{eq-cocycle}\om(g,h)\om(gh,l)=\om(g,hl)\om(h,l)\quad\text{and}\quad \om(g,e)=1=\om(e,g)\end{equation}
 for all $g,h,l\in G$, where $e$ denotes the neutral element of $G$.  We write $Z^2(G,\mathbb{T})$ for the set of all such Borel cocycles.
 Two cocycles $\om, \om'$ are \emph{equivalent} if there exists a Borel function $\varphi:G\to\mathbb T$ such that 
 \begin{equation}\label{eq-cocycleequiv}
 \om'(g,h)=\varphi(g)\varphi(h)\overline{\varphi(gh)}\om(g,h)
 \end{equation}
  and we write $H^2(G,\mathbb T)$ for the set of equivalence classes $[\om]$ for $\om\in Z^2(G,\mathbb T)$.
  Note that pointwise multiplication of cocycles induces a group multiplication on $Z^2(G,\mathbb T)$ and $H^2(G,\mathbb T)$, respectively.
 
An  {\em $\om$-representation} for a cocycle $\om \in Z^2(G,\mathbb T)$ is a 
 weakly Borel map $v:G\to \U(H)$ into the unitary group of a Hilbert space $H$ such that 
 $$v_g v_h=\om(g,h)v_{gh}\quad \forall g,h\in G.$$
The {\em regular $\om$-representation} $\lambda^\om:G\to\U(L^2(G))$ is defined by
$$\big(\lambda^\om_g\xi\big)(h)=\om(g, g^{-1}h)\xi(g^{-1}h).$$
The $\om$-twisted $L^1$-algebra $L^1(G,\om)$ consists of the Banach space $L^1(G)$ (with respect to Haar measure) 
with $\om$-twisted convolution and involution given by
$$f_1*_\om f_2(g)=\int_G f_1(h) f_2(h^{-1}g)\om(h, h^{-1}g)\,dg\;\text{and}\; f^*(g)=\Delta(g^{-1})\overline{\om(g,g^{-1})f(g^{-1})},$$
for $f, f_1, f_2\in L^1(G)$ and $g\in G$. Every  $\om$-representation $v:G\to\U(H)$ integrates to give a $*$-represen\-tation
$\tilde{v}:L^1(G,\om)\to \Bd(H)$ via
$$\tilde{v}(f)=\int_G f(g)v_g\,dg,$$
and the assignment $v\mapsto \tilde{v}$ gives a one-to-one correspondence between
$\om$-represen\-tations of $G$ and 
nondegenerate $*$-representations of $L^1(G,\om)$.

The maximal twisted group algebra $C^*_\max(G,\om)$ is the enveloping $C^*$-algebra
of $L^1(G,\om)$, i.e., the completion of $L^1(G,\om)$ by the $C^*$-norm
$\|f\|_\max=\sup_v\|\tilde{v}(f)\|$, where $v$ runs through all $\om$-representations of $G$.  The reduced twisted group algebra $C_\red^*(G,\om)$ is the completion of $L^1(G,\om)$ by the 
reduced norm $\|f\|_\red=\|\widetilde{\lambda^\om}(f)\|$. Up to isomorphism $C^*_\max(G,\om)$ and $C_\red^*(G,\om)$ only depend on the cohomology class $[\om]\in H^2(G,\mathbb T)$: if $\varphi:G\to\mathbb T$ implements an equivalence between $\omega$ and $\omega'$ as in (\ref{eq-cocycleequiv}), the map $f\mapsto \bar\varphi f$, $f\in L^1(G,\omega)$, extends to an isomorphism of the twisted group algebras.

Every $\omega$-representation $v:G\to \U(H)$  (and in particular $v=\lambda^{\om}$) determines an action $\alpha^{\omega}:=\Ad v:G\to\Aut(\mathcal K(H))$.   Following \cite{CKRW}*{Section 3}, let $\mathrm{Br}_G(\{\mathrm{pt}\})$ denote the equivariant Brauer group, which consists of all Morita equivalence classes (equivalently by \cite{CKRW}*{ Lemma 3.1}, stable outer conjugacy classes) of actions
$\alpha:G\to\Aut(\mathcal K(H))$ with multiplication given by 
$[\alpha]\cdot[\beta]=[\alpha\otimes \beta]$. It follows 
from \cite{CKRW}*{ Lemma 3.1 and Section 6.3} that, up to stable outer conjugacy, $\alpha^{\om}$ only depends on the class $[\om]\in H^2(G,\mathbb T)$, and that the map
$$
H^2(G,\mathbb T)\to \mathrm{Br}_G(\{\mathrm{pt}\}); \quad [\om]\mapsto [\alpha^{\om}]
$$
is an isomorphism of groups.

If $\om\in Z^2(G,\T)$ is a $2$-cocycle, its \emph{inverse} in $Z^2(G,\T)$ is given by the complex conjugate $\bar\om$ of $\om$,
and if $v:G\to \U(H)$ is an $\bar{\om}$-representation, we get an action $\alpha^{\bar\om}:=\Ad v:G\to \Aut(\mathcal K(H))$ 
of $G$ on the compact operators $\mathcal K:=\mathcal K(H)$ such that
\begin{equation}\label{cocycle mor ex}
C_\max^*(G,\om)\otimes \mathcal K \cong \mathcal K\rtimes_{\max}G \quad \text{and}\quad  
C^*_\red(G,\om)\otimes\mathcal K\cong \mathcal K\rtimes_\red G
\end{equation}
where both isomorphisms are extensions of the map
$$L^1(G,\om)\odot\mathcal K\to L^1(G,\mathcal K);  \quad f\otimes k\mapsto \Big(g\mapsto f(g)kv_g^*\Big)$$
(e.g., see  \cite{CELY}*{Remark 2.8.18 (1)}).
Conversely, since  every action of $G$ on $\mathcal K(H)$ 
is implemented by some $\bar\om$-representation for some Borel cocycle $\om\in Z^2(G,\T)$ (compare  \cite{CELY}*{pages 76-77}), we see that crossed products 
of group actions on $\mathcal K(H)$  always correspond to twisted group $C^*$-algebras.   As a direct consequence we get the following

\begin{observation}\label{ob1}
Assume that $\om$ is  a circle-valued Borel $2$-cocycle on the locally compact group $G$.
Then $C^*_\max(G,\om)\cong C^*_\red(G,\om)$ via the regular representation $\lambda^\om$ 
if and only if $\K\rtimes_{\alpha^{\bar\om},\max}G\cong \K\rtimes_{\alpha^{\bar\om},\red}G$
via the regular representation.
\end{observation}

The next  observation  follows easily from the definition of amenability together with  the fact that 
for any action $\alpha:G\to \Aut(\K(H))$ we have $\K(H)_\alpha''=\K(H)^{**}=\Bd(H)$ and $Z(\Bd(H))=\C$.

\begin{observation}\label{ob2}
An action $\alpha:G\to \Aut(\K(H))$ of a locally compact group $G$ on the algebra of compact operators on a Hilbert space $H$ 
is amenable if and only if $G$ is amenable.
\end{observation}

Thus, if we combine the above observations, in order to produce a non-amenable action $\alpha^{\bar\om}:G\to \Aut(\K(H))$ that satisfies the (WCP), it suffices  to find a non-amenable group $G$ and a circle-valued $2$-cocycle  $\om:G\times G\to \T$ 
such that $C^*_\max(G,\om)\cong C^*_\red(G,\om)$  via the $\om$-regular representation $\lambda^\om$.\medskip

In order to find such examples, we now consider central extensions
$$1\to Z\to L\to G\to 1$$
of second countable locally compact groups. In this situation the maximal group $C^*$-algebra
$C^*_\max(L)$ carries a canonical structure of a $C_0(\widehat{Z})$-algebra via the 
structure homomorphism
$$\Phi:C_0(\widehat{Z})\cong C^*(Z)\to Z\M(C^*_\max(L)); \quad  \Phi(\widehat{f})=\int_Z f(z)i_L(z)\,dz,$$
where $\widehat{f}\in C_0(\widehat{Z})$ denotes the Fourier transform of  $f\in C_c(Z)\subseteq C^*(Z)$, 
 and where $i_L:L\to \U\M(C^*_\max(L))$ is the canonical homomorphism. It follows from \cite{Dana-book}*{Proposition C.5} that the fibre
 $C^*_\max(L)_\chi$ over a character $\chi\in \widehat{Z}$ is  the 
quotient of $C^*_\max(L)$
by the ideal
\begin{equation}\label{eq-ideal}
I_\chi:=\bigcap\left\{\ker \tilde{u}: u\in \widehat{L}, u|_Z=\chi\cdot 1_{H_u}\right\}.
\end{equation}
Composing $\Phi:C_0(\widehat{Z})\to Z\M(C^*_\max(L))$ with the regular representation 
induces a similar $C_0(\widehat{Z})$-algebra structure on $C^*_\red(L)$ and if $\widehat{Z}$ is discrete\footnote{This is also true for  general $Z$ if $G$ is exact: the key points are the Packer-Raeburn stabilization trick \cite{PR}*{Proposition 1.1} and the work of Kirchberg and Wasserman \cite{KW-exact}*{Theorem 4.2} relating exactness to continuous fields; as we will not need this, we do not provide more details.} (i.e.\ $Z$ is compact), the fibre
$C^*_\red(L)_\chi$ at $\chi\in \widehat{Z}$ is the quotient
$$C^*_\red(L)_\chi=C^*_\red(L)/I_\chi^{\red}= C^*_\max(L)/(\ker\tilde{\lambda}_L+I_\chi)=C^*_\max(L)_\chi/q_\chi(\ker\tilde{\lambda}_L)$$
with 
\begin{equation}\label{eq-idealred}
I_\chi^{\red}:=\bigcap\left\{\ker \tilde{u}: u\in \widehat{L}, u|_Z=\chi\cdot 1_{H_u},  u\prec \lambda_L\right\}\subseteq C^*_\red(L),
\end{equation}
and where $q_\chi:C^*_\max(L)\to C_\max^*(L)_\chi$ denotes the quotient map.

If $Z$ is compact 
(we will  actually only use the case that $Z$ is finite), we  get direct sum decompositions 
\begin{equation}\label{eq-decompose}
C^*_\max(L)=\bigoplus_{\chi\in \widehat{Z}} C_\max^*(L)_\chi \quad\text{and} \quad 
C^*_\red(L)=\bigoplus_{\chi\in \widehat{Z}} C_\red^*(L)_\chi. \end{equation}
The fibres $C_\max^*(L)_\chi$ and $C_\red^*(L)_\chi$ have alternative descriptions as twisted group algebras.
This well-known fact, which goes back to Mackey's analysis of unitary representations for group extensions, 
can be deduced, for example, from \cite{PR}*{Theorem 1.2} (see also \cite{EW}*{Lemma 6.3}), but in order to give a complete picture, 
we present the main ideas below.

First we choose once and for all a Borel cross-section $c:G\to L$  for the quotient map $q:L\to G$ such that $c(e_G)=e_L$ (such a cross-section exists  by \cite{Arv-book}*{Theorem 3.4.1}, for example).
Then each character $\chi\in \widehat{Z}$ determines a $2$-cocycle
$\om_\chi:G\times G\to\T$  by 
$$\om_\chi(g,h)=\chi\big(c(g)c(h)c(gh)^{-1}\big), \quad g,h\in G.$$
The following lemma  is a special case of \cite{Kaniuth}*{Theorem 4.53}, but can be traced back to Mackey \cite{Mackey}.

\begin{lemma}\label{lem-reps}
For $\chi\in \widehat{Z}$, let $\om_\chi$ be as above. Then for every Hilbert space $H$, the assignment
 $$u\mapsto v:=u\circ c$$
 gives a one-to-one correspondence between
the unitary representations $u:L\to \U(H)$ which restrict to $\chi\cdot 1_H$ on $Z$ and the 
$\om_\chi$-representations $v:G\to \U(H)$.  \qed
\end{lemma}

\begin{remark}\label{rem-induced}
If $\chi\in \widehat{Z}$, then, using the Borel section $c:G\to L$, the induced representation  $\Ind_Z^L\chi$ in the sense of Mackey and Blattner (e.g., see \cite{CELY}*{Section 2.7} for the definition in this setting) can  
be realized on the Hilbert space $L^2(G)$ by the formula
\begin{equation}\label{eq-Ind}
\big(\Ind_Z^L\chi(c(g)z)\xi\big)(h)=\chi(z)\om_\chi(g, g^{-1}h)\xi(g^{-1}h) \quad g,h\in G, z\in Z.
\end{equation}
Indeed, if $ H_\chi=\overline{\mathcal F_\chi}$ is the Hilbert space of the induced representation as defined 
preceding \cite{CELY}*{Proposition 2.7.7}, then  the map 
$u: \mathcal F_\chi\to L^2(G); \xi\mapsto \xi\circ c$
extends to a unitary intertwiner between Blattner's realization of $\Ind_Z^L\chi$ and the representation  defined
by the formula  (\ref{eq-Ind}).

It follows that $\Ind_Z^L\chi$ corresponds to the $\om_\chi$-regular representation 
$\lambda^{\om_\chi}$ under the correspondence of Lemma \ref{lem-reps} above. In particular, if $Z$ is compact, the regular representation of $L$ decomposes as the direct sum $\oplus_{\chi\in \widehat{Z}}\lambda^{\om_\chi}$ under the direct sum decomposition in (\ref{eq-decompose}).
\end{remark}

The proof of the  next proposition follows along the lines of \cite{PR}*{Theorem 1.2} (see also \cite{EW}*{Lemma 6.3}).
Since every irreducible representation of $L$ restricts to a multiple of some character $\chi\in \widehat{Z}$, the statement for the maximal group algebra can  be deduced from  Lemma \ref{lem-reps} and a straightforward computation.
The statement for the reduced group $C^*$-algebra $C_{\red}^*(L)$ follows in a similar way by using Remark \ref{rem-induced} and 
 \cite{CELY}*{Lemma 2.8.13}.

\begin{proposition}[cf \cite{PR}*{Theorem 1.2}]\label{prop-fibres}
Let $1\to Z\to L\to G\to 1$ be a central extension of second countable groups, and for each $\chi\in \widehat{Z}$ let $\om_\chi\in Z^2(G,\T)$ be the cocycle
determined by $\chi$ and the Borel cross-section $c:G\to L$. Then the map
$$\varphi_\chi: C_c(L)\to L^1(G,\om_\chi);\quad  \varphi_\chi(f)(g):=\int_Z f(c(g)z)\chi(z)\, dz$$
extends to a surjective $*$-homomorphism $\phi_\chi:C^*_\max(L)\to C_\max^*(G,\om_\chi)$ with 
kernel $I_\chi$, and therefore induces an isomorphism
$$ C_\max^*(L)_\chi \cong C^*_\max(G,\om_\chi).$$
Similarly, if $Z$ is compact (or $G$ is exact), the map $\varphi_\chi$ induces an isomorphism $ C_\red^*(L)_\chi \cong C^*_\red(G,\om_\chi).$ \qed
\end{proposition}

We are now ready to formulate the following principle:

\begin{proposition}\label{prop-principle}
Suppose that $1\to Z\to L\to G\to 1$ is  a central extension of second countable groups such that $Z$ is compact (or $G$ is exact).
Suppose further that $G$ is not amenable and there exists a character $\chi\in \widehat{Z}$ such that the following holds:
\begin{itemize}
\item[(Z)] Every irreducible unitary representation $u:L\to\U(H)$ of $L$ that restricts to $\chi\cdot 1_H$ on $Z$ is weakly contained in the regular representation $\lambda_L$ of $L$.
\end{itemize}
Then $C^*_\max(G, \om_\chi)\cong C^*_\red(G,\om_\chi)$ via the $\om_\chi$-regular representation $\tilde\lambda^{\om_\chi}$.
In particular, there exists a {\em non-amenable action} $\alpha^\chi:G\to \Aut(\K(H))$ such that 
$$\K(H)\rtimes_\max G\cong \K(H)\rtimes_\red G$$
 via the regular representation.
\end{proposition}
\begin{proof} If $\chi\in \widehat{Z}$ is such that property (Z) holds for $\chi$, then it follows from the description of 
the fibres $C^*_\max(L)_\chi$ and $C^*_\red(L)_\chi$ as quotients of $C_\max^*(L)$ by the ideals described in 
(\ref{eq-ideal}) and (\ref{eq-idealred}) that the regular representation $\tilde\lambda_L:C^*_\max(L)\to C^*_\red(L)$ induces an isomorphism $C^*_\max(L)_\chi\cong C^*_\red(L)_\chi$. The result then follows from 
Proposition \ref{prop-fibres} together with Observations \ref{ob1} and \ref{ob2}.
 \end{proof}

We are grateful to Timo Siebenand for helpful discussions towards the following example for a 
central extension satisfying all the assumptions of the above proposition.

\begin{example}\label{ex-SL(2,R)}
Consider the central extension 
$$1\to C_2\to \SL(2,\C)\to \PSL(2,\C)\to 1,$$ where $C_2$ denotes the cyclic group of order two
sitting in $\SL(2,\C)$ via $\pm I$ with $I$ the identity matrix. Write $\N$ for the positive natural numbers, and $\N_0$ for $\N\cup\{0\}$.  The representation theory of $\SL(2,\C)$ is well known and,
following \cite{Fell}*{Chapter III}, 
a complete list of (equivalence classes) of irreducible representations of $\SL(2,\C)$ 
can be parametrized by the parameter space
$\mathcal P$ consisting of the disjoint union of the following subsets of $\N_0\times \C$:
$$\mathcal P= \big[\N\times i\R\big]\cup \big[\{0\}\times i [0,\infty)\big]\cup \big[\{0\}\times (0,1)\big]\cup \{(0,2)\}.$$
If $(0,2)\neq (n,s)\in \mathcal P$, the corresponding irreducible representation $u^{(n,s)}:\SL(2,\C)\to \U(H_{(n,s)})$
acts on a Hilbert space $H_{(n,s)}$ consisting of certain functions $\xi:\C\to \C$ by the formula
\begin{equation}\label{eq-repsSL2C}
\Big(u^{(n,s)}\left(\begin{smallmatrix} a&b\\ c&d\end{smallmatrix}\right) \xi\Big)(z)
=(bz+d)^{-n}|bz+d|^{n-2s-2}f\Big(\frac{az+c}{bz+d}\Big).
\end{equation}
The point $(0,2)\in \mathcal P$ parametrizes the trivial representation of $\SL(2,\C)$. 
It has been shown by Lipsman in \cite{Lips}
 that the only representations in this list which are not weakly contained in 
the regular representation $\lambda$ are the trivial representation (with parameter $(0,2)$) and 
the representations $u^{(0,t)}$ with $t\in (0,1)$. But formula  (\ref{eq-repsSL2C}) easily shows that all these
representations restrict to a multiple of the trivial character of $C_2$. Hence it follows that 
the non-trivial character $\chi$ of $C_2$ satisfies all assumptions of Proposition \ref{prop-principle}.

A similar direct approach, using the  well-known representation theory of $\SL(2,\R)$, shows that 
the central extension $1\to C_2\to \SL(2,\R)\to \PSL(2,\R)\to 1$ together with the non-trivial 
character of $C_2$ also gives an example  satisfying the assumptions of Proposition \ref{prop-principle}.
\end{example}

\begin{remark}\label{rem:op}
We should point out that the above example does not contradict Theorem \ref{thm-Mats-2}. Indeed, if 
$\om\in Z^2(G,\T)$ and $\alpha^\om=\Ad\lambda^{\om}:G\to\Aut(\K)$ is a corresponding action on $\K=\K(L^2(G))$ under the 
isomorphism $H^2(G,\mathbb T)\cong \mathrm{Br}_G(\{\mathrm{pt}\})$,  then
$(\alpha^{\om})^{\op}:G\to \Aut(\K^{\op})$ can be identified (up to stable  outer conjugacy) with
$\alpha^{\bar\om}:G\to \Aut(\K)$. This follows from the fact that under the bijection $T\mapsto T^{\op}$ from $\mathcal K\to \mathcal K^{\op}$ we get
$$\alpha^{\op}_g(T^{\op})=(\lambda^{\om}_gT\lambda^{\om}_{g^{-1}})^{\op}=(\lambda^{\om}_{g^{-1}})^{\op}T^{\op}(\lambda^{\om}_g)^{\op}$$
and that $g\mapsto v_g:=(\lambda_{g^{-1}}^{\om})^{\op}$ is a $\tilde{\om}$-representation for the cocycle
$$\tilde{\om}(g,h)=\om(h^{-1}, g^{-1})\quad g,h\in G,$$
which is equivalent to $\bar{\om}$ (e.g., see \cite{BNS}*{p.\ 989}). 
 But then the diagonal action 
$\alpha^\om\otimes(\alpha^{\om})^{\op}\sim \alpha^\om\otimes\alpha^{\bar \om}$  corresponds to the 
trivial cocycle $1=\om\cdot\bar\om$.  Therefore 
$\alpha^\om\otimes(\alpha^{\om})^{\op}$ represents the trivial class in the Brauer group 
$\mathrm{Br}_G(\{\pt\})$ and hence is Morita equivalent to the trivial action of $G$ on $\K(H)$.
However, the trivial action on a non-zero $C^*$-algebra $A$  satisfies the (WCP) if and only if $G$ is amenable: indeed, one has canonical surjections
$$
A\rtimes_{\id,\max} G\cong A\otimes_{\max} C^*_{\max}(G)\to A\otimes C^*_{\max}(G)\to A\otimes C^*_{\red}(G)\cong A\rtimes_{\id,\red}G
$$
and if the action has the (WCP) all these maps must be isomorphisms, which forces the canonical quotient $C^*_{\max}(G)\to C^*_{\red}(G)$ to be an isomorphism, too.   It follows in particular that  for the actions $\alpha:G\to\Aut(\K)$ constructed by the principle in
 Proposition \ref{prop-principle}, the actions $\alpha\otimes \alpha^{\op}$ never satisfy the (WCP).
 \end{remark}

A similar construction of actions on $\K=\K(H)$ of a non-amenable {\em discrete} group $\Gamma$ seems to 
be unlikely, due to the following observation:

\begin{proposition}\label{prop-discrete}
Suppose that $\Gamma$ is a discrete group which contains a copy of the free group $\mathbb{F}_2$ in two generators.
Then,  for every action $\alpha:\Gamma\to\Aut(\K)$ the regular representation $\Lambda: \K\rtimes_\max\Gamma\to \K\rtimes_\red\Gamma$ is not 
faithful. Hence $(\K, \alpha)$ does not satisfy the (WCP).
\end{proposition}

For the proof we need:

\begin{lemma}\label{lem-restrict}
Suppose that $\alpha:G\to \Aut(A)$ is an action of the locally compact group $G$  on the $C^*$-algebra $A$
and let $H\subseteq G$ be an open subgroup of $G$. If $(A,\alpha)$ satisfies the (WCP), then so does $(A,\alpha|_H)$.
\end{lemma}
\begin{proof}
This follows from the commutative diagram
$$
\begin{CD}
A\rtimes_\max H @>\Lambda_{(A,H)}>> A\rtimes_\red H\\
@VVV @VVV\\
A\rtimes_\max G@>\Lambda_{(A,G)}>> A\rtimes_\red G
\end{CD}
$$ together with the fact that the vertical arrows  are faithful, since  $H$ is open in $G$, and the lower horizontal
arrow is faithful if $(A,\alpha)$ satisfies the (WCP).
\end{proof}

\begin{proof}[Proof of Proposition \ref{prop-discrete}]
Assume that $\alpha:\Gamma\to\Aut(\K)$ is an action which satisfies the (WCP). 
By assumption, $\Gamma$ contains the free group $\mathbb{F}_2$. Thus it follows from Lemma \ref{lem-restrict} 
that the restriction of $\alpha$ to $\mathbb{F}_2\subseteq \Gamma$ also satisfies the (WCP).
Since every automorphism of $\K=\K(H)$ is of the form $\Ad u$ for some $u\in \U(H)$,
it follows from the freeness of $\mathbb{F}_2$ that there exists a unitary representation 
$u:\mathbb{F}_2\to \U(H)$ such that $\alpha_g=\Ad u_g$ for each $g\in \mathbb{F}_2$.  But then $\alpha$ is Morita equivalent to the trivial action. Since $\mathbb{F}_2$ is not amenable, it follows as  in Remark \ref{rem:op} above that $\alpha$ does not have the (WCP).
\end{proof}

We now discuss applications to restricted actions.  In \cite{Anantharaman-Delaroche:2002ij}*{Question 9.2 (b)} Anantharaman-Delaroche asked: if $(A,G,\alpha)$ is a $G$-$C^*$-algebra with the (WCP), and if $H$ is a closed subgroup of $G$, then does $(A,H,\alpha|_H)$ also have the (WCP)?  We will use our work above to give two examples showing that the answer is no.

Here is our first source of examples where the (WCP) does not behave well with respect to restrictions.  The proof is immediate from Proposition \ref{prop-discrete}.

\begin{proposition}\label{cor: wc closed}
Let $(\K,G,\alpha)$ be a non-amenable action with the (WCP) as in the conclusion of Proposition \ref{prop-principle}, and let $\Gamma$ be a discrete subgroup of $G$ that contains a copy of the free group $\mathbb{F}_2$.  Then the restricted action $(\K,\Gamma,\alpha|_\Gamma)$ does not have the (WCP).  \qed
\end{proposition}

\begin{example}
There are many examples satisfying the condition in the statement.  To give a concrete example, note that using Example \ref{ex-SL(2,R)}, we may assume that $G={\PSL}(2,\R)$.  The fundamental group of any closed, orientable surface of genus at least two embeds as a discrete subgroup of ${\PSL}(2,\R)$ by the uniformization theorem.  Such a group $\Gamma$ always contains a copy of $\mathbb{F}_2$, as is clear from its standard presentation.  
\end{example}

Below we give a second source of examples where the (WCP) does not behave well with respect to restriction.  This second collection of examples is perhaps even more striking, as the subgroup in this case is just the diagonal subgroup of a product $G\times G$. 

\begin{proposition}\label{diag}
Let $(\K,G,\gamma)$ be a non-amenable action of $G={\PSL}(2,\C)$ or $G={\PSL}(2,\R)$ satisfying the (WCP) as in Example \ref{ex-SL(2,R)}.  Let $A=\K\otimes \K^\op$ be equipped with the $G\times G$ action defined as $\alpha_{(g,h)}:=\gamma_g\otimes\gamma_h^\op$.  Let $H:=\{(g,g)\in G\times G\mid g\in G\}\subseteq G\times G$ be the diagonal subgroup.

Then $(A,G\times G,\alpha)$ satisfies the (WCP), but $(A,H,\alpha|_H)$ does not.
\end{proposition}

\begin{proof}
We note first that $(\K^\op,G, \gamma^\op)$ satisfies the (WCP).  This follows as the cocycle $\omega$ corresponding to $\gamma^\op$ takes values in $C_2\subseteq \T$, and thus $\omega$ equals its complex conjugate $\overline{\omega}$.  As in Remark \ref{rem:op}, $\gamma^\op$ is the action on $\K^\op\cong \K$ that corresponds to $\overline{\omega}$, whence $(\K^\op,\gamma^\op)$ and $(\K,\gamma)$ are Morita equivalent.  

Note moreover that $\K\rtimes_\red G$ is nuclear (and therefore also $\K\rtimes_\max G$, $\K^\op\rtimes_\red G$ and $\K^\op\rtimes_{\max} G$ are nuclear).  Indeed, using line \eqref{cocycle mor ex} above, $\K\rtimes_{\red} G$ is Morita equivalent to $C^*_{\red}(G,\omega)$.  On the other hand, using line \eqref{eq-decompose} and Proposition \ref{prop-fibres}, $C^*_\red(G,\omega)$ is a direct summand of 
$C^*_\red(L)$, for $L=\SL(2,\C)$.  Finally, note that $L$  is connected, whence $C^*_\red(L)$ is nuclear by \cite{Connes:1976fj}*{Corollary 6.9 (c)}.

Now, to see that $(A,G\times G,\alpha)$ satisfies the (WCP), note that 
\begin{multline*}
A\rtimes_\max (G\times G)\cong(\K\rtimes_\max G)\otimes (\K^\op\rtimes_\max G)\\\cong(\K\rtimes_\red G)\otimes (\K^\op\rtimes_\red G)\cong A\rtimes_\red (G\times G).
\end{multline*}
Notice that all $C^*$-algebras involved here are nuclear, so that we do not need to worry about the choice of maximal or minimal tensor product. On the other hand, the restriction of the $G\times G$-action on $A$ to $H$ is $\gamma\otimes\gamma^\op$ which, by the non-amenability of $\gamma$ (or $G$), does not satisfy the (WCP) by Theorem~\ref{thm-Mats-2} 
(or by Remark \ref{rem:op}).
\end{proof}

We close this section with an application of the above results towards property (WF3)  as considered by  Bekka and Valette in \cite{BV}.
A locally compact group  $G$ satisfies property (WF3) if (and only if) for every closed subgroup $H$ and every irreducible unitary representation $v$ of $H$ there exists a unitary representation $u$ of $G$ such that $v$ is weakly contained in the restriction $u|_H$.
As pointed out in \cite{BV}, this is equivalent to asking whether the canonical $*$-homomorphism
$$j_H:C^*_\max(H)\to\M(C^*_\max(G))$$
given as the integrated form of the restriction $i_G|_H$ of the canonical map $i_G:G\to \U\M(C^*_\max(G))$ is faithful for 
all closed subgroups $H$ of $G$. Note that amenable groups and discrete groups always satisfy (WF3) but it is shown in \cite{BV} that (WF3) might fail in general. 
Indeed, \cite{BV}*{Theorem 1.3} shows that  an almost connected group $G$ satisfies (WF3) if and only if $G$ is amenable.
It is an interesting problem for a given group $G$ to determine all closed subgroups $H$ for which the 
map $j_H:C^*_\max(H)\to\M(C^*_\max(G))$ fails to be injective. The following result 
complements the results on lattices in $\SL(2,\R)$ and $\SL(2,\C)$ as given in \cite{BV}*{Section 5}:

\begin{theorem}\label{thm-WF3}
Let $\Gamma$ be any nonamenable discrete subgroup of $G=\SL(2,\R)$ (resp.~ $G=\SL(2,\C)$) which contains the centre $Z$ of 
$G$. Then the canonical $*$\nb-homomorphism $j_\Gamma: C^*_\max(\Gamma)\to \M(C^*_\max(G))$ is not injective.
\end{theorem}
\begin{proof} 
As $\Gamma$ is nonamenable, it contains a free subgroup on two generators by the Tits alternative \cite{Tits:1972ut}*{Corollary 1}.
Since $C_2\cong Z\subseteq \Gamma$ we may write $C^*_\max(\Gamma)$ as the direct sum 
$C^*_\max(\Gamma)_{1_Z}\oplus C^*_\max(\Gamma)_\chi$ as in (\ref{eq-decompose}) where $1_Z$ denotes the 
trivial character and $\chi$ the nontrivial character of $Z$. With the similar decomposition of $C^*_\max(G)$ 
as $C^*_\max(G)_{1_Z}\oplus C^*_\max(G)_\chi$ it is easy to see that $j_\Gamma$ decomposes into the direct sum 
of the two canonical $*$-homomorphisms
$$j_1:C^*_\max(\Gamma)_{1_Z}\to \M(C^*_\max(G)_{1_Z})\quad \text{and}\quad j_\chi: C^*_\max(\Gamma)_\chi\to \M(C^*_\max(G)_\chi).$$
Thus $j_\Gamma$ is faithful if and only if both $j_1$ and $j_\chi$ are faithful. 
If $\alpha_\chi:G/Z\to\Aut(\K)$ is the corresponding action on the compacts, faithfulness of $j_\chi$ 
translates into faithfulness of the canonical $*$-homomorphism
$$j_\alpha: \K\rtimes_{\alpha_\chi,\max}  (\Gamma/Z)\to \M(\K\rtimes_{\alpha_\chi,\max} (G/Z)).$$
By Example  \ref{ex-SL(2,R)} we have $\K\rtimes_\max (G/Z)= \K\rtimes_\red (G/Z)$ and hence
$j_\alpha$ factors through the composition 
$$\K\rtimes_\max  (\Gamma/Z)\stackrel{\Lambda}{\to} \K\rtimes_\red{(\Gamma/Z)}\into  \M(\K\rtimes_\red (G/Z))$$
which is not faithful by Proposition \ref{prop-discrete}.
\end{proof}

\chapter{Actions on $X\rtimes G$-algebras and type I $C^*$-algebras}\label{chap:examples}
In this chapter we study amenability of actions on $C^*$-algebras with particularly good structure.  Some results require separability and exactness assumptions: see the theorems below for precise assumptions.

Very roughly, a $G$-space $X$ is regular if the quotient $G\backslash X$ is not too wild.   In Section \ref{sec:regular} we discuss $X\rtimes G$-algebras over a regular $G$-space $X$.   The key tool used throughout Section \ref{sec:regular} is Theorem \ref{thm-Mats-2}; as a consequence most results of this section require exactness of the acting group.  A simple example of an $X\rtimes G$\nb-algebra with $X$ regular occurs if $H$ is a closed subgroup of $G$ and we consider the $G$\nb-space $X=G/H$.  Then, if $A$ is an $H$-$C^*$-algebra, there is an induced $G$\nb-$C^*$\nb-algebra $\Ind_H^GA$, which turns out to be a $(G/H)\rtimes G$-algebra.  In the special case of $\Ind_H^GA$ our results imply that for $G$ exact, $(A,H)$ is amenable if and only if $(\Ind_H^GA,G)$ is amenable, partially generalizing a result of Anantharaman-Delaroche \cite{Anantharaman-Delaroche:1987os}*{Th\'eor\`eme 4.6} for actions of discrete groups.   Another important application occurs when $A$ is a type I $C^*$-algebra, and the induced action of the exact group $G$ on $\widehat{A}$ (which need not be Hausdorff) is regular.  According to a classical result of Glimm, $\widehat{A}$ then admits a sort of decomposition into \emph{Hausdorff} regular $G$-spaces, and we use this to show that $(A,G,\alpha)$ is amenable if and only if for each $[\pi]\in \widehat{A}$, the stabilizer subgroup $G_\pi$ of $[\pi]$ is an amenable group.  

In Section \ref{sec:continuous-trace} we continue to study type I $C^*$-algebras, but under the additional assumption that $\widehat{A}$ is Hausdorff.  In this case, we show that $(A,G)$ is amenable if and only if $(C_0(\widehat{A}),G)$ is amenable.  

Both of these characterizations of amenability for actions on type I $C^*$-algebras can be seen as partial generalizations of Observation \ref{ob2} above, which is the case of an action on a type I $C^*$-algebra with $\widehat{A}$ a single point.

\section{Amenability of regular $X\rtimes G$-algebras}\label{sec:regular}

In this section, we study actions on $C^*$-algebras that decompose in a good way over a $G$-space $X$.  We give applications to induced $G$-$C^*$-algebras and type I $G$-$C^*$-algebras.

Recall that if $X$ is a locally compact $G$-space, then a $G$-$C^*$-algebra $(A,\alpha)$ is called an \emph{$X\rtimes G$-algebra}, if it is equipped with a nondegenerate $G$-equivariant
$*$\nb-homo\-mor\-phism $\Phi:C_0(X)\to Z\M(A)$. 
If $A$ is an $X\rtimes G$-algebra, then for each $x\in X$ the {\em fibre} $A_x$ of $A$ at $x$ is the
quotient $A_x:=A/I_x$, where $I_x:=\Phi(C_0(X\smallsetminus\{x\}))A$ is the ideal of ``sections''  vanishing at $x$.
It follows from 
 \cite{Dana-book}*{Theorem C.26} that a $G$-$C^*$-algebra $(A,\alpha)$ has the structure of an $X\rtimes G$-algebra 
 if and only if  there exists a continuous $G$-equivariant map $\varphi:\Prim(A)\to X$. More precisely, 
 given $\Phi:C_0(X)\to Z\M(A)$ as above, then the corresponding map $\varphi:\Prim(A)\to X$
 sends the closed subspace $\Prim(A_x)\subseteq \Prim(A)$ to the point $x$. 
 
If $G_x=\{g\in G: gx=x\}$ denotes the stabilizer at a point $x\in X$, then 
$\alpha$ induces an action $\alpha^x:G_x\to \Aut(A_x)$ via $\alpha^x_g(a+I_x)=\alpha_g(a)+I_x$. 
In what follows, we  denote  by $G(x)=\{gx: g\in G\}$ the $G$-orbit of $x\in X$.

Recall that a topological space $Z$ is called {\em almost Hausdorff} if  every closed subset $Y$ of $Z$ contains 
a relatively open dense Hausdorff subset  $U\subseteq Y$, and $Z$ is called a T$_0$-space if for two points $y,z\in Z$ with 
$y\neq z$ at least one of these points is not in the closure  of the other.

\begin{definition}\label{def-regular}
A locally compact $G$-space $X$ is called {\em regular} if at least one of the following conditions hold:
\begin{enumerate}
\item \label{regular 1} For each $x\in X$ the canonical map $G/G_x\to G(x)$; $gG_x\mapsto gx$ is a homeomorphism
and the orbit space $G\backslash X$ is either almost Hausdorff or  second countable;
\item \label{regular 2} $G\backslash X$  is almost Hausdorff and $G$ is $\sigma$-compact;
\item \label{regular 3} $G$ and $X$ are second countable  and $G\backslash X$ is a T$_0$-space.
\end{enumerate}
\end{definition}

To help orient the reader, let us point out that it is shown in \cite{Rieffel}*{Proposition 7.1} that 
 \eqref{regular 2} $\Rightarrow$ \eqref{regular 1} and it follows from  \cite{Glimm}*{Theorem} (see Theorem \ref{thm-Glimm} below)
 that  all  three properties are equivalent if $X$ and $G$ are second countable.  For further background on such regularity properties   see \cite{Ech-reg} and \cite{Dana-book}*{Section 6.1}.

As an application of Theorem \ref{thm-Mats-2} and Corollary \ref{cor-weakam} we now prove
\begin{theorem}\label{thm-complete-reg}
Suppose that $G$ is an exact group, that $X$ is a regular locally compact $G$-space, and that $(A,\alpha)$ is an $X\rtimes G$-algebra. Then the following are equivalent:
\begin{enumerate}
\item \label{thm-complete-reg 1} $\alpha:G\to \Aut(A)$ is amenable;
\item \label{thm-complete-reg 2} For every $x\in X$ the  action $\alpha^x:G_x\to \Aut(A_x)$  is  amenable.
\end{enumerate}
\end{theorem}
 For references on  facts about induced representations of crossed products  as used in the proof below we refer to 
 \cite{CELY}*{Chapter 2} or \cite{Dana-book}.
 
\begin{proof} To see \eqref{thm-complete-reg 1} $\Rightarrow$ \eqref{thm-complete-reg 2} we first apply Proposition \ref{prop-restrict} to see 
that the restriction of $\alpha$  to $G_x$ is amenable for all $x\in X$ (note that closed  subgroups
of exact groups are exact by 
\cite{KW}*{Theorem 4.1}.)  Since  $A_x$ is a quotient of $A$ by the $G_x$\nb-invariant ideal $I_x$, it follows then from
Proposition \ref{prop-ext-amenable} that the resulting action on $A_x=A/I_x$ is amenable as well.

For \eqref{thm-complete-reg 2} $\Rightarrow$ \eqref{thm-complete-reg 1} we show that for every $G$-$C^*$-algebra $(B,\beta)$ the regular representation 
$$\Lambda: (A\otimes_\max B)\rtimes_\max G\to (A\otimes_\max  B)\rtimes _\red G$$
is an isomorphism. The result then follows from  Corollary \ref{cor-weakam}.

Indeed, if $(A,\alpha)$ is an $X\rtimes G$-algebra via the structure map $\Phi:C_0(X)\to Z\M(A)$, 
then $(A\otimes_\max B, \alpha\otimes \beta)$ is an 
$X\rtimes G$-algebra with respect to the structure map
$$\Phi\otimes 1: C_0(X)\to Z\M(A\otimes_\max B).$$
Moreover, it follows from the exactness of the maximal tensor product that 
 the fibre  $(A\otimes_\max B)_x$ is isomorphic to $A_x\otimes_\max B$ with action 
$(\alpha\otimes\beta)^x=\alpha^x\otimes \beta:G_x\to \Aut(A_x\otimes_\max B)$.
If $\alpha^x$ is amenable, the same is true for $\alpha^x\otimes \beta$ by Theorem \ref{thm-Mats-2}.
It follows that the $X\rtimes G$-algebra $(A\otimes_\max B, \alpha\otimes \beta)$ again satisfies 
all assumptions of the theorem.

It therefore suffices to show that, under the assumptions of the theorem, the maximal and reduced 
crossed products coincide. For  this it  suffices  to show that every primitive ideal of $A\rtimes_\max G$ 
contains the kernel  of the regular representation $\Lambda:A\rtimes_\max G\to A\rtimes_\red G$.
To see this, let $\varphi:\Prim(A)\to X$ be the continuous $G$-map corresponding to $\Phi$.
It follows then from \cite{Ech-reg}*{Proposition 3} that $\varphi$ is a complete regularization in the 
sense of \cite{Ech-reg}*{Definition 1}, which then implies (using \cite{Ech-reg}*{Proposition 2})
that  every primitive ideal $P\in \Prim(A\rtimes_\max G)$ can be realized as the kernel
of an induced representation $\Ind_{G_x}^G(\rho\rtimes u)$, where $(\rho, u)$ is the inflation of some irreducible 
representation of $(A_x, G_x, \alpha_x)$ to $(A, G_x, \alpha)$.
Since  $\alpha_x:G_x\to \Aut(A_x)$ is amenable, the representation 
$\rho\rtimes u$ is weakly contained in the inflation of the regular representation of $(A_x, G_x, \alpha^x)$ to $(A,G_x,\alpha)$,
 which, in turn,  is weakly contained in the regular representation 
 of $A\rtimes_\max G_x$. Since induction preserves  weak containment and since
the regular  representation of $A\rtimes_\max G_x$ induces to the regular representation $\Lambda$ of $A\rtimes_\max G$,
we see that $\Ind_{G_x}^G(\rho\rtimes u)$ is weakly contained in $\Lambda$, which just means that 
$P=\ker(\Ind_{G_x}^G(\rho\rtimes u))\supseteq \ker\Lambda$ and the 
result follows.
\end{proof}

A trivial action of $G$ on $X$
 is always regular, so the following is an immediate consequence of Theorem \ref{thm-complete-reg}.

\begin{corollary}\label{cor-field}
Suppose that $(A,\alpha)$ is an $X\rtimes  G$-algebra with $X$ a trivial $G$-space and that $G$ is  exact. 
Then $\alpha$ is amenable if and only if all fibre actions $\alpha^x:G\to\Aut(A_x)$ are amenable. \qed
\end{corollary}

If $H$ is a closed  subgroup of $G$ and $\alpha:H\to \Aut(A)$ is an action, then  
$$\Ind_H^G(A,\alpha):=\left\{f\in C_b(G, A): \begin{matrix} f(gh)=\alpha_{h^{-1}}(f(g)) \;\text{for all $g\in G, h\in H$,}\\
\text{and}\;
(gH\mapsto \|f(g)\|)\in C_0(G/H)\end{matrix}\right\}
$$ 
is a $G$-$C^*$-algebra with respect to the action 
$$\Ind\alpha:G\to \Aut(\Ind_H^G(A,\alpha)); \quad \Ind\alpha_g(f)(t):=f(g^{-1}t).$$
The system $(\Ind_H^G(A,\alpha), G, \Ind\alpha)$ is called the {\em system induced from $(A,H,\alpha)$ to $G$.}
Note that there is  a canonical $G$-equivariant structure  map 
$$\Phi:C_0(G/H)\to Z(\Ind_H^G(A,\alpha));\quad \big(\Phi(\varphi)f\big)(g)=\varphi(gH)f(g),$$
which gives $(\Ind_H^G(A,\alpha), G, \Ind\alpha)$ the structure of a $G/H\rtimes G$-algebra.
The evaluation maps $f\mapsto f(g)$ then identify  the fibres $(\Ind_H^G(A,\alpha))_{gH}$ with $A$
and the actions $(\Ind\alpha)^{gH}: G_{gH}=gHg^{-1}\to \Aut((\Ind_H^G(A,\alpha))_{gH})$  with 
$\alpha^g:gHg^{-1}\to\Aut(A)$ given by $\alpha^g_{ghg^{-1}}=\alpha_h$. Thus, as a direct corollary of 
Theorem \ref{thm-complete-reg} we  get

\begin{corollary}\label{cor-ind}
Let $H$ be a closed subgroup of the exact group $G$ and let $\alpha:H\to \Aut(A)$ be an action.
Then the induced action $\Ind\alpha:G\to \Aut(\Ind_H^G(A,\alpha))$ is amenable if and only if $\alpha$ is amenable. \qed
\end{corollary}

Note  that for discrete groups $G$ the above result has been shown by Anantharaman-Delaroche 
in \cite{Anantharaman-Delaroche:1987os}*{Th\'eor\`eme 4.6} without any exactness conditions on $G$.

 Before we state our next result, we need to recall 
a theorem of Glimm (\cite{Glimm}*{Theorem 1}):

\begin{theorem}[Glimm]\label{thm-Glimm}
Suppose that the second countable locally compact group $G$ acts on the {\em almost Hausdorff} second countable 
 locally compact space $X$. Then the following are equivalent:
 \begin{enumerate}
 \item \label{thm-Glimm 1}$G\backslash X$ is a T$_0$ space.
 \item \label{thm-Glimm 2}Each orbit $G(x)$ is locally closed, that is  $G(x)$ is relatively open in its closure.
 \item \label{thm-Glimm 3}For each $x\in X$ the canonical map 
 $$G/G_x\to G(x);\quad gG_x\to gx$$ 
 is a homeomorphism.
 \item \label{thm-Glimm 4}There  exists an increasing system $(U_\nu)$ of open $G$-invariant subsets of $X$ 
 indexed over the ordinal numbers $\nu$ such that 
 \begin{enumerate}
 \item $U_0=\emptyset$ and $X=U_{\nu_0}$ for some ordinal $\nu_0$;
 \item for every limit ordinal $\nu$ we have $U_\nu=\bigcup_{\mu<\nu}U_\mu$;   and 
\item  the orbit space $G\backslash(U_{\nu+1}\smallsetminus U_{\nu})$ is Hausdorff for all  $\nu$.  \qed
\end{enumerate}
 \end{enumerate}
 \end{theorem}
 
  Note that Glimm's original theorem lists a number of other equivalent statements, but the above  are  all we need.
 If $G\curvearrowright X$ satisfies the statements of Glimm's theorem, we say that 
 $X$ is a \emph{regular} $G$-space.  Item \eqref{thm-Glimm 1} in Glimm's theorem implies that this is compatible with 
  Definition \ref{def-regular} if $X$ is Hausdorff and $G$ and $X$ are second countable.
  
  The theorem applies in particular to actions $\alpha:G\to\Aut(A)$ of second countable groups on 
  separable type I $C^*$-algebras: in this situation the dual
 space $\widehat{A}$
   of unitary equivalence classes of irreducible $*$-representations of $A$ is almost Hausdorff and locally compact
  with respect to  the Jacobson (or Fell) topology as described in \cite{Dix}*{Chapter 3}.
If  $\alpha:G\to\Aut(A)$ is an action on a type I $C^*$-algebra, then there is a corresponding topological action 
 $G\curvearrowright \widehat{A}$ given by $(g,[\pi])\mapsto[\pi\circ \alpha_{g^{-1}}]$. 
Moreover,  for each $[\pi]\in\widehat{A}$ the action $\alpha$ 
induces an action $\alpha^\pi$ of the stabilizer $G_\pi$ on the algebra of compact operators
$\mathcal K(H_\pi)$ given by $\alpha^\pi_g=\Ad v_g$, where $v_g\in \U(H_\pi)$ is a choice of unitary
which implements the unitary equivalence $\pi\simeq \pi\circ\alpha_g$ (see for example \cite{CELY}*{Remark 2.7.28}).

\begin{theorem}\label{thm-type I}
Let $G$ be a second countable, exact, locally compact group.  Let $(A,G,\alpha)$ be a $G$-$C^*$-algebra such that $A$ is type I, and such that the induced action on $\widehat{A}$ is regular.
Then $\alpha$ is amenable if and only if for all $[\pi]\in \widehat{A}$ the stabilizer $G_\pi$ is an amenable group.
\end{theorem}
\begin{proof}
We first observe that it follows from part \eqref{thm-Glimm 3} of Glimm's theorem (Theorem \ref{thm-Glimm}), that for each $[\pi]\in \widehat{A}$ the orbit
$G([\pi])=\{ [\pi\circ \alpha_{g}]: g\in G\}$ is locally closed in $\widehat{A}$. This implies that there are 
$G$-invariant ideals $J\subseteq I\subseteq A$ such that $G([\pi])\cong \widehat{I/J}$:
just take  
$$J=\bigcap\{\ker{\rho}:[\rho]\in G([\pi])\}\quad\text{and}\quad 
I=\bigcap\left\{\ker\sigma: [\sigma]\in \overline{G([\pi])}\smallsetminus G([\pi])\right\}$$ 
and use the well-known 
correspondences between open (respectively closed) subsets of $\widehat{A}$ with the duals of ideals (respectively quotients)
of $A$ as explained in \cite{Dix}*{Chapter 3}. In what follows, we shall write $A_{G([\pi])}$ for this subquotient $I/J$.

Since by Proposition \ref{prop-ext-amenable} amenability passes to ideals and quotients,  it follows that 
amenability of $\alpha$ implies amenability of the induced action $\alpha^{G([\pi])}$ of $G$ on the subquotient $A_{G([\pi])}$.
Since $G([\pi])\cong \widehat{A_{G([\pi])}}$, it follows from \cite{Ech-ind}*{Theorem} that 
$(A_{(G[\pi])}, \alpha^{G([\pi])})$ is isomorphic to the induced system   $\big(\Ind_{G_\pi}^G(\mathcal K(H_\pi),\alpha^\pi), \Ind\alpha^\pi\big)$.
Thus, it follows from Corollary \ref{cor-ind} that $\alpha^{G([\pi])}$ is amenable if and only if the action 
$\alpha^\pi:G_\pi\to \Aut(\mathcal K(H_\pi))$ is amenable, which by Observation \ref{ob2} is equivalent to 
amenability of $G_{[\pi]}$.

So far we observed that under the assumptions of the theorem amenability of $\alpha$ implies amenability of $G_\pi$ for all $[\pi]\in \widehat{A}$. To see 
the converse, let $(U_\nu)$ be a system of $G$-invariant open subsets of $\widehat{A}$ over the ordinal numbers 
as in part \eqref{thm-Glimm 4} of Glimm's theorem (Theorem \ref{thm-Glimm}). Then for each $\nu$ there is a unique $G$-invariant ideal $I_\nu\subseteq A$ such that $U_\nu\cong\widehat{I_\nu}$.
Since $\widehat{A}=U_{\nu_0}=\widehat{I_{\nu_0}}$ for some ordinal number $\nu_0$, we  also have $A=I_{\nu_0}$.

We show by transfinite induction that, if all stabilizers $G_\pi$ are amenable, then 
 the  restrictions  $\alpha^\nu:G\to \Aut(I_\nu)$  of  $\alpha$ to the $G$-ideals $I_\nu$ are
amenable for all $\nu$. For $\nu=0$ we have $U_0=\emptyset$ and therefore $I_0=\{0\}$ and the result is clear.
So suppose now that $0<\nu$ is an ordinal number such that $\alpha^\mu:G\to \Aut(I_\mu)$ is amenable for all 
$\mu<\nu$. If $\nu$ is a limit ordinal, then $U_\nu=\bigcup_{\mu<\nu}U_\mu$ from which it follows that 
$I_\nu=\overline{\bigcup_{\mu<\nu}I_\mu}$. It follows  then from Proposition \ref{prop-inductive} that $\alpha^\nu:G\to \Aut(I_\nu)$
is amenable.

So assume now that $\nu=\mu+1$ for some ordinal $\mu$. By the conditions in item \eqref{thm-Glimm 4} of Glimm's theorem (Theorem \ref{thm-Glimm}) the orbit space
$X_\nu:=G\backslash(U_\nu\smallsetminus U_\mu)\cong G\backslash (\widehat{I_\nu/I_\mu})$ is Hausdorff.
Therefore the $G$-$C^*$-algebra $A_\nu:=I_\nu/I_\mu$ has the structure of an $X_\nu\rtimes G$-algebra
for the trivial $G$-space $X_\nu$, and one checks that the  fibre actions 
$\alpha^{G([\pi])}:G\to \Aut(A_{G([\pi])})$ at  orbits $G([\pi])\in X_\nu$ coincide with the actions 
$(A_{G([\pi])}, \alpha^{G([\pi])})$ as studied above.  As seen above, these fibre systems are amenable
if and only if the groups $G_\pi$ are amenable.  
Thus it follows  from Corollary \ref{cor-field} that amenability of $G_\pi$ for all $[\pi]\in \widehat{A}$ 
implies amenability of the action on $A_\nu =I_\nu/I_\mu$. By assumption, the action on 
$I_\mu$ is amenable as well. Hence Proposition \ref{prop-ext-amenable}  now implies amenability of
$\alpha^\nu:G\to \Aut(I_\nu)$. This finishes the proof.
\end{proof}

\section{Actions on type I $C^*$-algebras with Hausdorff spectrum}\label{sec:continuous-trace}

If $(A,\alpha)$ is a separable type I $G$-$C^*$-algebra with Hausdorff spectrum $\widehat{A}=X$ such that the action of the second countable exact group $G$ is regular in the sense of the previous section, then it is an easy consequence of Theorem \ref{thm-type I} that $\alpha$ is amenable if and only if the action on $C_0(X)$ is amenable. 
We will now show with different methods that  this result holds true 
 without any regularity conditions on the action.
Note that if $\widehat{A}=X$ is Hausdorff, then $A$ has 
a canonical $C_0(X)$-algebra structure via the identification $C_b(X)\cong Z(M(A))$ given by 
the Dauns-Hofmann theorem (see e.g.\ \cite{RW}*{Section A.3} for more details).

\begin{theorem}\label{thm-T2-dual}
Let $\alpha:G\to \Aut(A)$ be an action of a second countable locally compact group
on a separable type I $C^*$-algebra $A$ such that $X=\widehat{A}$ is Hausdorff. Then  $\alpha$ is amenable if and only if the  corresponding action 
on $C_0(X)$  is amenable.
\end{theorem}

Recall  that a $C^*$-algebra $A$ is called a {\em continuous-trace algebra} if 
it is type I with Hausdorff spectrum $\widehat{A}$ and such that for every $[\pi]\in \widehat{A}$ there
exists an open neighbourhood $U$ of $[\pi]$  in $\widehat{A}$ and a positive element 
$a\in A$ such that $\rho(a)$ is a projection of rank one for all $[\rho]\in U$ (see e.g.\ \cite{Dix}*{Proposition 4.5.3}).
In this case the  ideal 
$A_U\subseteq A$ satisfying $\widehat{A_U}=U$ is Morita equivalent to $C_0(U)$. We refer to \cite{Dix}*{Chapter 10} or 
\cite{RW}*{Chapter 5} for  detailed treatments of 
continuous-trace algebras.

The proof of Theorem \ref{thm-T2-dual} will  use the structure of the equivariant Brauer group as introduced in \cite{CKRW}, and which has  been used already  in the special case $X=\{\pt\}$ in Section \ref{sec:example} above. 
For this recall that if $X$ is a  paracompact locally compact  $G$-space, then the elements of the equivariant Brauer group 
$\mathrm{Br}_G(X)$ are the $X\rtimes G$-equivariant  Morita equivalence classes $[A,\alpha]$ 
of systems $(A,G,\alpha)$ 
in which $A$ is a separable continuous-trace $C^*$-algebra with spectrum $\widehat{A}\cong X$ and $\alpha:G\to\Aut(A)$ is an
action which covers the given action of $G$ on $X$ via the identification $X\cong \widehat{A}$.
It is shown in \cite{CKRW}*{Theorem 3.6} that $\mathrm{Br}_G(X)$ becomes an abelian group 
if we define multiplication of two elements $[A,\alpha]$ and $[B,\beta]$ by 
$$[A,\alpha]\cdot [B,\beta]=[A\otimes_X B, \alpha\otimes_X\beta],$$
where $A\otimes_XB$ denotes the $C_0(X)$-balanced tensor product of $A$ with $B$. This can be defined as
the quotient
$(A\otimes B)/J_X$, where $J_X$ denotes the ideal in $A\otimes B$ generated by all elements of the form
$$af\otimes b-a\otimes fb;\quad a\in A, ~b\in B, ~f\in C_0(X).$$
The neutral element is given by the class
$[C_0(X), \tau]$, where $\tau:G\to \Aut(C_0(X))$ is the action associated to the given $G$-action on $X$.
If $X=\{\pt\}$, this just boils down to the group of Morita equivalence classes of actions on compact operators on Hilbert spaces as 
used in Section \ref{sec:example} above.

\begin{lemma}\label{lem-conttrace}
The statement of Theorem \ref{thm-T2-dual} holds true if, in addition, $A$ is a continuous-trace algebra.
\end{lemma}
\begin{proof}
Theorem \ref{thm-Mats-2} implies that if $\alpha:G\to \Aut(A)$ is amenable and $\beta:G\to\Aut(B)$ is any action, 
then the diagonal action $\alpha\otimes\beta:G\to \Aut(A\otimes_{\mathrm{\max}} B)$  is amenable. Since 
amenability always passes to  quotients by $G$-invariant ideals (Proposition \ref{prop-ext-amenable}) it then follows that amenability of $\alpha$ also passes to 
diagonal actions on balanced tensor products $A\otimes_XB$. 
 Thus, if $\alpha:G\to \Aut(A)$ is an amenable action on 
the continuous-trace algebra $A$ with spectrum $X=\widehat{A}$, and  if $(B,G,\beta)$ is a representative of the inverse class $[A,\alpha]^{-1}$ in $\mathrm{Br}_G(X)$, 
then the amenable action $(A\otimes_XB, G,\alpha\otimes_X\beta)$ is equivariantly Morita equivalent to $(C_0(X), G,\tau)$.
Since amenability is stable under $G$-equivariant Morita equivalences by Proposition \ref{pro:Morita},
it follows that $\tau:G\to \Aut(C_0(X))$ is amenable.

For the converse, assume that $\tau:G\to \Aut(C_0(X))$ is amenable. Then 
 $(A,\alpha)\cong (C_0(X)\otimes_XA, \tau\otimes_X\alpha)$
 is amenable as  well.
\end{proof}

The following lemma, which is possibly well known, provides a tool to reduce the proof of Theorem \ref{thm-T2-dual} 
to the case of continuous-trace algebras. 

\begin{lemma}\label{lem-conttrace-ideal}
Let $(A,\alpha)$ be a separable type I $G$-$C^*$-algebra such that $\widehat{A}$ is Hausdorff. Then there exists a non-zero closed $G$-invariant continuous-trace ideal $I$ of $A$.
\end{lemma}
\begin{proof} It follows from \cite{Dix}*{Theorem 4.5.5} that there exists a non-zero con\-tin\-uous-trace closed ideal $J$ of $A$.
Let $V=\widehat{J}\subseteq \widehat{A}$ and let $U=\bigcup_{g\in G} g\cdot V$. Then $U$ is a non-empty $G$-invariant open subset of 
$\widehat{A}$. Let $I\subseteq A$ denote the $G$-invariant ideal such that $U=\widehat{I}\subseteq \widehat{A}$.
Then $\widehat{I}$ is Hausdorff. Thus in order to check that $I$ is a continuous-trace algebra, we only need to check 
that  for each $[\pi]\in \widehat{I}$ there exists an open neighbourhood  $W$ of $[\pi]$ and a positive element $a\in I$ such that 
$\rho(a)$ is a rank one projection for all $[\rho]\in W$. But this follows easily from the fact that $[\pi]\in g\cdot V=\widehat{\alpha_g(J)}$
for some $g\in G$ and that $\alpha_g(J)\subseteq I$ has continuous trace for all $g\in G$.
\end{proof}

\begin{proof}[Proof of Theorem \ref{thm-T2-dual}]
It follows from Lemma \ref{lem-conttrace-ideal} together with transfinite induction that there exists an increasing  system of 
$G$-invariant ideals $(I_\nu)_\nu$,  indexed by the ordinal numbers,  such that $A=I_{\mu_0}$ for some ordinal $\mu_0$,   $I_\mu=\overline{\bigcup_{\nu<\mu}I_\nu}$ for every limit ordinal $\mu$, and  for every $\nu$ the quotient $I_{\nu+1}/I_\nu$ is a contin\-uous-trace algebra.  Likewise, if $U_\nu=\widehat{I_\nu}$, $(C_0(U_\nu))_\nu$ is a system of ideals in $C_0(X)$ 
with the same properties. It follows then from a combination of Proposition \ref{prop-ext-amenable} with Proposition \ref{prop-inductive}
that $\alpha:G\to\Aut(A)$ is amenable if and only if the actions on the quotients $I_{\nu+1}/I_\nu$ are all amenable
and, similarly,  that the action on $C_0(X)$  is amenable if and only if the actions on
 $C_0(U_{\nu+1}\smallsetminus U_\nu)=C_0(U_{\nu+1})/C_0(U_\nu)$ are all amenable.
 Since $U_{\nu+1}\smallsetminus U_\nu=(I_{\nu+1}/I_\nu)^\dach$, the result now follows from Lemma \ref{lem-conttrace}.
 \end{proof}

Since by \cite{Bearden-Crann}*{Corollary 4.14}, an action $\alpha:G\to \Aut(C_0(X))$ is amenable if and only if it is strongly amenable, we get the following corollary:

\begin{corollary}\label{cor-strong-typeI}
Let $(A,\alpha)$ be a separable type I $G$-$C^*$-algebra such that $\widehat{A}$ is Hausdorff, and $G$ is second countable.
 Suppose also that $X=\widehat{A}$ is Hausdorff.  Then $\alpha$ is amenable if and only if $\alpha$ is strongly amenable. \qed
\end{corollary}

\begin{remark}
Notice that the results of this section have no analogue if one replaces amenability by the weaker notion of commutant amenability (or weak containment): it can happen, for instance, that a continuous trace $G$-$C^*$-algebra $A$ is commutant amenable but the induced action on $C_0(\widehat{A})$ is not commutant amenable. This follows from our counterexamples in Section~\ref{sec:example} where $A$ is the algebra of compact operators on a Hilbert space, so continuous trace.
\end{remark}

\chapter{Regularity properties}\label{sec:G-WEP}

In this chapter, we collect together some miscellaneous results about regularity properties.  

In Section \ref{sec:pass} we show that various $C^*$-algebraic regularity properties -- for example, nuclearity -- pass from $A$ to $A\rtimes_{\max} G\cong A\rtimes_\red G$ if the action is amenable.  The key tools here are standard facts about the interactions between crossed products and tensor products, and Proposition \ref{prop-general} which tells us that the reduced and maximal crossed products agree for an amenable action.

In Section \ref{sec:inj} we characterize exactness of a group in terms of amenability of actions on injective $C^*$-algebras; this generalizes an earlier result of ours \cite{Buss:2019}*{Theorem 8.3} from actions of discrete groups to actions of general locally compact groups.  

Finally, in Section \ref{sec:gwep} we introduce a property called the \emph{continuous $G$-WEP}.  This is an equivariant analogue of the WEP that seems appropriate to locally compact groups.  We show that the continuous $G$-WEP is closely related to amenability and the weak containment property.

\section{Properties passing to the crossed product by an amenable action}\label{sec:pass}

In this section, we show that amenable actions have good permanence properties with respect to nuclearity, exactness, the WEP, the LLP, and the LP: roughly, we show that an amenable $G$-$C^*$-algebra $A$ has any of these properties if and only if $A\rtimes_{\red}G$ does. 

With the exception of the LP, we state all results only in terms of the reduced crossed product for brevity, but note that 
 amenability implies $A\rtimes_{\max}G=A\rtimes_\red G$ by Proposition \ref{prop-general}.
Indeed, at least nuclearity, exactness, the WEP, and the LLP  pass from $A\rtimes_\red G$ to $A$ for arbitrary actions. 
We are grateful to \mcomment{one of the referees} for pointing out to us an easy argument for this fact,
which replaces a  more complicated one we used in an earlier version. 
We do not know whether the LP always passes from $A\rtimes_\red G$ to $A$,\footnote{This is true in many cases: for instance, if $G$ is discrete, then the LP passes from $A\rtimes_\red G$ to $A$ as it is the image of a conditional expectation.} but we shall see below that 
(for separable systems $(A,G,\alpha)$) it always passes from the maximal crossed product $A\rtimes_\max G$ to $A$.

We first start with a lemma.  For the statement, let us say that a pair $(A,D)$ of $C^*$-algebras is nuclear if the canonical map $A\otimes_{\max} D\to A\otimes D$ is an isomorphism. The proof of the first statement in the lemma was suggested to us by \mcomment{one of the referees}.

\begin{lemma}\label{nuc pair}
Let $(A,\alpha)$ be a $G$-$C^*$-algebra, and let $D$ be a $C^*$-algebra.  If $(A\rtimes_\red G, D)$ is a nuclear pair, then so is $(A,D)$. Conversely, if $(A,\alpha)$ is amenable and $(A,D)$ is a nuclear pair, then so is $(A\rtimes_\red G, D)$. 
\end{lemma}

\begin{proof}
For every $C^*$-algebra $D$ equipped with the trivial $G$-action, there is a canonical sequence of surjective $*$-homomorphisms
$$(A\rtimes_\red G)\otimes_\max D\onto (A\otimes_\max D)\rtimes_\red G\onto (A\otimes D)\rtimes_\red G\cong (A\rtimes_\red G)\otimes D,$$
where the first map exists by the universal property of ``$\otimes_\max$'' applied to the  canonical  homomorphisms
of $A\rtimes_\red G$ and $D$ into $\M((A\otimes_\max D)\rtimes_\red G)$.
The composition  is the canonical quotient map $(A\rtimes_\red G)\otimes_\max D\onto (A\rtimes_\red G)\otimes D$.
Thus if $(A\rtimes_\red G, D)$ is a nuclear pair, it follows that 
$(A\otimes_\max D)\rtimes_\red G\onto (A\otimes D)\rtimes_\red G$ is injective, which then forces $A\otimes_\max D\to A\otimes D$ to be injective as well, that is, $(A,D)$ is a nuclear pair.

Suppose now that $(A,\alpha)$ is amenable. There is then a sequence of canonical $*$-homomorphisms 
\begin{align*}
(A\rtimes_\red G)\otimes_\max D&\stackrel{(1)}\cong
(A\rtimes_\max G)\otimes_\max D\\
& \stackrel{(2)}{\cong} (A\otimes_\max D)\rtimes_\max G\\
 &\stackrel{(3)}{\cong} (A\otimes_\max D)\rtimes_\red G \\ 
 & \stackrel{(4)}{\to} (A\otimes D)\rtimes_\red G  \\
& \stackrel{(5)}{\cong} (A\rtimes_\red G)\otimes D
\end{align*}
that we now explain.  The isomorphisms (1) and (3) come from  Proposition \ref{prop-general} and amenability of $(A,\alpha)$, which also implies amenability of $(A\otimes_\max D, \alpha\otimes\id)$ (see Theorem \ref{thm-Mats-2}).
The isomorphisms (2) and (5) are the canonical untwisting isomorphisms (see for example \cite{CELY}*{Lemma 2.4.1}).  
  The homomorphism labeled (4) is induced from the canonical quotient map $A\otimes_\max D\to A\otimes D$.  
  Thus if $(A,D)$ is a nuclear pair, then the composition of all the maps (1)-(5) above is an isomorphism,  which then implies  that $(A\rtimes_\red G,D)$ is a nuclear pair.
\end{proof}

\begin{theorem}\label{prop:Nuclear-Pairs}
Let $(A,\alpha)$ be a $G$-$C^*$-algebra. Consider the statements below.
\begin{enumerate}
\item \label{prop:Nuclear-Pairs 1}$A$ is nuclear (respectively~has the LLP, has the WEP).
\item \label{prop:Nuclear-Pairs 2}$A\rtimes_\red G$ is nuclear (respectively~has the LLP, has the WEP).
\end{enumerate}
Then \eqref{prop:Nuclear-Pairs 2} $\Rightarrow$ \eqref{prop:Nuclear-Pairs 1} in general, and if $(A,\alpha)$ is amenable then \eqref{prop:Nuclear-Pairs 1} $\Rightarrow$ \eqref{prop:Nuclear-Pairs 2}.
\end{theorem}

\begin{proof}
A $C^*$-algebra is nuclear if and only if $(A,D)$ is a nuclear pair for any $C^*$-algebra $D$.  Hence the statement on nuclearity is immediate from Lemma \ref{nuc pair}.  The  assertions on the LLP and the WEP are also direct consequences of Lemma \ref{nuc pair}: by \cite{Brown:2008qy}*{Corollary 13.2.5}  the LLP for a $C^*$-algebra $A$ is equivalent to the statement that $\Bd(\ell^2)$ forms a  nuclear pair with $A$, and the WEP is equivalent to the statement that  $C^*(\mathbb{F})$ forms  a nuclear pair with $A$, where $\mathbb{F}$ denotes a free group on countably infinitely many generators.
\end{proof}

\begin{remark}
If $G$ is discrete and $(A,\alpha)$ is a $G$-$C^*$-algebra, Anantharaman-Delaroche shows in \cite{Anantharaman-Delaroche:1987os} that 
$A\rtimes_{\red} G$ is nuclear if and only if $A$ is  nuclear and $\alpha$ is amenable.
An analogous result cannot be true for actions of general locally compact groups, since 
by a famous  result of Connes \cite{Connes:1976fj}*{Corollary 6.9 (c)} we know that the reduced group algebra $C_{\red}^*(G)=\C\rtimes_\red G$ of any second countable, {\em connected}, locally compact group $G$  is nuclear; however, the trivial action of $G$ on $\C$ is amenable if and only if $G$ is amenable. 
Note that there are many connected locally compact groups that are not amenable (e.g.\ $\mathrm{SL}(2,\R)$). 
\end{remark}

\begin{theorem}\label{thm:LP}
Let $(A,\alpha)$ be a $G$-$C^*$-algebra with $A$ separable and $G$ second countable.  Consider the statements
\begin{enumerate}
\item \label{thm:LP 1}$A$ has the LP.
\item \label{thm:LP 2}$A\rtimes_\max G$ has the LP.
\end{enumerate}
Then \eqref{thm:LP 2} $\Rightarrow$ \eqref{thm:LP 1} in general, and if $(A,\alpha)$ is amenable then \eqref{thm:LP 1} $\Rightarrow$ \eqref{thm:LP 2}.
\end{theorem}

\begin{proof}
We use a recent characterization of the LP due to Pisier. Pisier shows in \cite{Pisier}*{Theorem 0.2} that a separable 
$C^*$-algebra $A$ has the LP if and only if for any index set $I$ and any collection of $C^*$-algebras $\{D_i:i\in I\}$ 
the canonical $*$-homomorphism 
\begin{equation}\label{eq-LP1}\ell^\infty(\{D_i\})\otimes_\max A\to \ell^\infty(\{D_i\otimes_\max A\})\end{equation}
is faithful, where $\ell^\infty(\{D_i\})$ denotes the $C^*$-algebra of bounded $I$-tuples $(d_i)\in \prod_{i\in I}D_i$ equipped with the supremum norm.  We now look at the map
\begin{equation}\label{eq-LP}
 \ell^\infty(\{D_i\})\otimes_\max (A\rtimes_{\max} G)\to \ell^\infty\big(\{D_i\otimes_\max(A\rtimes_{\max} G)\}\big).\end{equation}
Let $\{(B_i, \beta_i):i\in I\}$ be any collection of $G$-$C^*$-algebras, and let $\ell^\infty(\{B_i\})_c$ denote the algebra of continuous elements in $\ell^\infty(\{B_i\})$ with respect to the component-wise action $\beta_g(b_i):=(\beta_{i,g}(b_i))$.  Let
\begin{equation}\label{eq-LP2}
\varphi:\ell^\infty(\{B_i\})_c\rtimes_\red G\into \ell^\infty(\{B_i\rtimes_\red G\})
\end{equation}
be defined on $C_c(G, \ell^\infty(\{B_i\})_c)$ by evaluation at each $i$.  Note that $\varphi$ is injective: this follows as if $\pi_i:B_i\to \Bd(H_i)$ is a faithful representation with induced regular representation $\widetilde{\pi_i}\rtimes \lambda:B_i\rtimes_\red G\to \Bd(H_i\otimes L^2(G))$ (see line \eqref{eq-regular}) and if $\pi:=\bigoplus_{i\in I} \pi_i$ with induced regular representation $\widetilde{\pi}\rtimes\lambda$, then 
$$
\Big(\bigoplus_{i\in I}(\widetilde{\pi_i}\rtimes\lambda)\Big)\circ \varphi=\widetilde{\pi} \rtimes\lambda
$$
as maps $\ell^\infty(\{B_i\})_c\rtimes_\red G\to \Bd\big(\bigoplus_{i\in I}(H_i\otimes L^2(G))\big)$, and the right hand side is injective.

Using this, if $(A,\alpha)$ is amenable, 
we now observe that the map in (\ref{eq-LP}) can be written as the composition of canonical maps
\begin{align*}
 \ell^\infty(\{D_i\})\otimes_\max (A\rtimes_\max G)&\cong  \big(\ell^\infty(\{D_i\})\otimes_\max A\big)\rtimes_\max G\\
& =\big(\ell^\infty(\{D_i\})\otimes_\max A\big)\rtimes_\red G\quad\quad\text{(since $\alpha$ is amenable)}\\
 &\stackrel{\text{(*)}}{\to} \ell^\infty\big(\{D_i\otimes_\max A\}\big)_c\rtimes_\red G\quad\quad\; \text{(induced from (\ref{eq-LP1}))}\\
&\into \ell^\infty\big(\{(D_i\otimes_\max A)\rtimes_\red G\}\big)\quad\quad \text{(since (\ref{eq-LP2}) is faithful)}\\
&=  \ell^\infty\big(\{(D_i\otimes_\max A)\rtimes_\max G\}\big) \quad\quad\text{(since $\alpha$ is amenable)}\\
 &\cong  \ell^\infty\big(\{D_i\otimes_\max (A\rtimes_\max G)\}\big).
\end{align*}
Thus $A\rtimes_\max G$ has the LP if and only if the map (*) is faithful, which holds true if and only if 
the map $\ell^\infty(\{D_i\})\otimes_\max A\to \ell^\infty\big(\{D_i\otimes_\max A\})_c$ is faithful, and hence if and only if the map in (\ref{eq-LP1}) is faithful.
By Pisier's theorem this holds true if and only if $A$ has the LP. 
Thus we proved that  \eqref{thm:LP 1} $\Leftrightarrow$  \eqref{thm:LP 2} if $(A,\alpha)$ is amenable.

In general\footnote{We thank \mcomment{one of the referees} for suggesting a simpler argument in this case.}, we can factor the map in (\ref{eq-LP}) as 
\begin{equation}\label{eq-LPsplit}
\begin{split}
 \ell^\infty(\{D_i\})\otimes_\max (A\rtimes_\max G)&\cong  \big(\ell^\infty(\{D_i\})\otimes_\max A\big)\rtimes_\max G\\
 &\stackrel{\text{(**)}}{\to}  \ell^\infty\big(\{(D_i\otimes_\max A\})_c\rtimes_\max G\big)\\
 &\stackrel{\text{(***)}}{\to} \ell^\infty\big(\{(D_i\otimes_\max A)\rtimes_\max G\}\big)\\
 &\cong  \ell^\infty\big(\{D_i\otimes_\max (A\rtimes_\max G)\}\big),
\end{split}
\end{equation}
where the map (**) is induced by the canonical $G$-map 
\begin{equation}\label{cont pisier}
\ell^\infty(\{D_i\})\otimes_\max A\to \ell^\infty\big(\{D_i\otimes_\max A\})_c,
\end{equation}
and the map (***) is induced by the coordinate projections 
$$\ell^\infty\big(\{(D_i\otimes_\max A\})_c\rtimes_\max G\to (D_i\otimes_\max A)\rtimes_\max G.$$
  Using Pisier's theorem, if $A\rtimes_{\max} G$ has the LP, then the composition is injective, whence the map (**) is injective.  
   This implies that the map in line \eqref{cont pisier} is injective, which implies that the map 
   in line \eqref{eq-LP1} is also injective, so we are done by Pisier's theorem again. 
 \end{proof}

\begin{proposition}\label{prop:ex}
Let $(A,G,\alpha)$ be an amenable $G$-$C^*$-algebra. Then $A$ is exact if and only if $A\rtimes_\red G$ is exact. 
\end{proposition}

\begin{proof}
Assume that $A$ is exact and let $J\into B\onto C$ be a short exact sequence of $C^*$-algebras equipped with the trivial action.  We thus get a commutative diagram
{\small $$
\begin{CD}
0 @>>> (A\otimes J)\rtimes_\max G @>>>  (A\otimes B)\rtimes_\max G @>>>  (A\otimes C)\rtimes_\max G @>>>0\\
@. @V\cong VV  @V\cong VV @VV\cong V @.\\
0 @>>> (A\otimes J)\rtimes_\red G @>>>  (A\otimes B)\rtimes_\red G @>>>  (A\otimes C)\rtimes_\red G @>>>0\\
@. @V\cong VV  @V\cong VV @VV\cong V @.\\
0 @>>> (A\rtimes_\red G)\otimes J @>>>  (A\rtimes_\red G)\otimes B@>>> (A\rtimes_\red G)\otimes C @>>>0\\
\end{CD}
$$}
for which the upper vertical arrows are isomorphisms by Theorem \ref{thm-Mats-2} and Proposition \ref{prop-general}, and the lower vertical arrows are the canonical untwisting isomorphisms (see \cite{CELY}*{Lemma 2.4.1}).
 Hence the bottom line is exact if and only if the top line is.  The former holds for any short exact sequence if and only if $A\rtimes_\red G$ is exact, and the latter holds for any short exact sequence if and only if $A$ is exact (as $\rtimes_\max G$ preserves short exact sequences -- see for example \cite{CELY}*{Proposition 2.4.8}).   
\end{proof}

\begin{remark}\label{rem-coaction}
It has been shown by Ng in  \cite{Ng}*{Corollary 4.6} that exactness of $A\rtimes_\red G$ implies exactness of $A$ for every $G$-$C^*$-algebra $(A,\alpha)$ without any amenability conditions. Note that exactness of $A\rtimes_\max G$ always implies exactness of $A\rtimes_\red G$ (and hence exactness of $A$), since exactness is inherited  by  quotients (see \cite{Brown:2008qy}*{Corollary 9.4.3}). 
\end{remark}

\section{Characterizing exactness via actions on $G$-injective algebras}\label{sec:inj}

In this section we use amenable actions on $G$-injective $C^*$-algebras to characterize exactness, generalizing our earlier result  \cite{Buss:2019}*{Theorem 8.3} for discrete groups to the locally compact case. 
Recall that a $G$-$C^*$-algebra $A$ is \emph{$G$-injective} if for any commutative diagram
$$
\xymatrix{ C \ar@{-->}[dr] &  \\ B \ar[u] \ar[r] & A }
$$
where the solid horizontal arrow is an equivariant ccp map and the vertical arrow is a ($*$-homomorphic) $G$-embedding, the diagonal arrow can be filled in by an equivariant ccp map.  Using for example \cite{Buss:2018nm}*{Corollary 2.4}, one can see that $G$-injectivity of $A$ is equivalent to the following formally weaker property: for 
every $G$-embedding $\varphi:A\into B$ of $A$ into a
$G$-$C^*$-algebra $B$ there exists a ccp $G$-map $\psi:B\to A$ such that $\psi\circ \varphi=\id_A$.   The class of $G$-injective $C^*$-algebras was introduced (in a more general setting) and extensively studied by Hamana \cite{Hamana:1985,Hamana:2011}.  

The property of being a $G$-injective $C^*$-algebra is a very strong one, but examples exist for any $G$:  for example, $\contub(G)$ is always $G$-injective by \cite{Buss:2018nm}*{Proposition 2.2} (this is also implicit in \cite{Hamana:2011}).  The following result is thus a generalization of the theorem of Brodzki, Cave, and Li \cite{Brodzki-Cave-Li:Exactness} and Ozawa and Suzuki \cite{Ozawa-Suzuki}*{Proposition 2.5} that exactness of $G$ is equivalent to amenability of its canonical action on $\contub(G)$ (which is used in our proof).   It also generalizes a result of Kalantar and Kennedy \cite{Kalantar:2014sp}*{Theorem 1.1} characterizing exactness of a discrete group in terms of amenability of the action of a  group on its Furstenberg boundary. 

\begin{theorem}\label{inj the}
Let $G$ be a locally compact group. Then the following are equivalent:
\begin{enumerate}
\item \label{inj the 1}$G$ is exact;
\item \label{inj the 2}every $G$-injective $G$-$C^*$algebra $(A,\alpha)$ is strongly amenable;
\item \label{inj the 3}there exists a non-zero strongly amenable  $G$-injective $G$-$C^*$-algebra  $(A,\alpha)$.
\end{enumerate}
\end{theorem}
\begin{proof} For \eqref{inj the 1} $\Rightarrow$ \eqref{inj the 2}, recall that \cite{Ozawa-Suzuki}*{Proposition 2.5} gives that $G$ is exact if and only if the 
translation action of $G$ on $\contub(G)$ is strongly amenable. 
Let $(A,\alpha)$ be $G$-injective. Then \cite{Buss:2019}*{Lemma 4.3} implies that  
$A$ is unital. Consider the diagonal action of  $G$  on $\contub(G)\otimes A$
and define $\iota: A\into \contub(G)\otimes A$ by $\iota(a)=1\otimes a$. 
By $G$\nb-injectivity of $A$, there exists a ucp $G$-map $\varphi:\contub(G)\otimes A \to A$
such that $\varphi\circ \iota(a)=a$ for all $a\in A$. The restriction of $\varphi$ to $\contub(G)\cong  \contub(G)\otimes 1_A\subseteq \contub(G)\otimes A$ then gives a ucp $G$-map $\Phi:\contub(G)\to A$. 
Since $A$ lies in the multiplicative domain of $\varphi$, it follows that $\Phi$ takes its values in $Z(A)$.
Thus, if $(\theta_i)$ is a net of compactly supported positive type functions as in Definition \ref{def-amenable (SA)} which establishes
strong amenability of the translation action on $\contub(G)$, then the net $(\Phi\circ \theta_i)$ 
establishes strong amenability of $(A,\alpha)$, which implies \eqref{inj the 2}.

\eqref{inj the 2} $\Rightarrow$ \eqref{inj the 3} follows as $G$-injective $C^*$-algebras exist for any $G$  (for example, $\contub(G)$).
For \eqref{inj the 3} $\Rightarrow$ \eqref{inj the 1} let $(A,\alpha)$ be a $G$-injective, strongly amenable $G$-$C^*$-algebra.
Since $A$ is unital, it follows that $\alpha$ restricts to a strongly amenable action of $G$ on the (unital)
centre $Z(A)$. The existence of a strongly amenable action on a unital commutative $G$-$C^*$-algebra implies that $G$ is exact by \cite{Anantharaman-Delaroche:2002ij}*{Theorem 7.2}.
\end{proof} 

\begin{remark}\label{ex unital}
In the spirit of the above result, let us point out that Ozawa and Suzuki show in \cite{Ozawa-Suzuki}*{Corollary 3.6} that a locally compact group admits an amenable action on a unital $C^*$-algebra if and only if it is exact (the authors previously established this for discrete groups \cite{Buss:2019}*{Corollary 6.2}).
\end{remark}

\section{The continuous $G$-WEP}\label{sec:gwep}

In this section we introduce a property called the continuous $G$-WEP for locally compact groups, and relate it to amenability and weak containment.  

\begin{definition}\label{def-weakGWEP}
A $G$-$C^*$-algebra $(A,\alpha)$ has the {\em continuous $G$-WEP} if 
for every $G$-embedding $\varphi:A\into B$ there exists a $G$-equivariant
ccp map $\psi:B\to A_\alpha''$ such that $\psi\circ \varphi=i_A$, where 
$i_A:A\to A_\alpha''\subseteq (A\rtimes_\max G)^{**}$ denotes the canonical inclusion.
\end{definition}

Recall from \cite{Buss:2018nm}*{Definition 3.9} that a $G$-$C^*$-algebra $A$ has the 
{\em $G$-equivariant weak expectation property} ($G$-WEP) if for every  $G$-embedding $\varphi:A\into B$ 
there exists a $G$-equivariant ccp map $\psi:B\to A^{**}$ such that 
$\psi\circ \varphi=\iota$, where $\iota:A\to A^{**}$ denotes the canonical 
inclusion. The continuous $G$-WEP is introduced here as it seems more appropriate to non-discrete groups.  In fact, however, since our first draft of this paper Bearden and Crann \cite{Bearden-Crann1}*{Proposition 4.5} showed that there is no difference between the $G$-WEP and continuous $G$-WEP, so either could be used in any of the statements we give here. 

\begin{proposition}\label{thm-GWEP}
Let $(A,\alpha)$ be a $G$-$C^*$-algebra and consider the statements:
\begin{enumerate}
\item \label{thm-GWEP 1}$(A,\alpha)$ has the continuous $G$-WEP;
\item \label{thm-GWEP 2}$(A, \alpha)$ is amenable.
\end{enumerate}
Then \eqref{thm-GWEP 1} $\Rightarrow$ \eqref{thm-GWEP 2} if $G$ is exact and \eqref{thm-GWEP 2} $\Rightarrow$ \eqref{thm-GWEP 1} if $A$ has the WEP. 
In particular, if $G$ is exact and $A$ has the WEP, then both statements are equivalent.
\end{proposition}
\begin{proof} It follows from an easy adaptation of \cite{Buss:2019}*{Lemma 7.9} that if $(A,\alpha)$ has the continuous $G$-WEP, then for each unital 
$G$-$C^*$-algebra $C$ there exists a ucp $G$-map $\psi:C\to Z(A_\alpha'')$. This applies in particular to $C=\contub(G)$ with translation action. If $G$ is exact, then the action on $\contub(G)$ is amenable by \cite{Ozawa-Suzuki}*{Proposition 2.5}, and this implies amenability of $(A,\alpha)$.

Suppose conversely that $A$ has the WEP and $\alpha$ is amenable. The WEP implies that for any embedding $\varphi:A\into B$ and any faithful representation $\pi:A\to \Bd(H)$ there exists a ccp map $\psi:B\to \pi(A)''$ such that $\psi\circ \varphi$ is the canonical inclusion $A\into \pi(A)''$.  In particular, for each $G$-embedding $\varphi:A\into B$ there exists a ccp map $\psi:B\to A_\alpha''$ such that $\psi\circ \varphi=i_A$.  Our goal is to replace $\psi$ by a $G$-equivariant ccp map with   the same property.

Recall from Definition \ref{l2wgm} that for a $G$-von Neumann algebra $M\subseteq \Bd(H)$ with action $\sigma:G\to\Aut(M)$, we define a Hilbert $M$-module $L_w^2(G,M)$ to be the weak closure of $C_c(G, M)$ inside $\Bd(H, L^2(G, H))$, where we view $\xi\in C_c(G,M)$ as acting via
$$H\ni v\mapsto \xi\cdot v\in L^2(G,H) \quad\text{given by} \quad (\xi\cdot v)(g)=\xi(g)v.$$
For each $b\in B$, define $m_b\in L^\infty(G,A_\alpha'')$ by
$$
m_b:G\to A_\alpha'',\quad g\mapsto \alpha''_g(\psi(\beta_{g^{-1}}(b))),
$$
which we consider as an adjointable operator on $L^2_w(G,A_\alpha'')$ as in the discussion below Definition \ref{l2wgm}.  Let $(\xi_i)_{i\in I}$ be a net in 
$C_c(G, Z(A_\alpha'')_c)\subseteq L^2_w(G,A_\alpha'')$ with the properties from item \eqref{amen 3} of Proposition \ref{prop:Amenability-conditions}.
For each $i$, define a map
$$
T_i:B\to A_\alpha'',\quad b\mapsto \braket{\xi_i}{m_b\xi_i}_{A_\alpha''}.
$$
One then checks that the net $(T_i)$ consists of ccp maps.  After passing to a subnet we may assume that $(T_i)$ has a pointwise ultraweak limit (see  \cite{Brown:2008qy}*{Theorem 1.3.7}), which is also a ccp map $T:B\to A_\alpha''$.  We claim that this limit has the right properties.

First, let us check that if $a$ is an element of $A$, then $T(\varphi(a))=a$.  
Indeed, in this case $m_a$ is just the operator of left-multiplication by $a$, and so we have 
$$
T_i(a)=\braket{\xi_i}{a\xi_i}_{A_\alpha''}=\int_G\xi_i(g)^*a\xi_i(g) \dd g.
$$
for all $i$.  As $\xi_i$ takes values in $Z(A_\alpha'')$, this just equals $\braket{ \xi_i}{\xi_i}_{A_\alpha''} a$, however, which converges ultraweakly to $a$ as $i$ tends to infinity.  

It remains to check that $T$ is equivariant.  Let $b\in B$ and $h\in G$.  Then 
\begin{align*}
T_i(\beta_h(b))=\braket{\xi_i}{m_{\beta_h(b)}\xi_i}_{A_\alpha''}=\int_G \xi_i(g)^*\alpha_g''\big(\psi(\beta_{g^{-1}h}(b))\big)\xi_i(g) \dd g.
\end{align*}
Replacing $g$ by $hg$, this becomes
\begin{align*}
\int_G  \xi_i(hg)^*\alpha_{hg}''&\big(\psi(\beta_{g^{-1}}(b))\big)\xi_i(hg) \dd g\\
&=\alpha_h''\Big(\int_G (\lambda^{\alpha''}_{h^{-1}}\xi_i)(g)^*\alpha_{g}''\big(\psi(\beta_{g^{-1}}(b))\big)(\lambda^{\alpha''}_{h^{-1}}\xi_i)(g) \dd g\Big) \\ 
& = \alpha_h''(\braket{ \lambda^{\alpha''}_{h^{-1}}\xi_i}{m_b(\lambda^{\alpha''}_{h^{-1}}\xi_i)}_{A_\alpha''}).
\end{align*}
To prove equivariance, it thus suffices to show that 
\begin{equation}\label{equi diff}
\begin{split}
\braket{ \lambda^{\alpha''}_{h^{-1}}\xi_i&}{m_b(\lambda^{\alpha''}_{h^{-1}}\xi_i)}_{A_\alpha''} -  \braket{\xi_i}{m_b\xi_i}_{A_\alpha''}\\
&=\braket{\lambda^{\alpha''}_{h^{-1}}\xi_i-\xi_i}{m_b(\lambda^{\alpha''}_{h^{-1}}\xi_i)}_{A_\alpha''} +
\braket{\xi_i}{m_b(\lambda^{\alpha''}_{h^{-1}}\xi_i-\xi_i)}_{A_\alpha''}
\end{split}
\end{equation}
tends ultraweakly to zero. Indeed, in the identity 
\begin{align*}
&\braket{\lambda^{\alpha''}_{h^{-1}}\xi_i - \xi_i}{\lambda^{\alpha''}_{h^{-1}}\xi_i - \xi_i}_{A_\alpha''}\\
&\quad\quad =\alpha_{h^{-1}}''(\braket{ \xi_i}{\xi_i}_{A_\alpha''}) +\braket{\xi_i}{\xi_i}_{A_\alpha''}-\braket{\xi_i}{\lambda^{\alpha''}_{h^{-1}}\xi_i}_{A_\alpha''}-\braket{ \lambda^{\alpha''}_{h^{-1}}\xi_i}{\xi_i}_{A_\alpha''},
\end{align*}
the right hand side tends ultraweakly to zero.  The expression in line \eqref{equi diff} therefore tends ultraweakly to zero 
by using the Cauchy-Schwartz inequality for the semi-inner products $\braket{\cdot}{\cdot}_\phi:=\phi(\braket{\cdot}{\cdot}_{A_\alpha''})$
for all states $\phi\in S(A)^c$, i.e., for all normal states on $A_\alpha''$ by Proposition~\ref{prop-Ikunishi}.
\end{proof}

We now turn to the relationship of the (continuous) $G$-WEP to weak containment type properties.  Recalling that $A\rtimes_\inj G=A\rtimes_\red G$ when $G$ is exact by \cite{Buss:2018nm}*{Proposition 4.2}, the property  ``$A\rtimes_\max G=A\rtimes_\inj G$'' should be viewed as a generalization of the weak containment property to actions of potentially non-exact groups.

In \cite{Buss:2018nm}*{Proposition 3.12} we showed that the $G$-WEP implies that $A\rtimes_\max G=A\rtimes_\inj G$.
The following is a slight strengthening of this result:

\begin{proposition}\label{prop-GWEP}
Let $(A,\alpha)$ be a $G$-$C^*$-algebra. Then the following are equivalent:
\begin{enumerate}
\item $A$ has the continuous $G$-WEP;
\item For every nondegenerate covariant representation $(\pi,u)\colon (A,G)\to \Bd(H)$ and every $G$-embedding $\iota\colon A\into B$ into another $G$-$C^*$-algebra $B$, there is a  ccp $G$-map $\varphi\colon B\to \pi(A)''$ with $\varphi\circ\iota=\pi$. 
\end{enumerate}
In particular, if $A$ has the continuous $G$-WEP, then every covariant representation is $G$-injective and $A\rtimes_{\max}G=A\rtimes_\inj G$. 
\end{proposition}
\begin{proof} 
Fix a nondegenerate covariant representation $(\pi,u)\colon (A,G)\to \Bd(H)$ and a $G$-embedding $\iota \colon  A\into B$. Assuming that $A$ has the continuous $G$-WEP, there is a ccp $G$-map $\psi\colon B\to A_\alpha''$ with $\psi\circ\iota=i_A\colon A\into A_\alpha''$. On the other hand, by Proposition~\ref{prop-universal} there is a normal $G$-equivariant homomorphism $\pi''\colon A_\alpha''\to \pi(A)''\sbe \Bd(H)$ with $\pi''\circ i_A=\pi$. It follows that $\varphi:=\pi''\circ \psi$ is a ccp $G$-map with $\varphi\circ\iota=\pi$, as desired. Conversely, if every covariant representation satisfies this property, then so does the universal representation of $A\rtimes_\max G$ and this gives the continuous $G$-WEP for $A$.

The fact that if $A$ has the continuous $G$-WEP, then every covariant representation is injective is now true by definition 
(see Definition \ref{def-injective-rep}), and $A\rtimes_{\max}G=A\rtimes_\inj G$ follows from Proposition \ref{prop-inj-cov}.
\end{proof}

If the $C^*$-algebra being acted on is commutative, we can do better and get a complete characterization of the weak containment type property ``$A\rtimes_\max G=A\rtimes_\inj G$''.  Note that when $G$ is exact, the following result reduces to Theorem~\ref{thm-exact-commutative}; it should therefore be viewed as a generalization of that theorem that is applicable outside the realm of exact groups.

\begin{theorem}
Let $A=\contz(X)$ be a commutative $G$-$C^*$-algebra. Then $A$ has the continuous $G$-WEP if and only if $A\rtimes_\max G=A\rtimes_\inj G$.
\end{theorem}
\begin{proof}
If $A$ has the continuous $G$-WEP, then Proposition~\ref{prop-GWEP} implies that $A\rtimes_\max G=A\rtimes_\inj G$.  For the converse, let $(\pi, u)$ be a Haagerup standard form representation of $(A_\alpha'', G,\alpha'')$ on a Hilbert space $H$ as in Theorem \ref{thm-Haagerup}.  Since $A$ is commutative, so is $A_\alpha''$.  In particular $A_\alpha''$ is injective, whence \cite{Buss:2018nm}*{Proposition 2.2} implies that $B:=\contub(G,A_\alpha'')$ is a commutative $G$-injective $G$-$C^*$-algebra with respect to the action induced by the left translation $G$-action on $G$ and the trivial $G$-action on $A_\alpha''$. Consider the canonical $G$-embedding $\iota\colon A\into B$ that sends $a\in A$ to the function $\iota(a)(g):=i_A(\alpha_{g^{-1}}(a))$, where $i_A\colon A\into A_\alpha''$ is the canonical embedding. 
Since $(\pi\circ i_A, u)$ is a nondegenerate covariant representation of $(A,G,\alpha)$
and since  $A\rtimes_\max G=A\rtimes_\inj G$,  it follows from Corollary \ref{inj lem} that $(\pi\circ i_A, u)$ is $G$-injective. Hence
there exists a ccp $G$-map $\varphi\colon B\to \Bd(H)$ with $\varphi\circ\iota=\pi\circ i_A$.
Since $B$ is commutative, it follows that $$\varphi(B)\sbe \pi(i_A(A))'=\pi(A_\alpha'')'=\pi(A_\alpha'')\cong A_\alpha'',$$ where the second equality follows as $\pi(A_\alpha'')$ is a masa in $\Bd(H)$ by the properties of the Haagerup standard form. We may therefore view $\varphi$ as a ccp $G$-map $B\to A_\alpha''$ splitting the inclusion $A\into B$ in the sense that $\varphi\circ \iota=i_A$. Since $B$ is $G$-injective, this implies that $A$ has the continuous $G$-WEP.
\end{proof}

\chapter{Some  new developments and questions}\label{sec-questions}
In what follows we report on some results related to  this work and discuss some open questions.

\section{A summary of new developments}

We start with results related to this work.  Between our first posting of a draft of this paper on the arXiv and the current version, Bearden and Crann \cite{Bearden-Crann,Bearden-Crann1}, McKee and Pourshashami \cite{McKeePour}, and Ozawa and Suzuki \cite{Ozawa-Suzuki} have produced beautiful results which both complement and generalize some of ours.  These authors also answered some questions that we explicitly raised in earlier versions of this paper.  To clarify the state of the art and how these different results are related to each other, and also to record when others answered questions we raised in earlier versions of this paper, we discuss some of this work.

\subsection*{1} The relationship between amenability (Definition \ref{def-amenable (A)}) and von Neumann amenability (Definition \ref{def-vonNeumann}) was first addressed by Anantharaman-Delaroche (using different language).  Indeed in \cite{Anantharaman-Delaroche:1987os}*{Th\'{e}or\`{e}me 3.3}, Anantharaman-Delaroche shows that these notions are equivalent for actions of discrete groups.  In the current paper, we formulated these notions for locally compact groups, and proved they are equivalent whenever the acting group is exact (see Proposition \ref{prop:Amenability-conditions}).  In an earlier version, we asked whether they are equivalent in general.  This was rapidly solved affirmatively by Bearden and Crann in \cite{Bearden-Crann}*{Theorem 4.2}.  This result of Bearden and Crann was  influential on subsequent versions of this paper, as it allowed us to generalize several of our theorems: for example, we originally established the equivalence of amenability of an action $\alpha:G\to \Aut(C_0(X))$ and measurewise amenability of the underlying action $G\curvearrowright X$ in Theorem \ref{thm-amenable-all} under the assumption that $G$ is exact, and \cite{Bearden-Crann}*{Theorem 3.6} allowed us to establish it in general.

\subsection*{2} In an earlier version of this paper, we claimed a proof that amenability and the (QAP) are equivalent for discrete groups.  Our proof unfortunately had a mistake, but the equivalence was subsequently shown to hold in general by Ozawa and Suzuki \cite{Ozawa-Suzuki}*{Theorem 3.2}. 
In this theorem they also show that amenability for an action $\alpha:G\to\Aut(A)$ is equivalent to 
the existence of a $G$-conditional expectation
$$P:L^\infty(G,A^{**})\to A^{**},$$
although, as noted before, $A^{**}$ fails to be a $G$-von Neumann algebra in general. This can be viewed as 
 a direct extension of the original definition 
by  Anantharaman-Delaroche for discrete $G$.
 Ozawa and Suzuki show in \cite{Ozawa-Suzuki}*{Theorem 2.13} that an action has the (QAP) if and only if it has the approximation property (AP) of Exel and Ng \cite{ExelNg:ApproximationProperty}, and therefore the (AP) and amenability are equivalent in general (see Chapter \ref{chap:QAP} for more details).

\subsection*{3} In \cite{Anantharaman-Delaroche:1987os}*{Th\'{e}or\`{e}me 4.9}  Anantharaman-Delaroche shows that if $G$ is discrete, then amenability of $\alpha:G\to\Aut(C_0(X))$ is equivalent to strong amenability.  In an earlier version of this paper, motivated by our results on measurewise amenability in Theorem \ref{thm-amenable-all}, we asked whether this extends to locally compact groups: more precisely, we asked the equivalent question of whether amenability of $\alpha:G\to\Aut(C_0(X))$ and topological amenability of $G\curvearrowright X$ are the same.  This was solved by Bearden and Crann in \cite{Bearden-Crann}*{Corollary 4.14}. 

\subsection*{4} In this paper, we show in Proposition \ref{prop-restrict} that amenability behaves well under restrictions to exact subgroups.  This was established without the exactness assumption by Ozawa and Suzuki in \cite{Ozawa-Suzuki}*{Corollary 3.4}.
\subsection*{5} In an earlier version of this paper, we asked whether a locally compact group is exact if and only if it admits an amenable action on a unital $C^*$-algebra.  This was answered in the affirmative by Ozawa and Suzuki in \cite{Ozawa-Suzuki}*{Corollary 3.6}.
\subsection*{6} In an earlier version of this paper we asked whether the $G$-WEP of \cite{Buss:2018nm}*{Definition 3.9} and the continuous $G$-WEP of Definition \ref{def-weakGWEP} above are the same.  This was solved affirmatively by Bearden and Crann in \cite{Bearden-Crann1}*{Proposition 4.5}, as part of a general study of the equivariant weak expectation property and related issues.
\subsection*{7} In Question \ref{question6} below we ask for which class $\mathcal{C}$ of locally compact groups does the following hold: ``For any $G$-$C^*$-algebra $(A,\alpha)$, if $A\rtimes_\red G$ is nuclear, then $\alpha$ is amenable.''  In a previous version of this paper we suggested that the class $\mathcal C$ might contain all groups with property (W) as studied by Anantharaman-Delaroche in 
\cite{Anantharaman-Delaroche:2002ij}*{Section 4}. In \cite{Crann}*{Theorem 3.5}, Crann shows that property (W) is equivalent to inner amenability.  Very recently, McKee and Pourshahami proved (see \cite{McKeePour}*{Corollary 6.6}) that all  inner amenable groups (and hence all groups with property (W)) are indeed contained in $\mathcal C$. 

\medskip

As a result of this progress, we are now in the very satisfactory situation of knowing that many versions of amenability are equivalent in complete generality: these include von Neumann amenability, amenability, the (wQAP), the (AP), and the (QAP).  On the other hand, it is now clear that other versions of amenability are different: work of Suzuki \cite{Suzuki:2018qo} (see also the discussion in
 \cite{Buss:2019}*{Section 3}) shows that strong amenability is strictly stronger than amenability; and commutant amenability is strictly weaker than amenability thanks to the results of Section \ref{sec:example}.

To conclude the discussion on recent developments, we would like to draw the reader's attention to another aspect of the recent work of Ozawa and Suzuki in \cite{Ozawa-Suzuki}, and in particular to Section 6 of that paper. Ozawa and Suzuki give  several exciting constructions of amenable actions of locally compact groups on purely infinite simple $C^*$-algebras. In particular, they show (among other interesting examples), that every amenable action of a second countable group $G$ on a separable $C^*$-algebra $A$ is $KK^G$-equivalent to an outer amenable action of $G$ on a separable simple purely infinite $C^*$-algebra $B$ so that the $KK^G$-equivalence can be  realized by a $G$-equivariant inclusion $A\subseteq B$ (see \cite{Ozawa-Suzuki}*{Theorem 6.1}).  
\section{Some questions}

We now turn to questions.  

Let us first discuss the connection between amenability and weak containment.  In Section \ref{sec:example}, we showed that there are non-amenable actions on the compact operators whose maximal and reduced crossed products are the same, i.e.\ that have the weak containment property.  However, the following basic question remains open.

\begin{question}\label{a vs wc unital}
Is there a non-amenable action on a unital $C^*$-algebra with the weak containment property?
\end{question}

We also observed that the method used in Section \ref{sec:example} for producing non-amenable examples with weak containment is unlikely to work for discrete groups.  The following question is thus very natural and interesting.

\begin{question}\label{a vs wc discrete}
Is there a non-amenable action of a discrete group with the weak containment property?    
\end{question}

Due to the close relationship between the weak containment property and commutant amenability, the following questions are closely related.

\begin{question}\label{a vs ca unital}
Is there a non-amenable, commutant amenable action on a unital $C^*$-algebra?
\end{question}

\begin{question}\label{a vs ca discrete}
Is there a non-amenable, commutant amenable action of a discrete group?    
\end{question}

Note that Questions \ref{a vs wc unital} and \ref{a vs wc discrete} are equivalent to Questions \ref{a vs ca unital} and \ref{a vs ca discrete} respectively for actions of \emph{exact} groups, as Theorem \ref{prop-general} shows the weak containment property and commutant amenability are equivalent for actions of exact groups.  However, for non-exact groups, this is not at all clear: note for example that Questions \ref{a vs ca unital} and \ref{a vs ca discrete} have a negative answer for actions on \emph{commutative} $G$-$C^*$-algebras by Theorem \ref{thm-exact-commutative}; however, Questions \ref{a vs wc unital} and \ref{a vs wc discrete} seem particularly interesting for commutative $G$-$C^*$-algebras.

The above comments also make the following question quite natural.

\begin{question}\label{wc vs ca}
Are weak containment and commutant amenability equivalent for all actions of locally compact groups?
\end{question}

A problem with commutant amenability is that it seems difficult to check: using Proposition \ref{ca uni rep}, it suffices to check commutant amenability  for the universal representation, but this is generally difficult due to the huge size of this object.  As in many interesting (and possibly all) cases, commutant amenability is equivalent to the weak containment property, it would be very satisfying to have an answer to the following question.

\begin{question}\label{intrinsic ca}
Is there an `intrinsic' characterization of commutant amenability for a $G$-$C^*$-algebra $A$, i.e.\ can one find an approximation property of the pair $(A,G)$ that is equivalent to commutant amenability?
\end{question}

Continuing a discussion of the difference between amenability and commutant amenability, we know from Theorem \ref{prop:Nuclear-Pairs} (and Proposition \ref{prop-general}) that if $(A,\alpha)$ is amenable and $A$ is nuclear, then $A\rtimes_\red G$ is nuclear.  The following is thus natural.

\begin{question}\label{question5}
Suppose that $(A,\alpha)$ is a $G$-$C^*$-algebra with $A$ nuclear.   Does commutant amenability of $(A,\alpha)$ imply nuclearity of $A\rtimes_\red G$?
\end{question}

In \cite{Buss:2019}*{Theorem 6.1, part (i)}, the authors show that if $A$ is nuclear and commutant amenable, and if $G$ is discrete, then the canonical inclusion $A\rtimes_\red G \to (A\otimes A^\op)\rtimes_\red G$ is a nuclear map.  This is  evidence for a positive answer to Question \ref{question5}, at least in the discrete case.

Continuing this circle of ideas, Anantharaman-Delaroche showed in \cite{Anantharaman-Delaroche:1987os}*{Th\'{e}or\`{e}me 4.5} that if the reduced crossed product $A\rtimes_\red G$ by an action of a {\em discrete} group $G$ is nuclear then the action of $G$ on $A$ must be amenable. 
A similar result cannot hold for all locally compact groups since one can find counterexamples even for $A=\C$ by taking $G$ to be any non-amenable group with nuclear reduced $C^*$-algebra, such as $G=\SL(2,\R)$.

\begin{question}\label{question6}
For what class $\mathcal{C}$ of locally compact groups $G$ does the following hold?  ``For any $G$-$C^*$-algebra 
$(A,\alpha)$, if $A\rtimes_\red G$ is nuclear, then $(A,\alpha)$ is amenable.''
\end{question}

As mentioned above, McKee and Pourshashami \cite{McKeePour}*{Corollary 6.6} showed that $\mathcal{C}$ contains all inner amenable groups (equivalently, all groups with property (W)), generalizing Anantharaman-Delaroche's result for discrete groups.  However, the exact extent of the class $\mathcal{C}$ remains open. Recall from Remark \ref{rem-coaction} that nuclearity of $A\rtimes_\red G$ always implies nuclearity of $A$. Hence one may as well restrict to nuclear $A$ in Question \ref{question6}.

As a last question, unrelated to the discussion above, let us recall that Bearden and Crann \cite[Corollary 4.14]{Bearden-Crann} showed that an action on a commutative $G$-$C^*$-algebra is amenable if and only if it is strongly amenable.  On the other hand, in Corollary \ref{cor-strong-typeI} we showed that an action of a second countable group on a separable type I $C^*$-algebra with Hausdorff spectrum is amenable if and only if it is strongly amenable.  It is natural to ask whether the Hausdorffness and separability assumptions can be dropped.

\begin{question}\label{sa and a}
Let $(A,G,\alpha)$ be a type I, amenable $G$-$C^*$-algebra.  Is $(A,G,\alpha)$ also strongly amenable?
\end{question}

Note that if one replaces ``type I'' with ``nuclear'', then the answer is ``no'' as shown by Suzuki in \cite{Suzuki:2018qo} (see also the discussion in \cite{Buss:2019}*{Section 3}).


\bibliographystyle{amsalpha}

\printindex

\end{document}